\documentclass[a4paper, 11pt]{article}

\usepackage[margin=2.5cm]{geometry}
\usepackage[utf8]{inputenc}
\usepackage{graphicx}
\usepackage{amssymb}
\usepackage{amsmath}
\usepackage{mathtools}
\usepackage[usenames,dvipsnames,svgnames,table]{xcolor}
\usepackage{enumerate}
\usepackage{enumitem}
\setlist[enumerate, 1]{label=(\roman*)}
\setlist[itemize]{leftmargin=1.5em}
\setlist[description]{leftmargin=1em}
\setlist[itemize, 1]{label=$\blacktriangleright$}
\setlist[itemize, 2]{label=$\bullet$}
\usepackage[normalem]{ulem} 

\usepackage{ifthen}

\usepackage{longtable}
\usepackage{multirow}
\usepackage{array}
\usepackage{capt-of}
\usepackage{arydshln}

\usepackage{pgf,tikz}
\usepackage{mathrsfs}
\usetikzlibrary{arrows}
\usepackage[outline]{contour}
\contourlength{0.6pt}

\usepackage{graphicx}
\usepackage{float}
\graphicspath{ {pics/} }

\usepackage[hidelinks]{hyperref}  
\usepackage[format=plain,labelsep=period,justification=justified,font=small,labelfont=bf, margin=1em]{caption}

\usepackage{amsthm}

\theoremstyle{plain}
\newtheorem{theorem}{Theorem}[section]
\newtheorem{corollary}[theorem]{Corollary}
\newtheorem{lemma}[theorem]{Lemma}

\theoremstyle{definition}
\newtheorem{definition}{Definition}[section]

\theoremstyle{remark}
\newtheorem{remark}{Remark}[section]

\newcommand{\Z}{\mathbb{Z}}
\newcommand{\N}{\mathbb{N}}
\newcommand{\R}{\mathbb{R}}

\renewcommand{\L}{\mathbb{L}}

\newcommand\varpm{\mathbin{\vcenter{\hbox{%
  \oalign{\hfil$\scriptstyle+$\hfil\cr
          \noalign{\kern-.3ex}
          $\scriptscriptstyle({-})$\cr}%
}}}}

\renewcommand{\eqref}[1]{(\refeq{#1})}

\mathtoolsset{showonlyrefs,showmanualtags}

\newcommand{\bela}[1]{\begin{equation}\label{#1}}
\newcommand{\ela}{\end{equation}}
\newcommand{\bear}[1]{\begin{array}{#1}}
\newcommand{\ear}{\end{array}}
\newcommand{\ba}{\mbox{\boldmath $a$}}
\newcommand{\bb}{\mbox{\boldmath $b$}}
\newcommand{\bc}{\mbox{\boldmath $c$}}
\newcommand{\bd}{\mbox{\boldmath $d$}}
\newcommand{\bx}{\mbox{\boldmath $x$}}

\newcommand{\bv}{\mbox{\boldmath $v$}}
\newcommand{\be}{\mbox{\boldmath $e$}}
\newcommand{\bl}{\mbox{\boldmath $l$}}
\newcommand{\bef}{\mbox{\boldmath $f$}}

\definecolor{grey}{rgb}{0.5,0.5,0.5}

\newcommand{\ignore}[1]{}
\newcommand{\jac}[1]{\operatorname{#1}}
\newcommand{\calH}{\mathcal{H}}
\newcommand{\sfK}{\mathsf{K}}
\newcommand{\blank}[1]{}

\title{Checkerboard incircular nets.\\ Laguerre geometry and parametrisation}
\author{Alexander I. Bobenko$^1$, Wolfgang K. Schief$^2$, Jan Techter$^1$ \bigskip\\  
$^1$Institut f\"ur Mathematik, TU Berlin, \\ Str.\@ des 17.\@ Juni 136, 10623 Berlin, Germany\bigskip\\
$^2$School of Mathematics and Statistics,\\ The University of New South Wales, Sydney, NSW 2052, Australia}
\date{\today}

\begin{document}

\maketitle

\begin{abstract}
We present a procedure which allows one to integrate explicitly the class of checkerboard IC-nets which has recently been introduced as a generalisation of incircular (IC) nets. The latter class of privileged congruences of lines in the plane is known to admit a great variety of geometric properties which are also present in the case of checkerboard IC-nets. The parametrisation obtained in this manner is reminiscent of that associated with elliptic billiards. Connections with discrete confocal coordinate systems and the fundamental QRT maps of integrable systems theory are made. The formalism developed in this paper is based on the existence of underlying pencils of conics and quadrics which is exploited in a Laguerre geometric setting. 
\end{abstract}



%
\begin{figure}[h]
  \centering
  \raisebox{-0.5\height}{\includegraphics[width=0.49\textwidth]{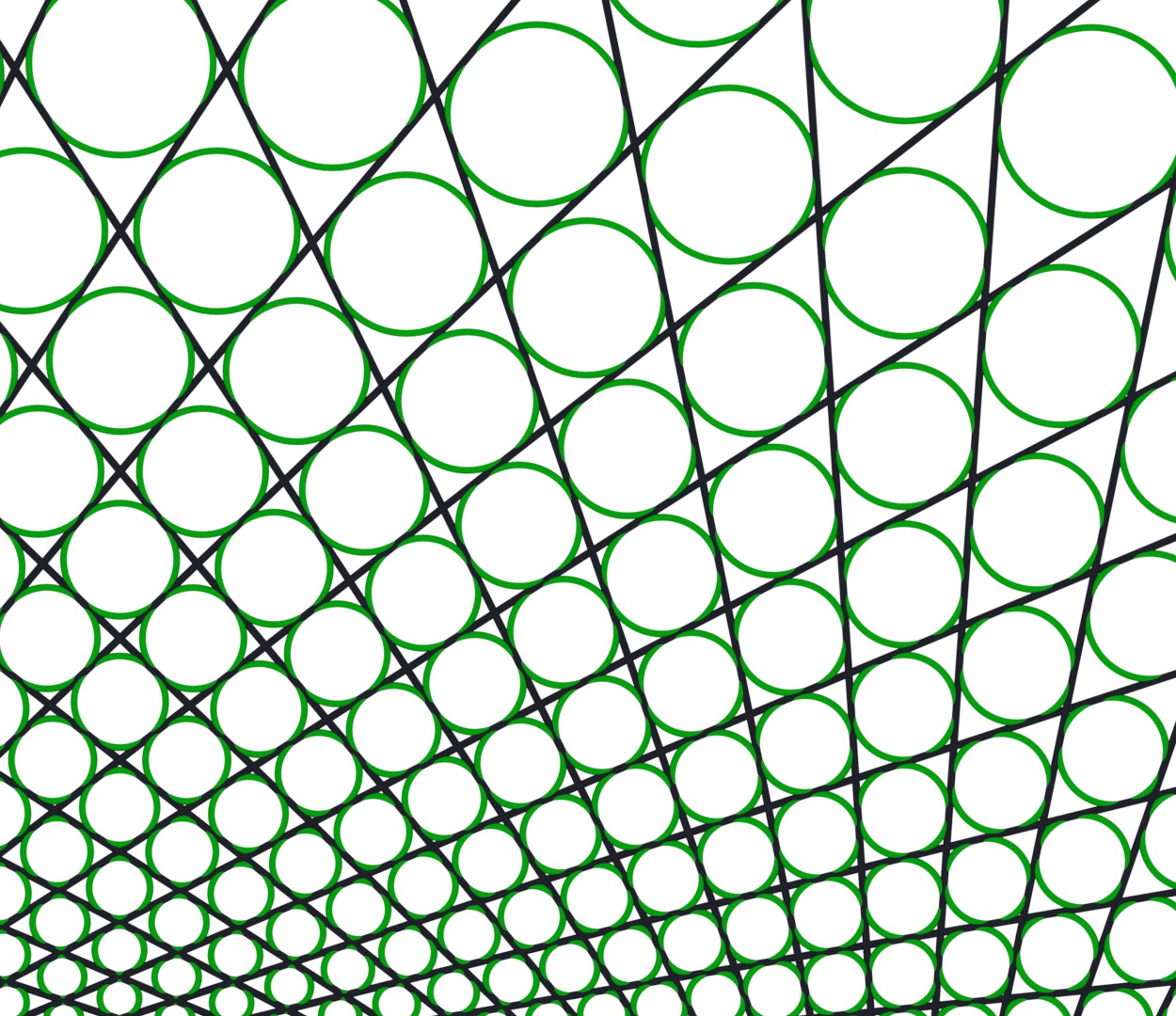}}
  \raisebox{-0.5\height}{\includegraphics[width=0.49\textwidth]{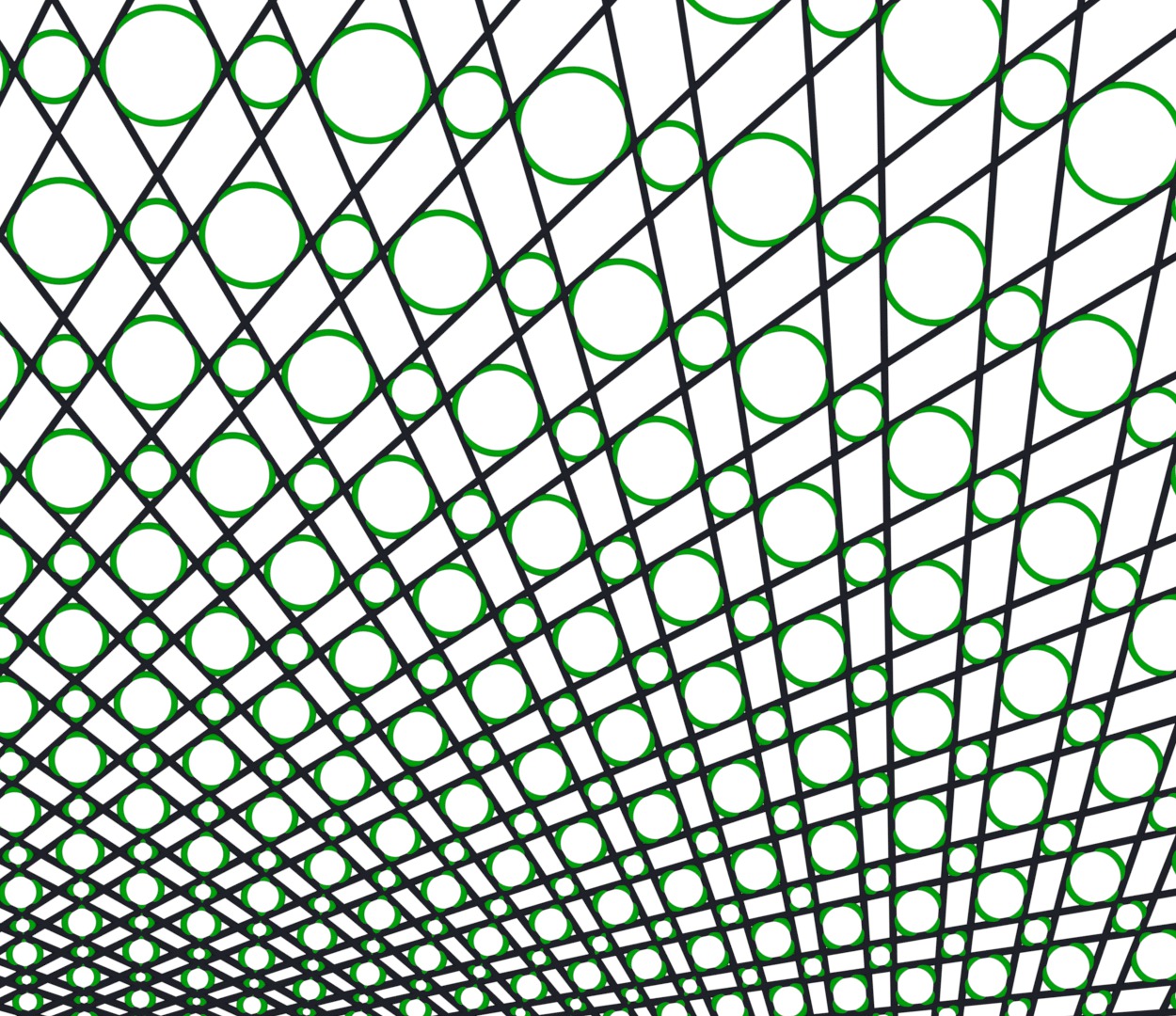}}
  \caption{
    \emph{Left:} An example of an IC-net.
    \emph{Right:} An example of a checkerboard IC-net.
  }
  \label{fig:IC-net}
\end{figure}  

\newpage

\section{Introduction}

The construction and geometry of {\em incircular nets} (IC-nets) and their generalisation to {\em checkerboard IC-nets} have recently been discussed in great detail in \cite{AB}. IC-nets were introduced by B\"ohm \cite{B} and are defined as congruences of straight lines in the plane with the combinatorics of the square grid such that each elementary quadrilateral admits an incircle as depicted in Figure~\ref{fig:IC-net} (left).
IC-nets have a wealth of geometric properties, including the distinctive feature that any IC-net comes with a conic to which its lines are tangent. Another important aspect is that IC-nets discretise confocal quadrics. In fact, it has been observed in \cite{BSST17} that IC-nets constitute particular instances of discrete confocal coordinate systems in the plane, which provides a first indication that IC-nets should be examined in the context of {\em integrable} discrete differential geometry \cite{BS}. In this connection, it is noted that an integrable systems approach to the discretisation of confocal quadrics has been taken in \cite{BSST16}. IC-nets are closely related to Poncelet(-Darboux) grids originally introduced by Darboux \cite{D8789} and further studied in \cite{LT} and \cite{Sch}.

Due to the combinatorial structure of IC-nets, their lines and circles may not be consistently oriented in such manner that these are in oriented contact. However, IC-nets are intimately related to checkerboard IC-nets which do exhibit this feature. Once again, the lines of checkerboard IC-nets have the combinatorics of the square grid but it is only required that every second quadrilateral admits an incircle, namely the ``black'' (or ``white'') quadrilaterals if the quadrilaterals of the net are combinatorially coloured like those of a checkerboard. An example of a checkerboard IC-net is displayed in Figure \ref{fig:IC-net} (right). 

{\em Confocal} checkerboard IC-nets constitute an important subclass of checkerboard IC nets and are characterised by their lines being tangent to a conic as in the case of IC-nets. This terminology is due to the remarkable fact that the points of intersection of the lines of a confocal checkerboard IC-net lie on conics which are confocal to the underlying conic.

In general, checkerboard IC-nets may be constructed in the following manner. One starts with a circle $\omega_{1,1}$ and four tangents $\ell_1$, $\ell_2$, $m_1$ and $m_2$. Subsequently, as indicated in Figure~\ref{fig:construction of checkerboard IC-net}, one chooses four circles $\omega_{0,0}$, $\omega_{0,2}$, $\omega_{2,2}$ and $\omega_{2,0}$ which touch the pairs of lines forming the ``corners'' of the configuration of given lines.
The four lines $\ell_0$, $\ell_3$, $m_0$, and $m_3$ being tangent to the respective pairs of circles are then fixed so that, in turn, the circles $\omega_{1,3}$ and $\omega_{3,1}$ are uniquely determined. An additional degree of freedom is obtained by choosing the circle $\omega_{3,3}$ Now, the entire checkerboard net is predetermined. Indeed, we first construct the lines $\ell_4$ and $m_4$, then the circles $\omega_{4,0}$, $\omega_{4,2}$, $\omega_{0,4}$, $\omega_{2,4}$ and, finally, the lines $\ell_5$ and $m_5$. The existence of the circle $\omega_{4,4}$ is non-trivial and follows from an incidence theorem \cite{AB}. 
In Section \ref{s.laguerre}, we present a simpler proof of this theorem, using the formalism developed in this paper. 
Iterative application of this theorem now generates a checkerboard IC-net of arbitrary size. In summary, a checkerboard IC-net is uniquely determined by five neighbouring circles $\omega_{0,0}$, $\omega_{2,0}$, $\omega_{0,2}$, $\omega_{2,2}$, $\omega_{1,1}$ and the circle $\omega_{3,3}$. Thus, up to Euclidean motions and homotheties, checkerboard IC-nets form a real eight-dimensional family of nets.

The main objective of this paper is the explicit integration of (generic) checkerboard IC-nets in terms of Jacobi elliptic functions \cite{NIST} similar to that of elliptic billiards \cite{DR}. As a result, we establish explicit connections with, for instance, the discrete confocal coordinate systems mentioned above and the celebrated (symmetric) QRT mappings \cite{QRT89} which play a fundamental role in the theory of discrete integrable systems (see, e.g., \cite{IR01} and references therein). We also prove constructively the existence and provide examples of confocal checkerboard IC-nets which are closed (embedded) in the ``azimuthal'' direction. In order to achieve these results, we adopt a Laguerre geometric point of view  which is natural due to the above-mentioned orientability of the lines and circles of checkerboard IC-nets. The necessary theoretical background of Laguerre geometry \cite{BT, BS} is provided in the Appendix. Thus, we first  determine the class of checkerboard IC-nets which may be mapped to confocal checkerboard IC-nets by means of real Laguerre transformations. Here, it is noted that, in the complex setting, all checkerboard IC-nets are Laguerre-equivalent to confocal checkerboard IC-nets. The classification of checkerboard IC-nets in the real setting is based on the standard classification of pencils of conics \cite{levy1964} which emerges due to the important observation that any checkerboard IC-net admits an underlying pencil of quadrics. 
It should be noted that Laguerre geometry is indispensable in the investigation of IC-nets. Recently, this classical (but lesser-known) geometry has been applied to solve problems not only in geometry \cite{SPG} but also in free-form architecture \cite{PGB}.
\begin{figure}
\begin{center}
	\includegraphics{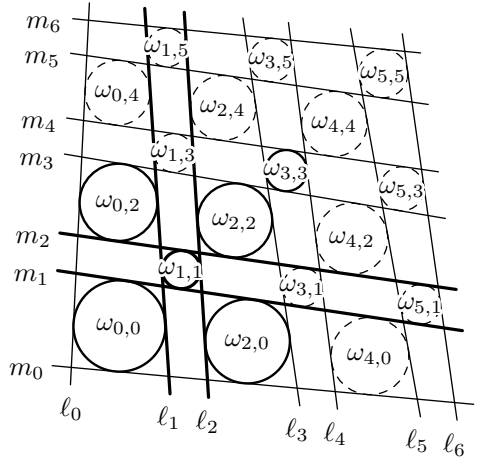}\\
	\caption{An elementary construction of a checkerboard IC-net}
	\label{fig:construction of checkerboard IC-net}
\end{center}
\end{figure}

The second step in the procedure is to parametrise confocal checkerboard IC-nets. This is done via the base curve which is shared by the pencil of quadrics associated with any given checkerboard IC-net. These base curves are known to be just another avatar of so-called hypercycles \cite{Blaschke} which constitute particular curves in the plane of degree 8. In fact, it is demonstrated that the lines of a checkerboard IC-net are tangent to a hypercycle. In the case of a confocal checkerboard IC-net, the hypercycle degenerates to the union of two identical conics with different orientations. The above-mentioned (hypercycle) base curves are of degree 4 and may be parametrised in terms of Jacobi elliptic functions. This will then lead to an explicit parametrisation of confocal checkerboard IC-nets and their Laguerre transforms. Furthermore, this parametrisation also applies to the generalised checkerboard IC-nets introduced at the end of the paper. Their geometric construction is very natural within the Laguerre-geometric framework established here and gives rise to a connection with ``non-autonomous'' QRT maps (see, e.g., \cite{RJ} and references therein).


\section{Checkerboard IC-nets. Definition and elementary properties}
\label{s.elementary}

Checkerboard IC-nets have the combinatorics of a checkerboard, where all ``black'' quadrilaterals have inscribed circles (see Fig.~\ref{fig:construction of checkerboard IC-net}). These were introduced in \cite{AB}. 
\begin{definition}
	A checkerboard IC-net is comprised of oriented lines $\ell_i, m_j$ in the plane with $i,j\in \mathbb{Z}$ such that for any  $k$ and $n$ the lines $\ell_{2k}, \ell_{2k+1}$, $m_{2n}, m_{2n+1}$ as well as the lines $\ell_{2k-1}, \ell_{2k}$, $m_{2n-1}, m_{2n}$ have a circle in oriented contact. The points of intersection $\ell_i \cap m_j$ are vertices of the corresponding quadrilateral lattice ${\mathbb Z}^2\to {\mathbb R}^2$.
\end{definition}	

\begin{remark}\label{rem.IC}
  We interpret the lines $\ell_i$ as combinatorially vertical and $m_j$ as combinatorially horizontal. Checkerboard IC-nets become IC-nets when every second combinatorially horizontal strip and every second vertical strip degenerates in the sense that the two lines of such a strip coincide up to their orientation. Then, all remaining quadrilaterals admit inscribed circles, and all lines are non-oriented. An important property of IC-nets is that all their lines are tangent to a conic \cite{AB}. The proof of this fact is based on the Graves-Chasles theorem \cite[\textsection 174]{Darboux} (see also \cite{IT}).
\end{remark}
\begin{figure}
  \centering
  \includegraphics[scale=1]{pics/fig-simpquad-6}
  \caption{Graves--Chasles theorem}
  \label{fig:graves-chasles}
\end{figure} 

\begin{theorem}[Graves--Chasles theorem]
        \label{lem:circumscribed quadrilateral}
        Suppose that all sides of a complete quadrilateral touch a conic $\alpha$.
		Denote pairs of its opposite vertices by $\{\ba, \bc\}$, $\{\bb, \bd\}$, and $\{\be, \bef\}$ (see Figure \ref{fig:graves-chasles}).
        Then, the following four properties are equivalent:
        \begin{enumerate}
                \item \label{lem:circumscribed quadrilateral:circumscirbed} $(\ba\bb\bc\bd)$ is circumscribed,
                \item \label{lem:circumscribed quadrilateral:ac} Points $\ba$ and $\bc$ lie on a conic confocal with $\alpha$,
                \item \label{lem:circumscribed quadrilateral:bd} Points $\bb$ and $\bd$ lie on a conic confocal with $\alpha$,
		\item \label{lem:circumscribed quadrilateral:ef} Points $\be$ and $\bef$ lie on a conic confocal with $\alpha$.
        \end{enumerate}
\end{theorem}

IC-nets are intimately related to a subclass of checkerboard IC-nets, namely confocal checkerboard IC-nets.
\begin{definition}\cite{AB}.
\label{def:checkerboard confocal IC-net}
A checkerboard IC-net is called {\em confocal} if all lines of it are tangent to a conic. 
\end{definition}
In fact, this class of checkerboard IC-nets constitutes a natural generalisation of IC-nets. However, in contrast to IC-nets, here, all circles and lines can be oriented so that the corresponding circles and lines are in oriented contact. 
Moreover, confocal checkerboard IC-nets can be regarded as subdivisions of IC-nets. This follows from the following lemma.

\begin{lemma}
\label{lem:subdivision}
Let the six lines $\ell_0,\ell_1,\ell_2,m_0,m_1,m_2$ in Figure \ref{fig:subdivision} (left) touch a conic. Then, any two of the incidences 
\begin{enumerate}
\item $\ell_0, \ell_1, m_0, m_1$ are tangent to a circle,
\item $\ell_1, \ell_2, m_1, m_2$ are tangent to a circle,
\item $\ell_0, \ell_2, m_0, m_2$ are tangent to a circle,
\end{enumerate}
imply the third one.
\end{lemma}
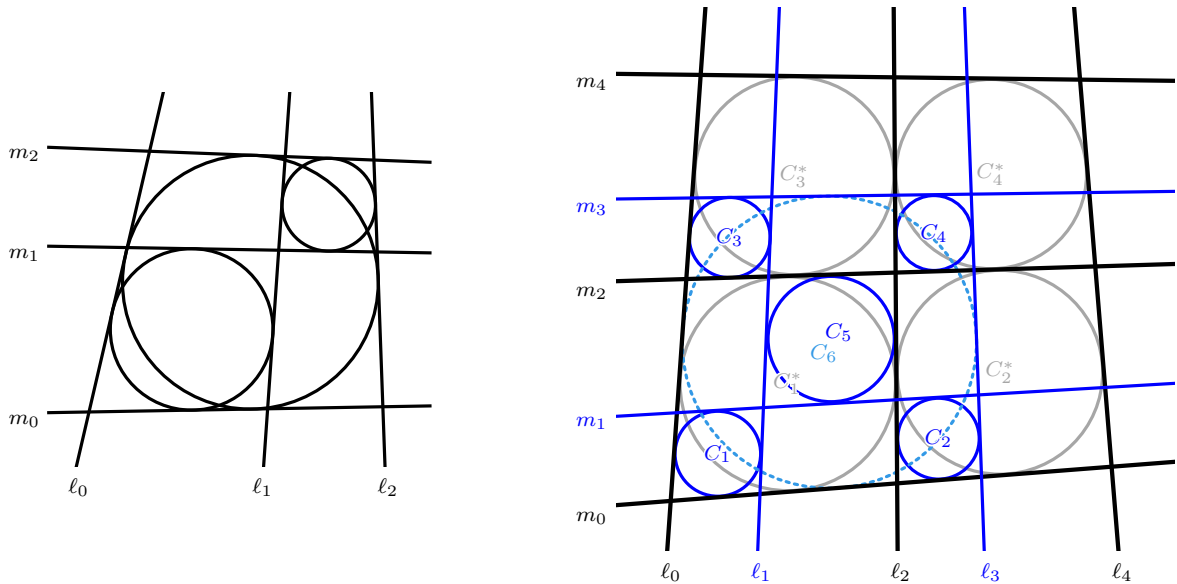
\begin{figure}
  \begin{center}
    \definecolor{xdxdff}{rgb}{0.49019607843137253,0.49019607843137253,1.}
    \definecolor{qqqqff}{rgb}{0.,0.,1.}
    \definecolor{brightblue}{rgb}{0.2,0.6,0.9}
    \definecolor{cqcqcq}{rgb}{0.65,0.65,0.65}
    \raisebox{-0.5\height}{
    \begin{tikzpicture}[line cap=line join=round,>=triangle 45,x=1.0cm,y=1.0cm, scale=1.1]
      \begin{scriptsize}
        \draw (0.75, 1.25) node {$m_2$};
        \draw (0.75, 0.05) node {$m_1$};
        \draw (0.75,-1.95) node {$m_0$};
        \draw (1.4, -2.75) node {$\ell_0$};
        \draw (3.6, -2.75) node {$\ell_1$};
        \draw (5.1, -2.75) node {$\ell_2$};
      \end{scriptsize}
      \clip(0.5,-2.5) rectangle (6,2);
      \draw [line width=1.2pt,domain=1.0218507454290813:5.620227103827782] plot(\x,{(--2.1171234317981646-0.06032777739217865*\x)/1.528303693935193});
      \draw [line width=1.2pt,domain=1.0218507454290813:5.620227103827782] plot(\x,{(--2.864455176505432-0.027308037642174376*\x)/-1.5292501079617975});
      \draw [line width=1.2pt,domain=1.0218507454290813:5.620227103827782] plot(\x,{(--1.6524575856362715-0.17074627539565726*\x)/9.703730177564767});
      \draw [line width=1.2pt,domain=1:6] plot(\x,{(2.9037862282022457-1.4903785925758262*\x)/-0.34369066384284586});
      \draw [line width=1.2pt,domain=1:6] plot(\x,{(-7.601848459696713+1.5285068728232307*\x)/-0.05493961551725765});
      \draw [line width=1.2pt,domain=1:6] plot(\x,{(2.2232511465521205-0.5874505735209161*\x)/-0.04074279205153397});
      \draw [line width=1.2pt] (2.750421457119531,-0.8510485519714623) circle (0.9727926759480807cm);
      \draw [line width=1.2pt] (4.3903677283120475,0.652331019971891) circle (0.5592061475810507cm);
      \draw [line width=1.2pt] (3.453026011141038,-0.28171205504568275) circle (1.5294939102922367cm);
    \end{tikzpicture}}
  \hspace{1cm}
  \raisebox{-0.5\height}{
    \begin{tikzpicture}[line cap=round,line join=round,>=triangle 45,x=0.6cm,y=0.6cm, scale=1.2]
      \begin{scriptsize}
        \draw                (-22.00, -10.6) node {$\ell_0$};
        \draw [color=qqqqff] (-20.35, -10.6) node {$\ell_1$};
        \draw                (-17.80, -10.6) node {$\ell_2$};
        \draw [color=qqqqff] (-16.15, -10.6) node {$\ell_3$};
        \draw                (-13.75, -10.6) node {$\ell_4$};
        \draw                (-23.45, -9.60) node {$m_0$};
        \draw [color=qqqqff] (-23.45, -7.85) node {$m_1$};
        \draw                (-23.45, -5.45) node {$m_2$};
        \draw [color=qqqqff] (-23.45, -3.90) node {$m_3$};
        \draw                (-23.45, -1.60) node {$m_4$};
      \end{scriptsize}
      \begin{scriptsize}

      \end{scriptsize}
      \clip(-23,-10.2) rectangle (-12.8,-0.2);
      \draw [line width=1.2pt,color=cqcqcq] (-16.15014408064675,-3.2663866065196037) circle (1.0454563838831328cm);
      \draw [line width=1.2pt,color=cqcqcq] (-19.746632129029095,-3.29962797404792) circle (1.0925239900503154cm);
      \draw [line width=1.2pt,color=cqcqcq] (-15.98960086073837,-6.91571556660574) circle (1.1260761627943001cm);
      \draw [line width=1.2pt,color=cqcqcq] (-19.866301365803707,-7.127387968034011) circle (1.1807951798397909cm);
      \draw [line width=1.2pt,color=qqqqff] (-20.92746803373379,-4.435714398543925) circle (0.4431998589424239cm);
      \draw [line width=1.2pt,color=qqqqff] (-17.113168569460893,-8.139347353097717) circle (0.4466873824444368cm);
      \draw [line width=1.2pt,color=qqqqff] (-21.14509390667169,-8.416916460456557) circle (0.4692494815403048cm);
      \draw [line width=1.2pt,color=qqqqff] (-19.07437586483562,-6.30459571172441) circle (0.6921114863355595cm);
      \draw [line width=1.2pt,color=qqqqff] (-17.195085539324268,-4.352056583614257) circle (0.41383581310473956cm);
      \draw [line width=1.2pt,dotted,color=brightblue] (-19.10661963212508,-6.361328416195585) circle (1.6134972619999448cm);
      \begin{scriptsize}
        \draw [color=cqcqcq] (-19.866301365803707,-7.127387968034011) node {\textcolor{cqcqcq}{\contour{white}{$C_1^*$}}};
        \draw [color=cqcqcq] (-15.98960086073837,-6.91571556660574) node {\textcolor{cqcqcq}{\contour{white}{$C_2^*$}}};
        \draw [color=cqcqcq] (-19.746632129029095,-3.29962797404792) node {\textcolor{cqcqcq}{\contour{white}{$C_3^*$}}};
        \draw [color=cqcqcq] (-16.15014408064675,-3.2663866065196037)  node {\textcolor{cqcqcq}{\contour{white}{$C_4^*$}}};
        \draw [color=qqqqff] (-21.14509390667169,-8.416916460456557)  node {\textcolor{qqqqff}{\contour{white}{$C_1$}}};
        \draw [color=qqqqff] (-17.113168569460893,-8.139347353097717) node {\textcolor{qqqqff}{\contour{white}{$C_2$}}};
        \draw [color=qqqqff] (-20.92746803373379,-4.435714398543925) node {\textcolor{qqqqff}{\contour{white}{$C_3$}}};
        \draw [color=qqqqff] (-17.195085539324268,-4.352056583614257) node {\textcolor{qqqqff}{\contour{white}{$C_4$}}};
        \draw [color=qqqqff] (-19.20661963212508,-6.561328416195585) node {\textcolor{brightblue}{$C_6$}};
        \draw [color=qqqqff] (-18.93437586483562,-6.17459571172441) node {\textcolor{qqqqff}{$C_5$}};
      \end{scriptsize}
      \draw [line width=1.6pt,domain=-23.533953826950885:-12.630265709267228] plot(\x,{(-1743.7983143712775-81.64478422243059*\x)/-5.698125565523831});
      \draw [line width=1.6pt,domain=-23.533953826950885:-12.630265709267228] plot(\x,{(--637.0156581045989-6.585032873989613*\x)/-84.21002803635429});
      \draw [line width=1.6pt,domain=-23.533953826950885:-12.630265709267228] plot(\x,{(-751.1287779319605-51.27279689621881*\x)/4.149629125596666});
      \draw [line width=1.6pt,domain=-23.533953826950885:-12.630265709267228] plot(\x,{(--144.03427938237343--1.0587402727070867*\x)/-84.22198468034219});
      \draw [line width=1.6pt,domain=-23.533953826950885:-12.630265709267228] plot(\x,{(--380.91912398166517-2.6153085136644116*\x)/-84.18802643986376});
      \draw [line width=1.6pt,domain=-23.533953826950885:-12.630265709267228] plot(\x,{(-922.4428255924315-51.43911950538876*\x)/0.368928111206678});
      \draw [line width=1.2pt,color=qqqqff,domain=-23.533953826950885:-12.630265709267228] plot(\x,{(--769.40837405141-7.199625587328178*\x)/-121.00233896295484});
      \draw [line width=1.2pt,color=qqqqff,domain=-23.533953826950885:-12.630265709267228] plot(\x,{(-1059.8413561113573-52.95831438594945*\x)/-2.0950285094678662});
      \draw [line width=1.2pt,color=qqqqff,domain=-23.533953826950885:-12.630265709267228] plot(\x,{(--291.2416014610257-1.1761445149498884*\x)/-85.43608595309678});
      \draw [line width=1.2pt,color=qqqqff,domain=-23.533953826950885:-12.630265709267228] plot(\x,{(-875.9054675146829-52.62171522526691*\x)/1.8994222472520157});
    \end{tikzpicture}} 
  \end{center}
  \caption{Confocal checkerboard IC-nets as subdivisions of IC-nets. Two incidence theorems.}
  \label{fig:subdivision}
\end{figure}
\begin{proof}
Let the lines $\ell_0,\ell_1,\ell_2,m_0,m_1,m_2$ be tangent to a conic $\alpha$. Consider the three points of intersection $\ell_0\cap m_0$, $\ell_1\cap m_1$, $\ell_2\cap m_2$. If, for instance, the two circles in (i) and (ii) exist then the Graves--Chasles theorem implies that the pairs of points
$(\ell_0\cap m_0,\ell_1\cap m_1)$ and $(\ell_1\cap m_1,\ell_2\cap m_2)$ lie on conics confocal to $\alpha$ and, hence, on a common conic confocal to $\alpha$. Application of the Graves--Chasles theorem to the third pair of points of intersection $(\ell_0\cap m_0,\ell_2\cap m_2)$ now implies that the circle in (iii) exists.   
\end{proof}

It is now evident that replacing pairs of diagonally neighbouring circles of a checkerboard IC-net by ``large'' circles inscribed in combinatorial $2\times 2$ quadrilaterals as in Lemma~\ref{lem:subdivision} leads to an associated IC-net. The converse is also true.

\begin{theorem}
\label{prop:subdivision}
For every IC-net there exists a one-parameter family of subdivisions into confocal checkerboard IC-nets with the same tangent conic $\alpha$.
\end{theorem}

\begin{proof}
In order to describe a subdivision of an IC-net into a checkerboard IC-net, one should consider a larger part of the net as shown in Figure~\ref{fig:subdivision} (right). Thus, let $\ell_{2n}, m_{2n}$ be lines of an IC-net with the tangent conic $\alpha$ and orient them as in Figure~\ref{fig:subdivision}. Choose an arbitrary line $\ell_1$ touching the conic $\alpha$ or, equivalently, a circle $C_1$ in oriented contact with $\ell_0, \ell_1, m_0$. Define $m_1$ as the fourth line touching $\alpha$ and $C_1$. The circle $C_2$ is the unique circle in oriented contact with $\ell_2, m_0, m_1$, and $C_3$ is in oriented contact with $\ell_0, \ell_1, m_2$. Define the lines $\ell_3$ and $m_3$ by the requirement that they touch $\alpha$ and $C_2$ and $C_3$ respectively. 

By applying the Graves--Chasles theorem and the above lemma, one can now show that the lines $\ell_2, \ell_3, m_2, m_3$ have a common circle $C_4$ in oriented contact. Indeed, according to Lemma \ref{lem:subdivision}, the existence of the circles $C_1$ and the (original) circle $C_1^*$ which touches the lines $\ell_0, \ell_2, m_0, m_2$ gives rise to a circle $C_5$ which is circumscribed by the lines $\ell_1,\ell_2,m_1,m_2$.
The existence of the circles $C_2, C_3,C_5$ implies, in turn, that the points of intersection $\ell_0\cap m_3, \ell_1\cap m_2, \ell_2\cap m_1, \ell_3\cap m_0$ lie on a conic confocal to $\alpha$. Accordingly, the lines $\ell_0, m_0, \ell_3, m_3$ circumscribe a circle $C_6$.  A second application of Lemma \ref{lem:subdivision} to the circles $C_1^*,C_6$ leads to the existence of the circle $C_4$.

Iterative application of the above procedure generates all lines of a checkerboard IC-net subdivision with odd indices, that is, $\ell_{2n+1}, m_{2n+1}$. The only free parameter of this subdivision is encoded in the line $\ell_1$ touching $\alpha$. 
\end{proof}


\section{Laguerre geometric description of checkerboard IC-nets}
\label{s.laguerre}

It is natural to study checkerboard IC-nets in terms of Laguerre geometry. Such a description will allow us to prove fundamental properties of these nets and will finally lead to an explicit description of them. Here, we present a brief description of checkerboard IC-nets using the Blaschke cylinder model, consigning more details of Laguerre geometry to the Appendix. 

Laguerre geometry in the plane deals with oriented circles and oriented straight lines. 
Lines 
$$
\{{\bx}\in \R^2: ({\bv},\bx)_{\R^2}=d\},
$$
with unit normals $\bv\in \mathbb{S}^1$ and $d\in \R$, can be put into correspondence with 3-tuples $(\bv,d)$. Opposite 3-tuples $(\bv,d)$ and  $(-\bv,-d)$ correspond to two different orientations of the same line. Thus, oriented lines are points of the {\em Blaschke cylinder}
$$
{\cal Z}=\{\ell=(\bv,d)\in\R^3: |\bv|=1 \}=\mathbb{S}^1\times \R\subset \R^3. 
$$
Oriented circles 
$$
\{\bx\in \R^2: |\bx-\bc|^2=r^2\}
$$
with centres $\bc\in\R^2$ and signed radii $r\in\R$ are in one-to-one correspondence with planes in $\R^3$ non-parallel to the axis of the Blaschke cylinder $\cal Z$:
$$
S= \{(\bv,d)\in\R^3: (\bc,\bv)_{\R^2}-d-r=0 \}.
$$
Pairs of signed radii $r$ and $-r$ correspond to two different orientations of the same circle. The intersection of such a plane $S$ with $\cal Z$ consists of points of $\cal Z$ which represent lines in oriented contact with the corresponding circle, i.e., oriented lines which are tangent to the circle and exhibit corresponding orientation.
Accordingly, the set of planes in $\R^3$ passing through a given point $\ell=(\bv,d)\in\cal Z$ may be identified with the set of oriented circles in oriented contact with the oriented line $\ell$. Finally, the set of planes in $\R^3$ passing through two points $\ell_1,\ell_2\in\cal Z$, i.e., the set of planes containing the line $L=(\ell_1,\ell_2)\subset \R^3$, is identified with the set of oriented circles in oriented contact with the lines $\ell_1$ and $\ell_2$ (see Figure~\ref{f.Blaschke-cylinder}). The latter identification, together with the fact that four oriented lines are in contact with a common oriented circle if and only if the four corresponding points of the Blaschke cylinder are coplanar, will be crucial for the description of checkerboard IC-nets.
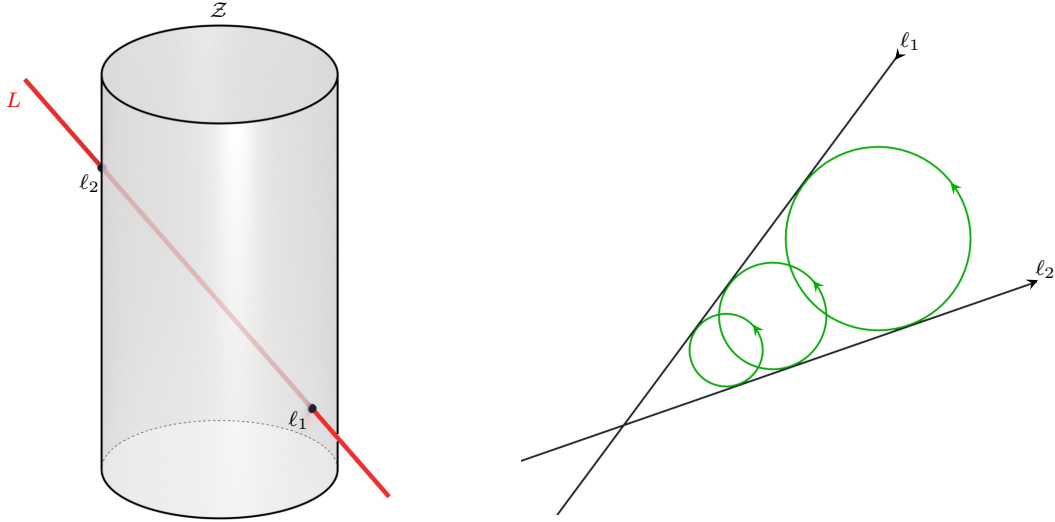
\begin{figure}
  \definecolor{darkgreen}{rgb}{0.,0.7,0.}
  \centering
    \raisebox{-0.5\height}{
    \begin{tikzpicture}[line cap=line join=round,>=stealth,x=1.0cm,y=1.0cm, scale=1]
      \node[anchor=south west,inner sep=0] (image) at (0,0) {\includegraphics[width=0.337\textwidth]{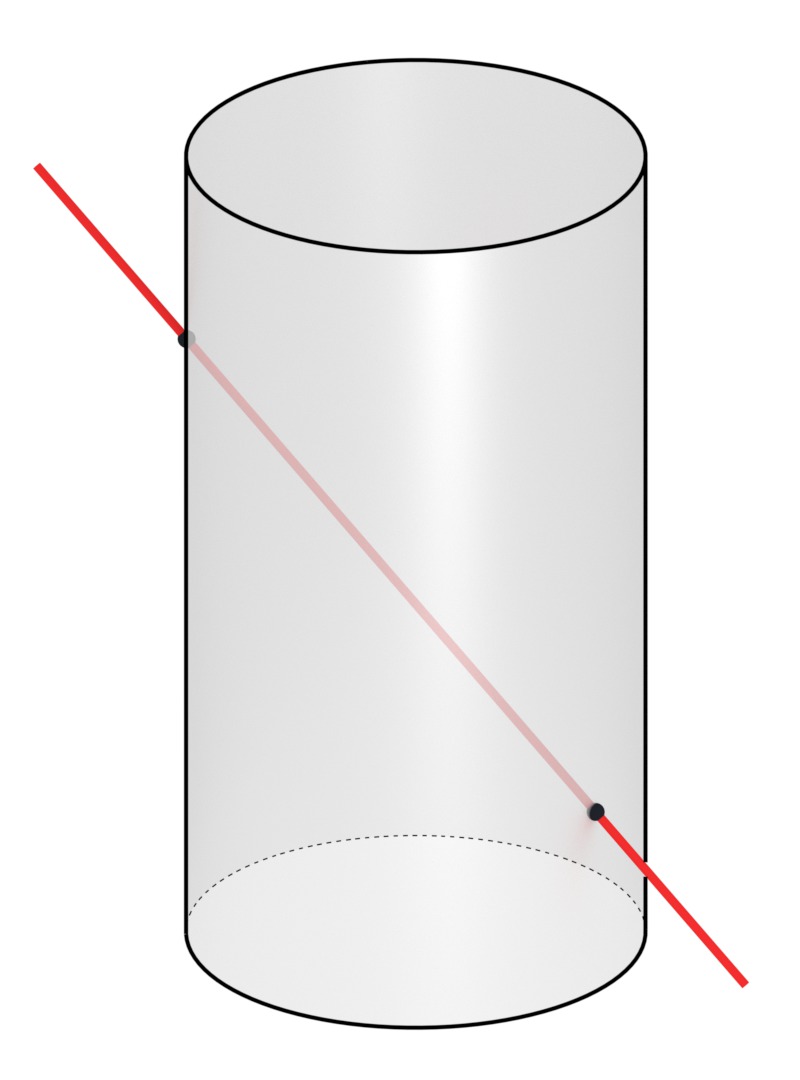}};
      \begin{scriptsize}
        \draw              (2.8, 7.1) node {$\mathcal{Z}$};
        \draw [color=red]  (0.1, 5.9) node {$L$};
        \draw              (3.85, 1.65) node {$\ell_1$};
        \draw              (1.1, 4.8) node {$\ell_2$};
      \end{scriptsize}
    \end{tikzpicture}}
    \hspace{1cm}
    \raisebox{-0.5\height}{
    \begin{tikzpicture}[line cap=line join=round,>=stealth,x=1.0cm,y=1.0cm, scale=1]
      \node[anchor=south west,inner sep=0] (image) at (0,0) {\includegraphics[width=0.4233\textwidth]{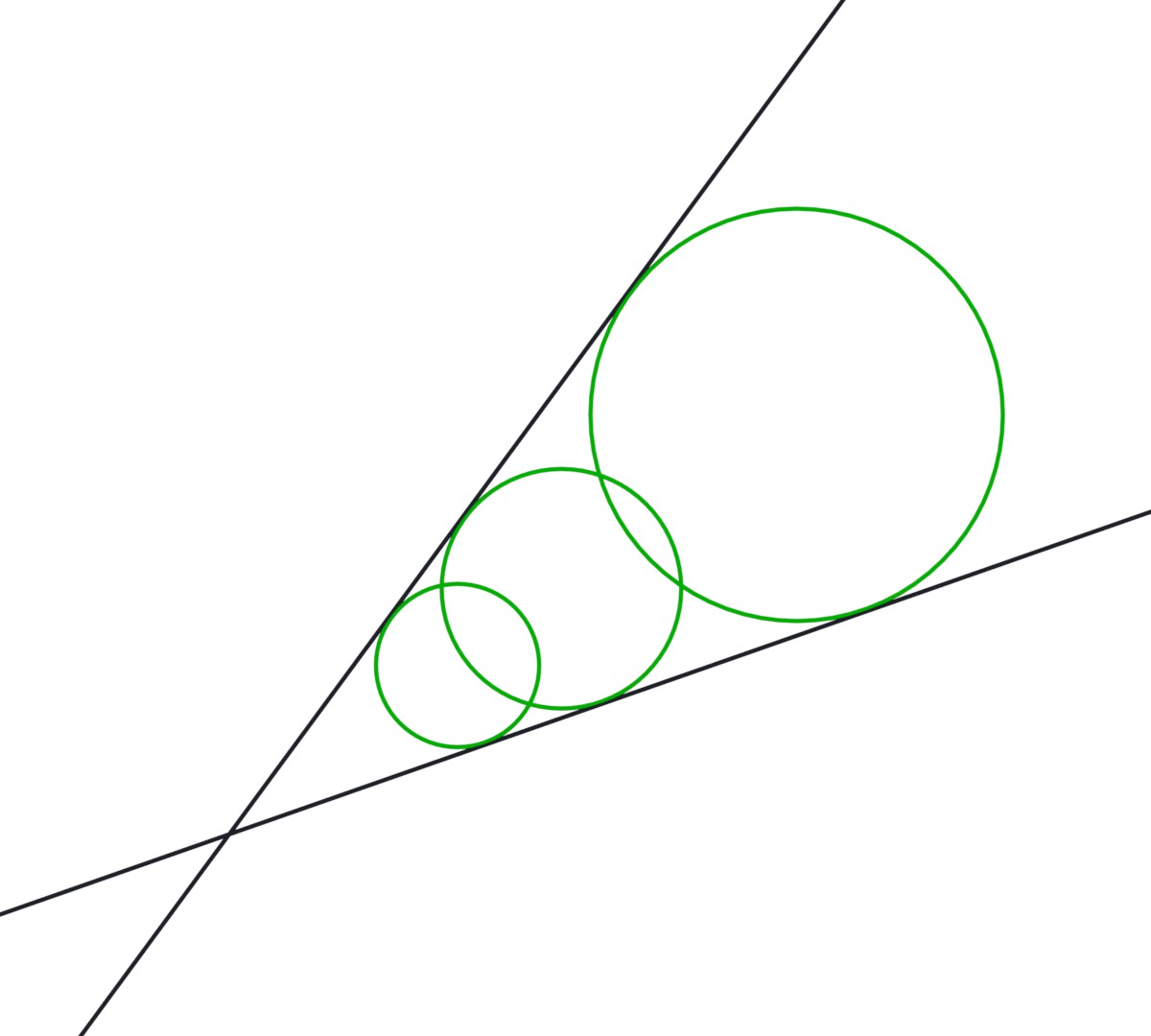}};
      \draw [<-, thick, draw opacity=0] (4.92, 6.03) -- (5.1,6.26);
      \draw [->, thick, draw opacity=0] (6.76,3.092) -- (6.8,3.11);
      \draw [->, thick, draw opacity=0, color=darkgreen] (6.54, 3.35) -- (5.65, 4.4);
      \draw [->, thick, draw opacity=0, color=darkgreen] (4.74, 2.0) -- (3.84, 3.1);
      \draw [->, thick, draw opacity=0, color=darkgreen] (3.89, 1.4) -- (3.05, 2.52);
      \begin{scriptsize}
        \draw (5.12, 6.28) node {$\ell_1$};
        \draw (6.9, 3.25) node {$\ell_2$};
      \end{scriptsize}
    \end{tikzpicture}}
\caption{Blaschke cylinder model: Oriented lines are points $\ell\in \cal Z$ and oriented circles are planes $S$ which do not intersect $\cal Z$ along its generators. Circles in oriented contact with two lines $\ell_1$ and $\ell_2$ are planes containing the line $L=(\ell_1,\ell_2)$.}
\label{f.Blaschke-cylinder}
\end{figure}

\subsection{The Laguerre geometry of checkerboard IC-nets}

As we have seen in Section~\ref{s.elementary}, the following incidence theorem is of crucial importance for the elementary construction of checkerboard IC-nets.

\begin{theorem} (Checkerboard incircles incidence theorem)
\label{t.IC-incidence}
Let $\ell_1,\ldots \ell_{6}$, $m_1,\ldots,m_6$ be 12 oriented lines which are in oriented contact with 12 oriented circles $S_1,\ldots,S_{12}$ as shown in Figure~\ref{f.IC-incidence} (top), corresponding to ``black'' quadrilaterals of a $5\times5$ checkerboard IC-net. In particular, the lines $\ell_1$, $\ell_2$, $m_1$, $m_2$ are in oriented contact with the circle $S_1$, the lines $\ell_3$, $\ell_4$, $m_1$, $m_2$ are in oriented contact with the circle $S_2$ etc. Then, the 13th ``black'' checkerboard quadrilateral also has an inscribed circle, i.e., the lines $\ell_5$, $\ell_6$, $m_5$, $m_6$ have a common circle $S_{13}$ in oriented contact.
\end{theorem}

\begin{figure}
  \begin{center}
      \begin{tikzpicture}[line cap=line join=round,>=stealth,x=1.0cm,y=1.0cm, scale=1.2]
        \begin{scriptsize}
          \draw (0, -0.8) node {$\ell_1$};
          \draw (1, -0.8) node {$\ell_2$};
          \draw (2, -0.8) node {$\ell_3$};
          \draw (3, -0.8) node {$\ell_4$};
          \draw (4, -0.8) node {$\ell_5$};
          \draw (5, -0.8) node {$\ell_6$};
          \draw (-0.8, 0) node {$m_1$};
          \draw (-0.8, 1) node {$m_2$};
          \draw (-0.8, 2) node {$m_3$};
          \draw (-0.8, 3) node {$m_4$};
          \draw (-0.8, 4) node {$m_5$};
          \draw (-0.8, 5) node {$m_6$};
          \draw [color=blue] (0.5, -0.6) node {$L_1$};
          \draw [color=red]  (1.5, -0.6) node {$L_2$};
          \draw [color=blue] (2.5, -0.6) node {$L_3$};
          \draw [color=red]  (3.5, -0.6) node {$L_4$};
          \draw [color=blue] (4.5, -0.6) node {$L_5$};
          \draw [color=blue] (-0.6, 0.5) node {$M_1$};
          \draw [color=red]  (-0.6, 1.5) node {$M_2$};
          \draw [color=blue] (-0.6, 2.5) node {$M_3$};
          \draw [color=red]  (-0.6, 3.5) node {$M_4$};
          \draw [color=blue] (-0.6, 4.5) node {$M_5$};
        \end{scriptsize}
        \clip(-0.5,-0.5) rectangle (6,6);
        \draw [line width=0.8pt, color=red,  domain=-0.3:5.3] plot(1.5,\x);
        \draw [line width=0.8pt, color=red,  domain=-0.3:5.3] plot(3.5,\x);
        \draw [line width=0.8pt, color=red,  domain=-0.3:5.3] plot(\x,1.5);
        \draw [line width=0.8pt, color=red,  domain=-0.3:5.3] plot(\x,3.5);
        \draw [line width=0.8pt, color=blue, domain=-0.3:5.3] plot(0.5,\x);
        \draw [line width=0.8pt, color=blue, domain=-0.3:5.3] plot(2.5,\x);
        \draw [line width=0.8pt, color=blue, domain=-0.3:5.3] plot(4.5,\x);
        \draw [line width=0.8pt, color=blue, domain=-0.3:5.3] plot(\x,0.5);
        \draw [line width=0.8pt, color=blue, domain=-0.3:5.3] plot(\x,2.5);
        \draw [line width=0.8pt, color=blue, domain=-0.3:5.3] plot(\x,4.5);
        \draw [line width=1pt, domain=-0.5:5.5] plot(0,\x);
        \draw [>-, thick] (0,-0.4) -- (0,-0.3);
        \draw [line width=1pt, domain=-0.5:5.5] plot(1,\x);
        \draw [<-, thick] (1,-0.4) -- (1,-0.3);
        \draw [line width=1pt, domain=-0.5:5.5] plot(2,\x);
        \draw [>-, thick] (2,-0.4) -- (2,-0.3);
        \draw [line width=1pt, domain=-0.5:5.5] plot(3,\x);
        \draw [<-, thick] (3,-0.4) -- (3,-0.3);
        \draw [line width=1pt, domain=-0.5:5.5] plot(4,\x);
        \draw [>-, thick] (4,-0.4) -- (4,-0.3);
        \draw [line width=1pt, domain=-0.5:5.5] plot(5,\x);
        \draw [<-, thick] (5,-0.4) -- (5,-0.3);
        \draw [line width=1pt, domain=-0.5:5.5] plot(\x,0);
        \draw [<-, thick] (-0.4,0) -- (-0.3,0);
        \draw [line width=1pt, domain=-0.5:5.5] plot(\x,1);
        \draw [>-, thick] (-0.4,1) -- (-0.3,1);
        \draw [line width=1pt, domain=-0.5:5.5] plot(\x,2);
        \draw [<-, thick] (-0.4,2) -- (-0.3,2);
        \draw [line width=1pt, domain=-0.5:5.5] plot(\x,3);
        \draw [>-, thick] (-0.4,3) -- (-0.3,3);
        \draw [line width=1pt, domain=-0.5:5.5] plot(\x,4);
        \draw [<-, thick] (-0.4,4) -- (-0.3,4);
        \draw [line width=1pt, domain=-0.5:5.5] plot(\x,5);
        \draw [>-, thick] (-0.4,5) -- (-0.3,5);
        \draw [->, thick] (1.0,0.5) arc (0:-685:0.5);
        \draw [->, thick] (3.0,0.5) arc (0:-685:0.5);
        \draw [->, thick] (5.0,0.5) arc (0:-685:0.5);
        \draw [->, thick] (1.0,2.5) arc (0:-685:0.5);
        \draw [->, thick] (3.0,2.5) arc (0:-685:0.5);
        \draw [->, thick] (5.0,2.5) arc (0:-685:0.5);
        \draw [->, thick] (1.0,4.5) arc (0:-685:0.5);
        \draw [->, thick] (3.0,4.5) arc (0:-685:0.5);
        \draw [->, thick, dotted] (5.0,4.5) arc (0:-685:0.5);
        \draw [->, thick] (2.0,1.5) arc (0:415:0.5);
        \draw [->, thick] (4.0,1.5) arc (0:415:0.5);
        \draw [->, thick] (2.0,3.5) arc (0:415:0.5);
        \draw [->, thick] (4.0,3.5) arc (0:415:0.5);
        \begin{tiny}
          \def\co{0.2}
          \draw ({0.5+\co},{0.5+\co}) node {$S_1$};
          \draw ({2.5+\co},{0.5+\co}) node {$S_2$};
          \draw ({4.5+\co},{0.5+\co}) node {$S_3$};
          \draw ({1.5+\co},{1.5+\co}) node {$S_4$};
          \draw ({3.5+\co},{1.5+\co}) node {$S_5$};
          \draw ({0.5+\co},{2.5+\co}) node {$S_6$};
          \draw ({2.5+\co},{2.5+\co}) node {$S_7$};
          \draw ({4.5+\co},{2.5+\co}) node {$S_8$};
          \draw ({1.5+\co},{3.5+\co}) node {$S_9$};
          \draw ({3.5+\co},{3.5+\co}) node {$S_{10}$};
          \draw ({0.5+\co},{4.5+\co}) node {$S_{11}$};
          \draw ({2.5+\co},{4.5+\co}) node {$S_{12}$};
          \draw ({4.5+\co},{4.5+\co}) node {$S_{13}$};
        \end{tiny}
      \end{tikzpicture}
      
      \vspace{0.3cm}
      
      \raisebox{-0.5\height}{
      \begin{tikzpicture}[line cap=line join=round,>=stealth,x=1.0cm,y=1.0cm, scale=0.9]
        \begin{scriptsize}
          \draw [color=blue] (0.5, -1.1) node {$L_1$};
          \draw [color=red]  (1.5, -1.1) node {$M_2$};
          \draw [color=blue] (2.5, -1.1) node {$L_3$};
          \draw [color=red]  (3.5, -1.1) node {$M_4$};
          \draw [color=blue] (4.5, -1.1) node {$L_5$};
          \draw [color=blue] (-1.1, 0.5) node {$M_1$};
          \draw [color=red]  (-1.1, 1.5) node {$L_2$};
          \draw [color=blue] (-1.1, 2.5) node {$M_3$};
          \draw [color=red]  (-1.1, 3.5) node {$L_4$};
          \draw [color=blue] (-1.1, 4.5) node {$M_5$};
          \draw [line width=1pt, color=red,  domain=-0.8:5.8] plot(1.5,\x);
          \draw [line width=1pt, color=red,  domain=-0.8:5.8] plot(3.5,\x);
          \draw [line width=1pt, color=red,  domain=-0.8:5.8] plot(\x,1.5);
          \draw [line width=1pt, color=red,  domain=-0.8:5.8] plot(\x,3.5);
          \draw [line width=1pt, color=blue, domain=-0.8:5.8] plot(0.5,\x);
          \draw [line width=1pt, color=blue, domain=-0.8:5.8] plot(2.5,\x);
          \draw [line width=1pt, color=blue, domain=-0.8:5.8] plot(4.5,\x);
          \draw [line width=1pt, color=blue, domain=-0.8:5.8] plot(\x,0.5);
          \draw [line width=1pt, color=blue, domain=-0.8:5.8] plot(\x,2.5);
          \draw [line width=1pt, color=blue, domain=-0.8:5.8] plot(\x,4.5);
          \coordinate[label=-45:$\ell_1$] (l1) at (0.5,-0.5);
          \fill (l1) circle (2pt);
          \coordinate[label=-45:$\ell_2$] (l2) at (0.5,1.5);
          \fill (l2) circle (2pt);
          \coordinate[label=-45:$\ell_3$] (l3) at (2.5,1.5);
          \fill (l3) circle (2pt);
          \coordinate[label=-45:$\ell_4$] (l4) at (2.5,3.5);
          \fill (l4) circle (2pt);
          \coordinate[label=-45:$\ell_5$] (l5) at (4.5,3.5);
          \fill (l5) circle (2pt);
          \coordinate[label=-45:$\ell_6$] (l6) at (4.5,5.5);
          \fill (l6) circle (2pt);
          \coordinate[label=-45:$m_1$] (m1) at (-0.5,0.5);
          \fill (m1) circle (2pt);
          \coordinate[label=-45:$m_2$] (m2) at (1.5,0.5);
          \fill (m2) circle (2pt);
          \coordinate[label=-45:$m_3$] (m3) at (1.5,2.5);
          \fill (m3) circle (2pt);
          \coordinate[label=-45:$m_4$] (m4) at (3.5,2.5);
          \fill (m4) circle (2pt);
          \coordinate[label=-45:$m_5$] (m5) at (3.5,4.5);
          \fill (m5) circle (2pt);
          \coordinate[label=-45:$m_6$] (m6) at (5.5,4.5);
          \fill (m6) circle (2pt);
          \draw (0.5,0.5) circle (3pt);
          \draw (2.5,0.5) circle (3pt);
          \draw (4.5,0.5) circle (3pt);
          \draw (1.5,1.5) circle (3pt);
          \draw (3.5,1.5) circle (3pt);
          \draw (0.5,2.5) circle (3pt);
          \draw (2.5,2.5) circle (3pt);
          \draw (4.5,2.5) circle (3pt);
          \draw (1.5,3.5) circle (3pt);
          \draw (3.5,3.5) circle (3pt);
          \draw (0.5,4.5) circle (3pt);
          \draw (2.5,4.5) circle (3pt);
          \draw (4.5,4.5) circle (3pt);
        \end{scriptsize}
      \end{tikzpicture}}
    \raisebox{-0.5\height}{
      \begin{tikzpicture}[line cap=line join=round,>=stealth,x=1.0cm,y=1.0cm, scale=1]
        \node[anchor=south west,inner sep=0] (image) at (0,0) {\includegraphics[width=0.5\textwidth]{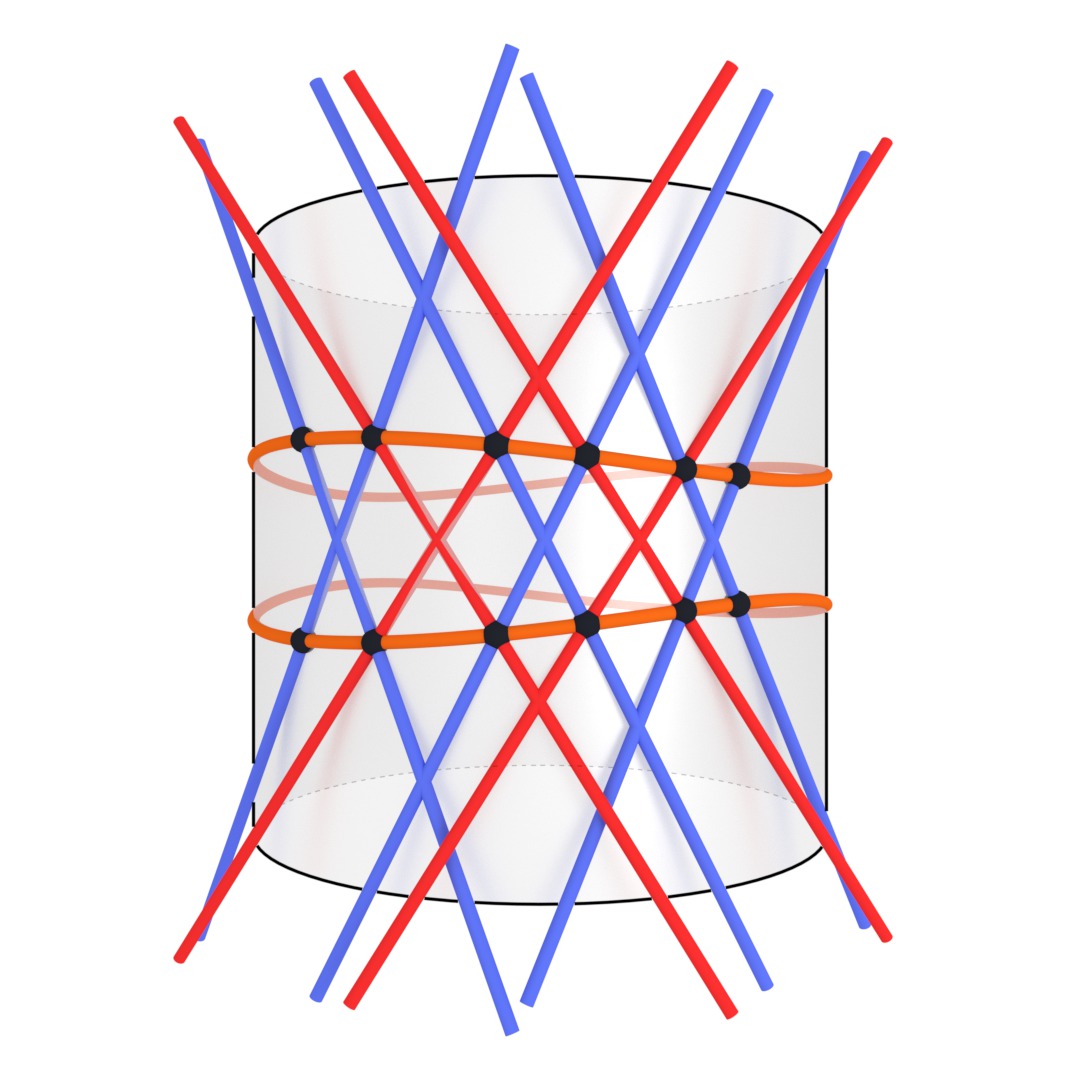}};
        \begin{scriptsize}
          \coordinate[label={[label distance=0.1cm]-90:\contour{white}{$\ell_1$}}] (l1) at (2.25,3.25);
          \coordinate[label={[label distance=0.1cm]90:\contour{white}{$\ell_2$}}] (l2) at (2.75,4.75);
          \coordinate[label={[label distance=0.1cm]-90:\contour{white}{$\ell_3$}}] (l3) at (3.6875,3.3125);
          \coordinate[label={[label distance=0.1cm]90:\contour{white}{$\ell_4$}}] (l4) at (4.375,4.625);
          \coordinate[label={[label distance=0.1cm]-90:\contour{white}{$\ell_5$}}] (l5) at (5.1025,3.4375);
          \coordinate[label={[label distance=0.1cm]90:\contour{white}{$\ell_6$}}] (l6) at (5.5,4.5);
          \coordinate[label={[label distance=0.1cm]90:\contour{white}{$m_1$}}] (m1) at (2.25,4.75);
          \coordinate[label={[label distance=0.1cm]-90:\contour{white}{$m_2$}}] (m2) at (2.75,3.25);
          \coordinate[label={[label distance=0.1cm]90:\contour{white}{$m_3$}}] (m3) at (3.6875,4.75);
          \coordinate[label={[label distance=0.1cm]-90:\contour{white}{$m_4$}}] (m4) at (4.375,3.375);
          \coordinate[label={[label distance=0.1cm]90:\contour{white}{$m_5$}}] (m5) at (5.0625,4.5);
          \coordinate[label={[label distance=0.1cm]-90:\contour{white}{$m_6$}}] (m6) at (5.55,3.5);
      \end{scriptsize}
        \end{tikzpicture}}
  \end{center}
  \caption{
    Checkerboard incircles incidence theorem.
    \emph{Top:}  Existence of the last circle $S_{13}$.
    \emph{Bottom:} The Blaschke cylinder description.
    Pairs of intersecting lines $L_i,M_k$ correspond to circles,
    whereas the points of intersection of the lines $L_i,L_{i+1}$ and $M_k,M_{k+1}$ encapsulate
    the common oriented lines $\ell_{i+1}$ and $m_{k+1}$ respectively.
  }
\label{f.IC-incidence}
\label{f.Blaschke_incidence}
\end{figure}
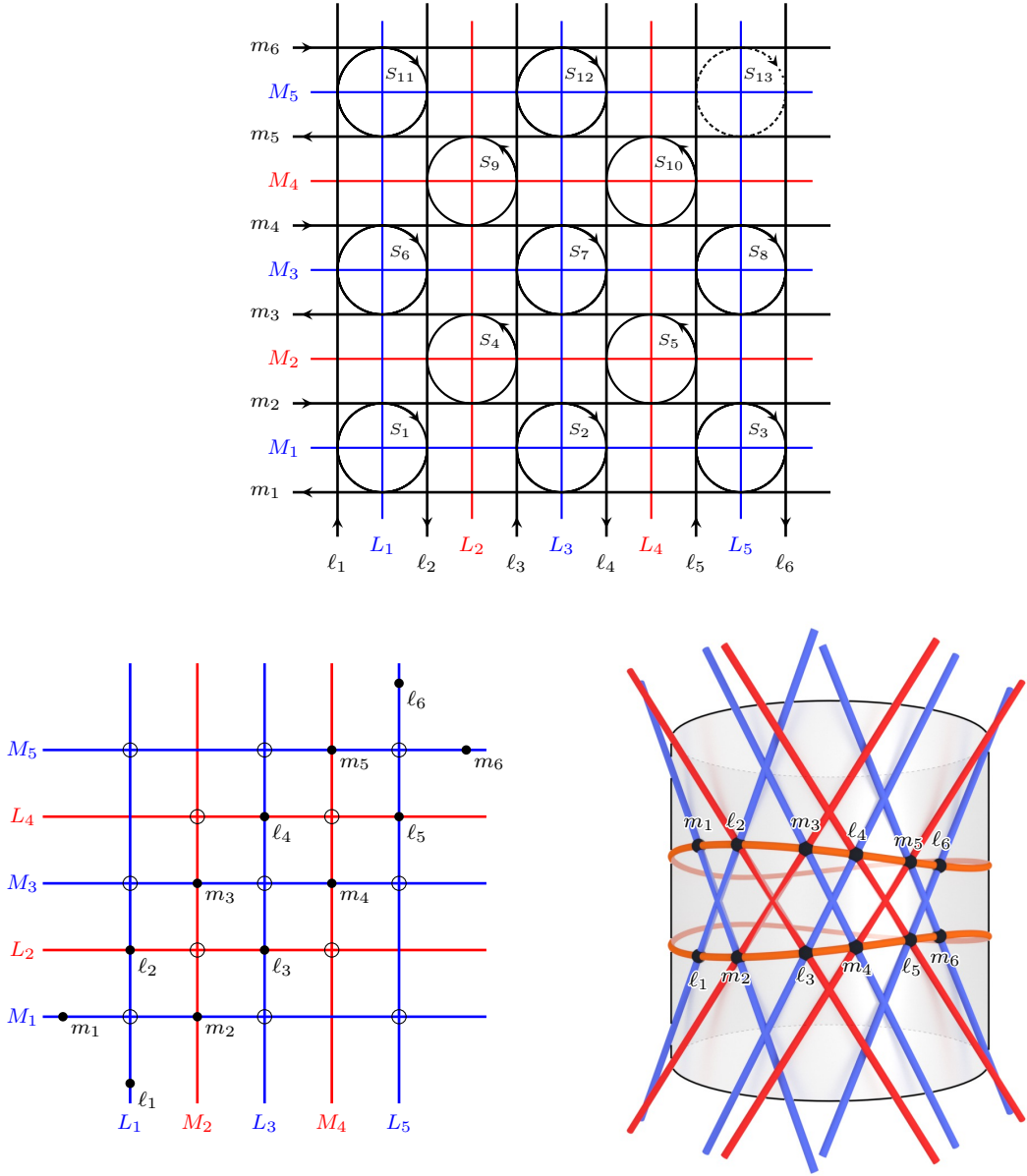

This theorem was originally proven in \cite{AB}. Here, we give a slightly simpler proof which is also instrumental in the explicit integration of the nets. We assume that, modulo the existence of the 12 oriented circles, the 12 oriented lines are in general position. We begin  with a simple lemma used in the proof.
\begin{lemma}
\label{l.quadric_in_pencil_through_line}
Let $p_1,p_2$ be two points which belong to all members of a pencil of quadrics $Q_t$. Then, there exists a unique quadric $Q_{t_{12}}$ from the pencil which contains the whole line \mbox{$L_{12}=(p_1,p_2)$.} If the line $L_{34}=(p_3,p_4)$ associated with another pair of common points $p_3,p_4$ intersects the line $L_{12}$ then the two quadrics  $Q_{t_{12}}$ and $Q_{t_{34}}$ coincide.
\end{lemma} 
\begin{proof}
Even though we will apply this lemma to quadrics in $\mathbb{R}^3$, we will prove it in its natural projective setting. Thus,
let $q_1,q_2$ be two quadratic forms generating the pencil with the quadratic form $q_t=q_1+tq_2$. The points $p_1=[v_1]$, $p_2=[v_2]$ with $v_1$ and $v_2$ being homogeneous coordinates belong to all quadrics of the pencil iff $q_1(v_1)=q_1(v_2)=q_2(v_1)=q_2(v_2)=0$. The line $L_{12}=(p_1,p_2)$ belongs to the quadric determined by $q_{t_{12}}$ iff $q_{t_{12}}(v_1,v_2)=0$ so that $t_{12}=-\frac{q_1(v_1,v_2)}{q_2(v_1,v_2)}$. Vanishing of the denominator is the case when the line lies on the quadric determined by $q_2$. Moreover, if the line $L_{34}=(p_3,p_4)$ passing through another pair of common points $p_3,p_4$ intersects the line $L_{12}$ then the point of intersection and $p_3,p_4$ belong to the quadric $Q_{t_{12}}$. Accordingly, the line $L_{34}$ is contained in $Q_{t_{12}}$ so that $Q_{t_{12}}=Q_{t_{34}}$. 
\end{proof}

\begin{proof}[Proof of Theorem \ref{t.IC-incidence}]
Here, it is convenient to interpret the statement of the theorem in terms of the Blaschke cylinder model. Thus, as explained above the lines $L_i$ intersecting the Blaschke cylinder $\cal Z$ in the points $\ell_i$ and $\ell_{i+1}$ describe one-parameter sets of circles in oriented contact with the lines $\ell_i$ and $\ell_{i+1}$. Some of the lines $L_i$ and $M_k$ intersect. For instance, the one-parameter families of circles corresponding to $L_1$ and $M_1$ contain the common circle $S_1$, corresponding to the plane determined by $L_1$ and $M_1$. We obtain the incidence picture shown in Figure~\ref{f.Blaschke_incidence} (bottom), wherein the points of intersection of the relevant pairs $L_i, M_k$ are indicated by small circles. Moreover the lines $L_i$ and $L_{i+1}$ intersect since they pass through the same points in $\cal Z$ and so do the lines $M_k$ and $M_{k+1}$. The points of intersection correspond to the lines $\ell_{i+1}$ and $m_{k+1}$ respectively. For example $\ell_2=L_1\cap L_2$.

In order to prove the existence of the circle $S_{13}$, we have to show that the lines $L_5$ and $M_5$ intersect. We first note that the three lines $L_1, L_3, L_5$ determine a hyperboloid ${\cal H}\subset \R^3$ and recall that a line which intersects a quadric in three points is contained in the quadric.  Accordingly, since $M_1$ and $M_3$ intersect the lines $L_1, L_3, L_5$, these are contained in $\mathcal{H}$. The intersection of the lines $L_5$ and $M_5$ is equivalent to the inclusion $M_5\subset {\cal H}$. The latter property may be proven as follows.

Since $\ell_2, \ell_3$ belong to both quadrics ${\cal Z},{\cal H}$, Lemma~\ref{l.quadric_in_pencil_through_line} implies that there exists a unique quadric $\tilde{\cal H}$ in the pencil of quadrics generated by ${\cal H}$ and ${\cal Z}$ which contains the whole line $L_2$. We now consider the three points $m_2, m_3, L_2\cap M_2$ of $M_2$. Since $m_2$ and $m_3$ are common to $\mathcal{H}$ and $\mathcal{Z}$, these are contained in $\tilde{\mathcal{H}}$. Hence, $m_2, m_3, L_2\cap M_2\in\tilde{\mathcal{H}}$ so that $M_2\subset\tilde{\mathcal{H}}$. This implies, in turn, that $\ell_4,\ell_5,L_4\cap M_2\in \tilde{\mathcal{H}}$ and, hence, $L_4\subset \tilde{\mathcal{H}}$. Consequently, $m_4, L_2\cap M_4, L_4\cap M_4\in \tilde{\mathcal{H}}$ so that $M_4\subset \tilde{\mathcal{H}}$. In particular, $m_5$ lies in $\tilde{\mathcal{H}}$ (and $\mathcal{Z}$) and, therefore, in $\mathcal{H}$. Moreover, $L_1\cap M_5$ and $L_3\cap M_5$ lie in $\mathcal{H}$ which finally implies that $M_5\subset \mathcal{H}$.
\end{proof}

It is observed that iterative application of Theorem \ref{t.IC-incidence} leads to the unique construction of an arbitrarily large checkerboard IC-net with lines $L_n$ and $M_n$, $n\in\Z$. The hyperboloids $\cal H$ and $\tilde{\cal H}$ as constructed above then contain all lines $L_{2k+1},M_{2k+1}$  and $L_{2k},M_{2k}$ respectively. Thus, we have come to the important conclusion that a checkerboard IC-net encodes two quadrics which belong to a pencil containing the Blaschke cylinder $\mathcal{Z}$.
 
\begin{corollary}
In the Blaschke cylinder model, the lines $L_{2k+1}=(\ell_{2k+1},\ell_{2k+2})$, $M_{2k+1}=(m_{2k+1},m_{2k+2})$ and $L_{2k}=(\ell_{2k},\ell_{2k+1})$, $M_{2k}=(m_{2k},m_{2k+1})$ associated with a checkerboard IC-net ``in general position'' are generators of hyperboloids $\cal H$ and $\tilde{\cal H}$ respectively which belong to a pencil of quadrics containing the Blaschke cylinder $\cal Z$.
\end{corollary}

For future reference we denote the common curve of intersection of the above-mentioned quadrics by
$$
{\cal C}={\cal H}\cap {\cal Z}={\tilde{\cal H}}\cap {\cal Z}={\cal H}\cap{\tilde{\cal H}}\cap {\cal Z}.
$$

\begin{definition}
A (non-empty) curve of intersection of the Blaschke cylinder with a quadric is called a {\em hypercycle base curve}.
\end{definition}

The straight lines corresponding to the points of a hypercycle base curve are tangent to a curve in the plane which is generically of degree 8.
This planar curve is called a {\em hypercycle} \cite{Blaschke}. It is noted that a hypercycle base curve and the corresponding hypercycle are merely two different incarnations of the same object $\mathcal{C}$. In terms of this terminology, we have proven the following theorem (see Figure \ref{f.hypercycle}).

\begin{theorem}
The lines of a checkerboard IC-net are tangent to a hypercycle.
\end{theorem}

\begin{figure}
  \centerline{
    \includegraphics[width=0.49\textwidth]{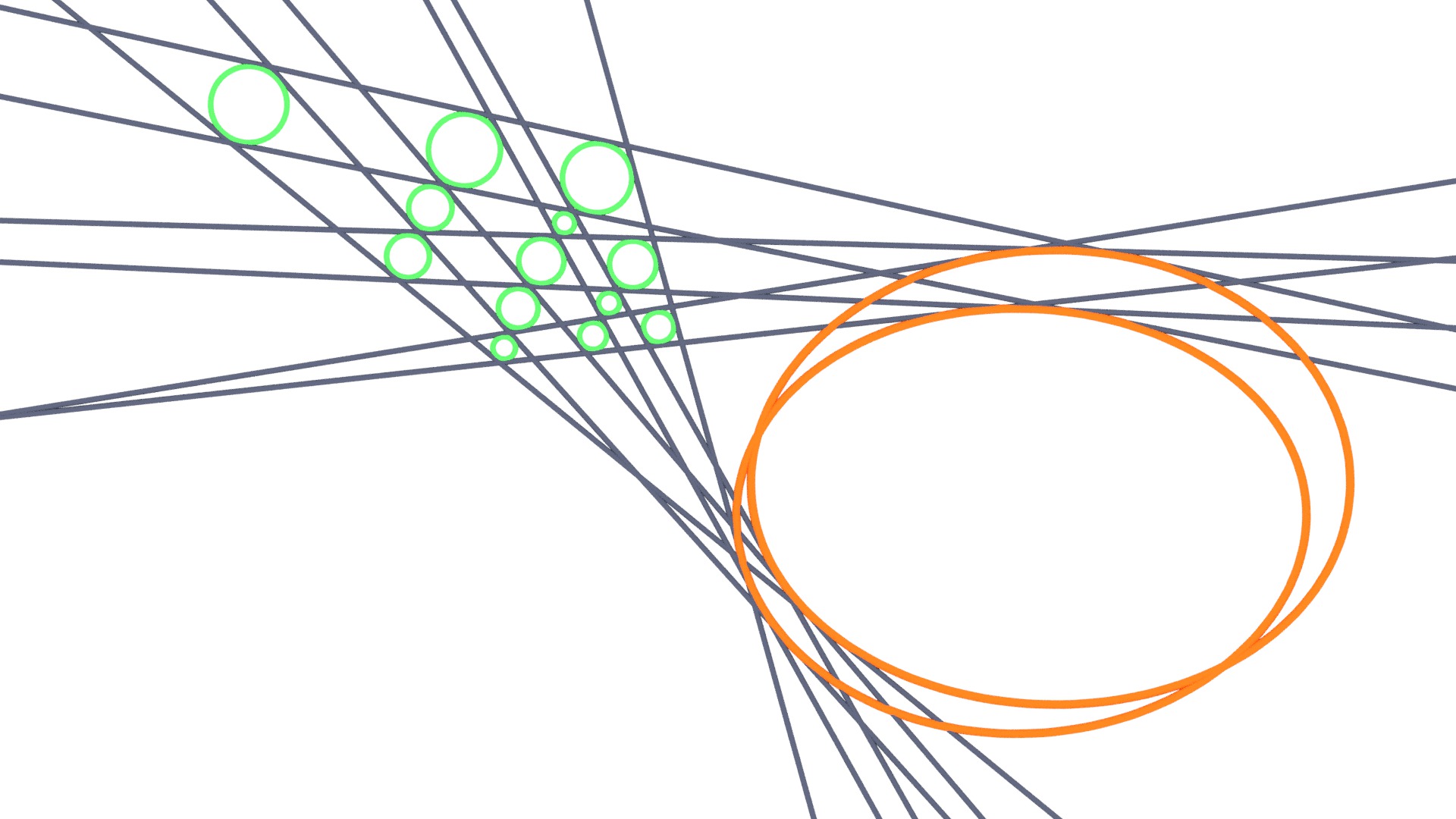}
    \includegraphics[width=0.49\textwidth]{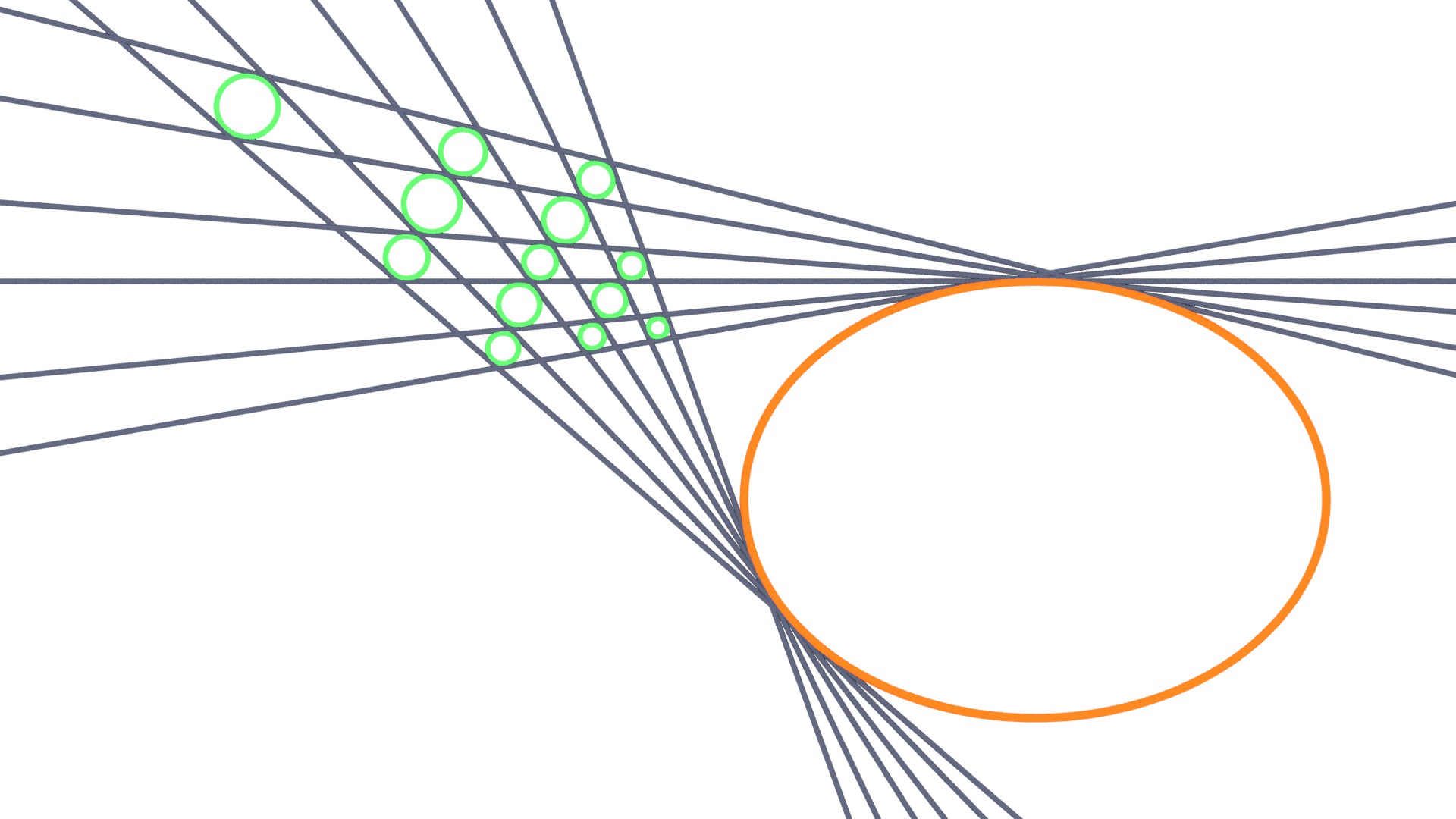}
  }
\caption{The lines of a checkerboard IC-net are tangent to a hypercycle. The hypercycle on the right consists of two coinciding ellipses of different orientations, encapsulating the tangent conic of a confocal checkerboard IC-net.}
\label{f.hypercycle}
\end{figure}

In the following, it is convenient to adopt a notion of genericity.

\begin{definition}
A quadric in $\R^3$ is termed {\em generic} if it does not contain the ``point at infinity'' on the axis of the Blaschke cylinder. A generic hypercycle base curve is the intersection of a generic quadric with the Blaschke cylinder. A pencil of quadrics containing the Blaschke cylinder is generic if one and, therefore, all quadrics of the pencil other than the Blaschke cylinder are generic. A checkerboard IC-net is generic if it is associated with a generic hypercycle base curve or, equivalently, a generic pencil of quadrics. 
\end{definition}

We note that the standard square grid (appropriately oriented) does not constitute a generic checkerboard IC-net. Moreover, the above definition implies that the hypercycle base curve ${\cal C}\subset {\cal Z}$ associated with a generic quadric has bounded $d$-coordinate. 

\begin{corollary}
The lines of a generic checkerboard IC-net lie in bounded distance to the origin.
\end{corollary}

\subsection{Construction of checkerboard IC-nets}

We are now in a position to formulate the construction of checkerboard IC-nets in the Blaschke cylinder model. One starts with two one-sheeted hyperboloids ${\cal H},\tilde{\cal H}$ of a pencil of quadrics containing $\cal Z$ and two points $\ell_1,m_1$ of the hypercycle base curve ${\cal C}={\cal H}\cap{\cal Z}$. Let us make a choice and refer to one of the families of straight lines (generators) of the hyperboloid $\cal H$ as the $L$-family and the other one as the $M$-family. Make the choice of the $L$- and $M$-families on $\tilde{\cal H}$ as well. Now, the checkerboard IC-net is uniquely determined in the following sense. Label by $L_1$ the line from the $L$-family of $\cal H$ passing through $\ell_1$ and denote by $\ell_2$ its second point of intersection with $\cal C$. Similarly, the point $m_2\in {\cal C}$ is the second point of intersection with $\cal Z$ of the $M$-line of $\cal H$ labelled by $M_1$ passing through $m_1$. Proceed further with the generators of the hyperboloid $\tilde{\cal H}$, where $L_2$ and $M_2$ are the $L$-line and $M$-line of $\tilde{\cal{H}}$ passing through $\ell_2$ and $m_2$ respectively. The additional points of intersection with $\mathcal{C}$ are denoted by $\ell_3$ and $m_3$ respectively. By alternating in this manner between the hyperboloids $\mathcal{H}$ and $\tilde{\mathcal{H}}$, the lines of a checkerboard IC-net $\ell_n$ and $m_n$ represented as points of the hypercycle base curve which are connected by generators $L_n = (\ell_n,\ell_{n+1})$ and $M_n = (m_n,m_{n+1})$ may be constructed (see Figure \ref{f.Blaschke_construction}). Thus, we conclude with the following theorem.
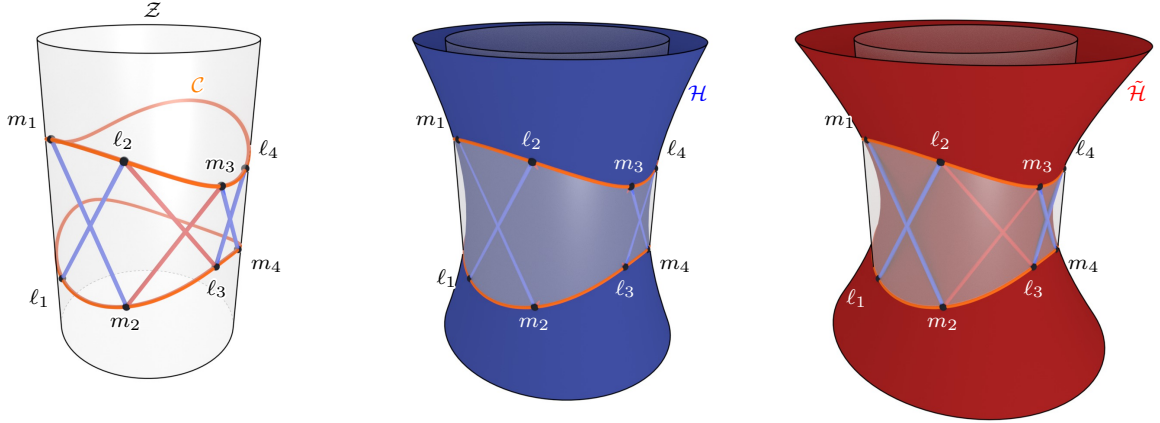
\begin{figure}
  \centering
  \raisebox{-0.5\height}{
    \begin{tikzpicture}[line cap=line join=round,>=stealth,x=1.0cm,y=1.0cm, scale=1]
      \node[anchor=south west,inner sep=0] (image) at (0,0) {\includegraphics[trim={3.2cm 0 3.2cm 0},clip,width=0.32\textwidth]{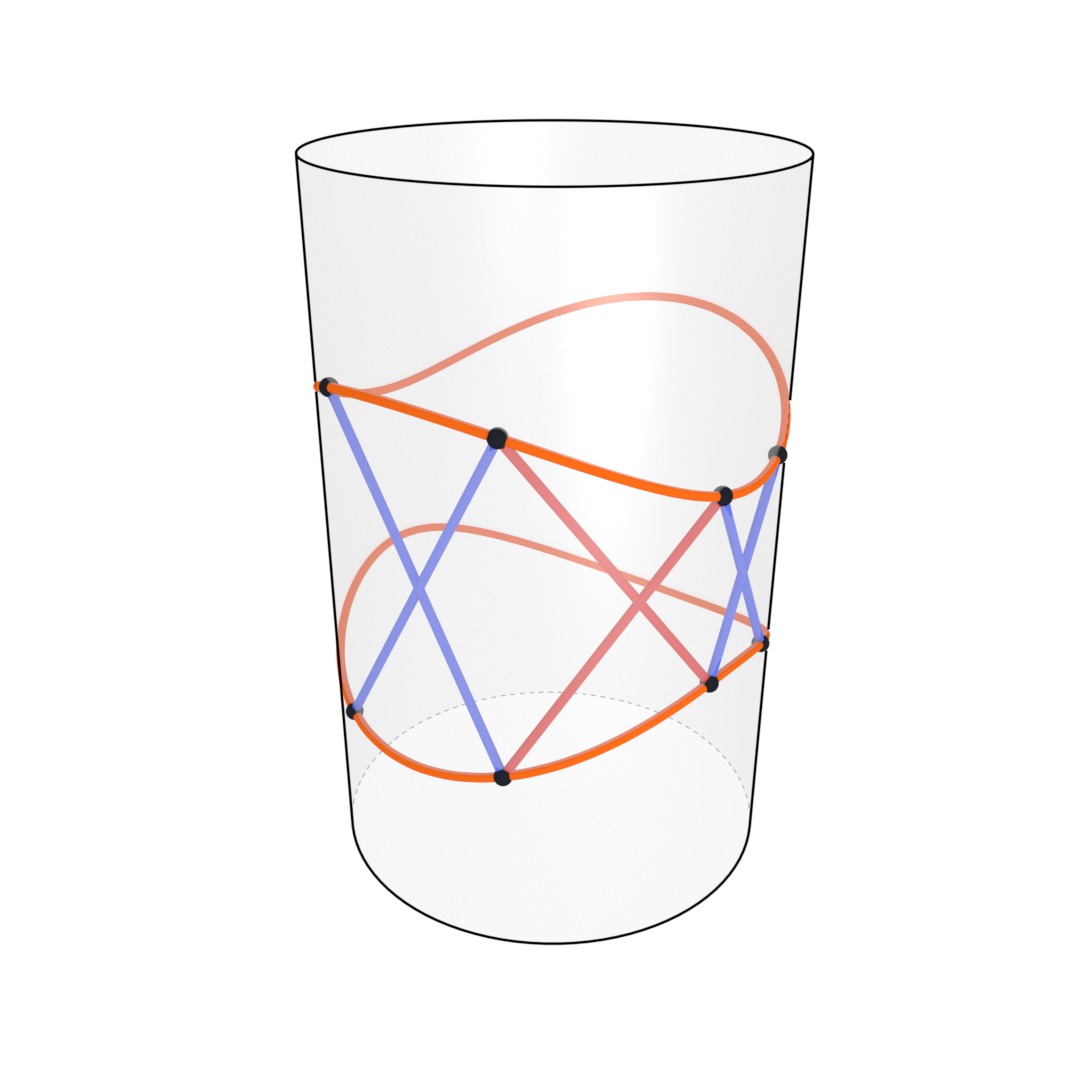}};
      \begin{scriptsize}
        \coordinate[label={[label distance=0.06cm]90:\contour{white}{$\mathcal{Z}$}}] (l1) at (2.7,5.4);
        \coordinate[label={[label distance=0.06cm]90:\textcolor{orange}{\contour{white}{$\mathcal{C}$}}}] (l1) at (3.3,4.4);
        \coordinate[label={[label distance=0.06cm]225:\contour{white}{$\ell_1$}}] (l1) at (1.5,2.1);
        \coordinate[label={[label distance=0.06cm]90:\contour{white}{$\ell_2$}}] (l2) at (2.32,3.65);
        \coordinate[label={[label distance=0.06cm]-90:\contour{white}{$\ell_3$}}] (l3) at (3.55,2.25);
        \coordinate[label={[label distance=0.06cm]45:\contour{white}{$\ell_4$}}] (l4) at (3.95,3.55);
        \coordinate[label={[label distance=0.06cm]135:\contour{white}{$m_1$}}] (m1) at (1.35,3.95);
        \coordinate[label={[label distance=0.06cm]-90:\contour{white}{$m_2$}}] (m2) at (2.35,1.75);
        \coordinate[label={[label distance=0.06cm]90:\contour{white}{$m_3$}}] (m3) at (3.6,3.35);
        \coordinate[label={[label distance=0.06cm]-45:\contour{white}{$m_4$}}] (m4) at (3.85,2.5);
      \end{scriptsize}
    \end{tikzpicture}}
  \raisebox{-0.5\height}{
    \begin{tikzpicture}[line cap=line join=round,>=stealth,x=1.0cm,y=1.0cm, scale=1]
      \node[anchor=south west,inner sep=0] (image) at (0,0) {\includegraphics[trim={3.2cm 0 3.2cm 0},clip,width=0.32\textwidth]{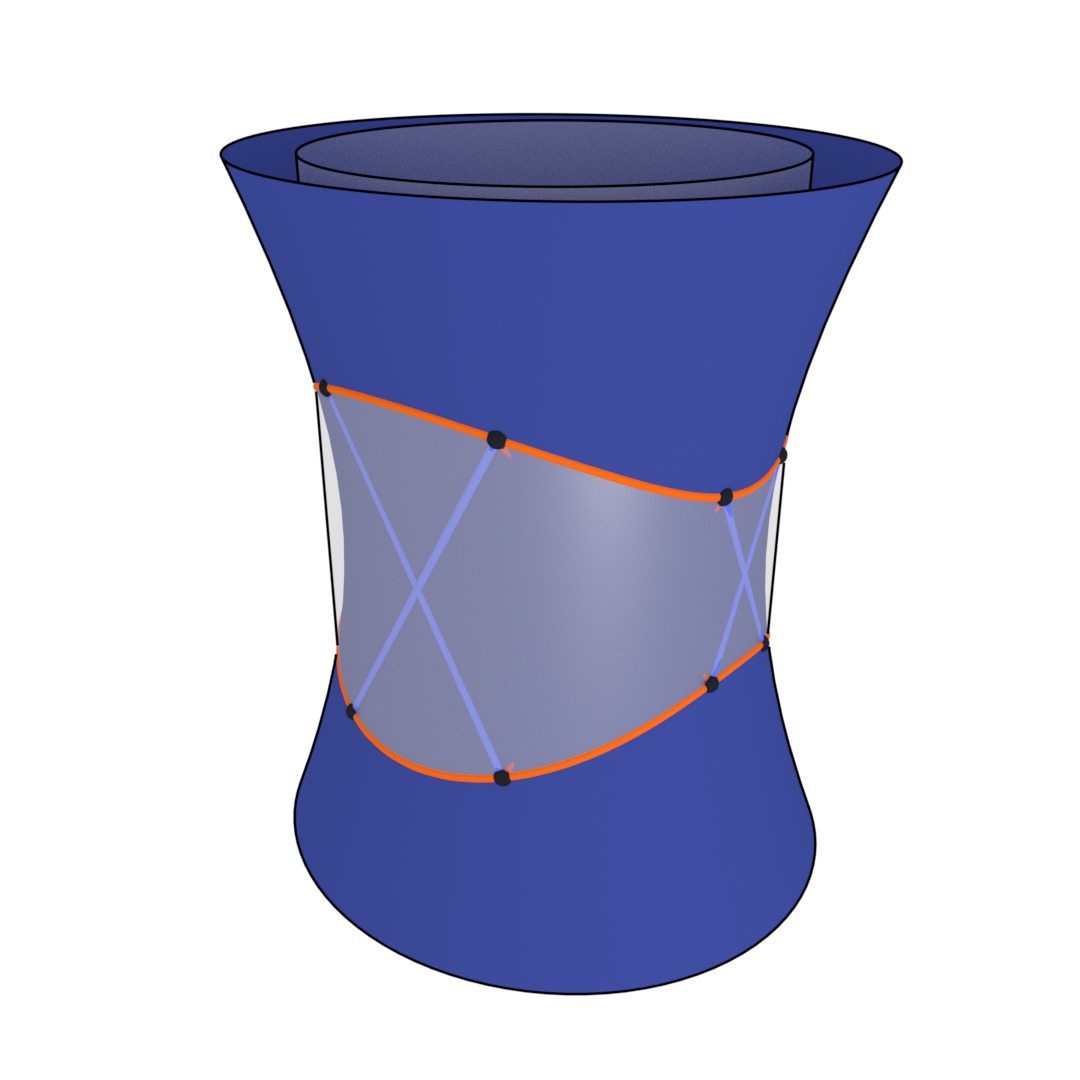}};
      \begin{scriptsize}
        \coordinate[label={[label distance=0.06cm]90:\textcolor{blue}{\contour{white}{$\mathcal{H}$}}}] (l1) at (4.5,4.3);
        \coordinate[label={[label distance=0.06cm]180:\contour{white}{$\ell_1$}}] (l1) at (1.5,2.1);
        \coordinate[label={[label distance=0.06cm]90:\textcolor{white}{$\ell_2$}}] (l2) at (2.32,3.65);
        \coordinate[label={[label distance=0.06cm]-90:\textcolor{white}{$\ell_3$}}] (l3) at (3.55,2.25);
        \coordinate[label={[label distance=0.06cm]45:\contour{white}{$\ell_4$}}] (l4) at (3.95,3.55);
        \coordinate[label={[label distance=0.06cm]135:\contour{white}{$m_1$}}] (m1) at (1.35,3.95);
        \coordinate[label={[label distance=0.06cm]-90:\textcolor{white}{$m_2$}}] (m2) at (2.35,1.75);
        \coordinate[label={[label distance=0.06cm]90:\textcolor{white}{$m_3$}}] (m3) at (3.6,3.35);
        \coordinate[label={[label distance=0.06cm]-45:\contour{white}{$m_4$}}] (m4) at (3.85,2.5);
      \end{scriptsize}
    \end{tikzpicture}}
  \raisebox{-0.5\height}{
    \begin{tikzpicture}[line cap=line join=round,>=stealth,x=1.0cm,y=1.0cm, scale=1]
      \node[anchor=south west,inner sep=0] (image) at (0,0) {\includegraphics[trim={3.2cm 0 3.2cm 0},clip,width=0.32\textwidth]{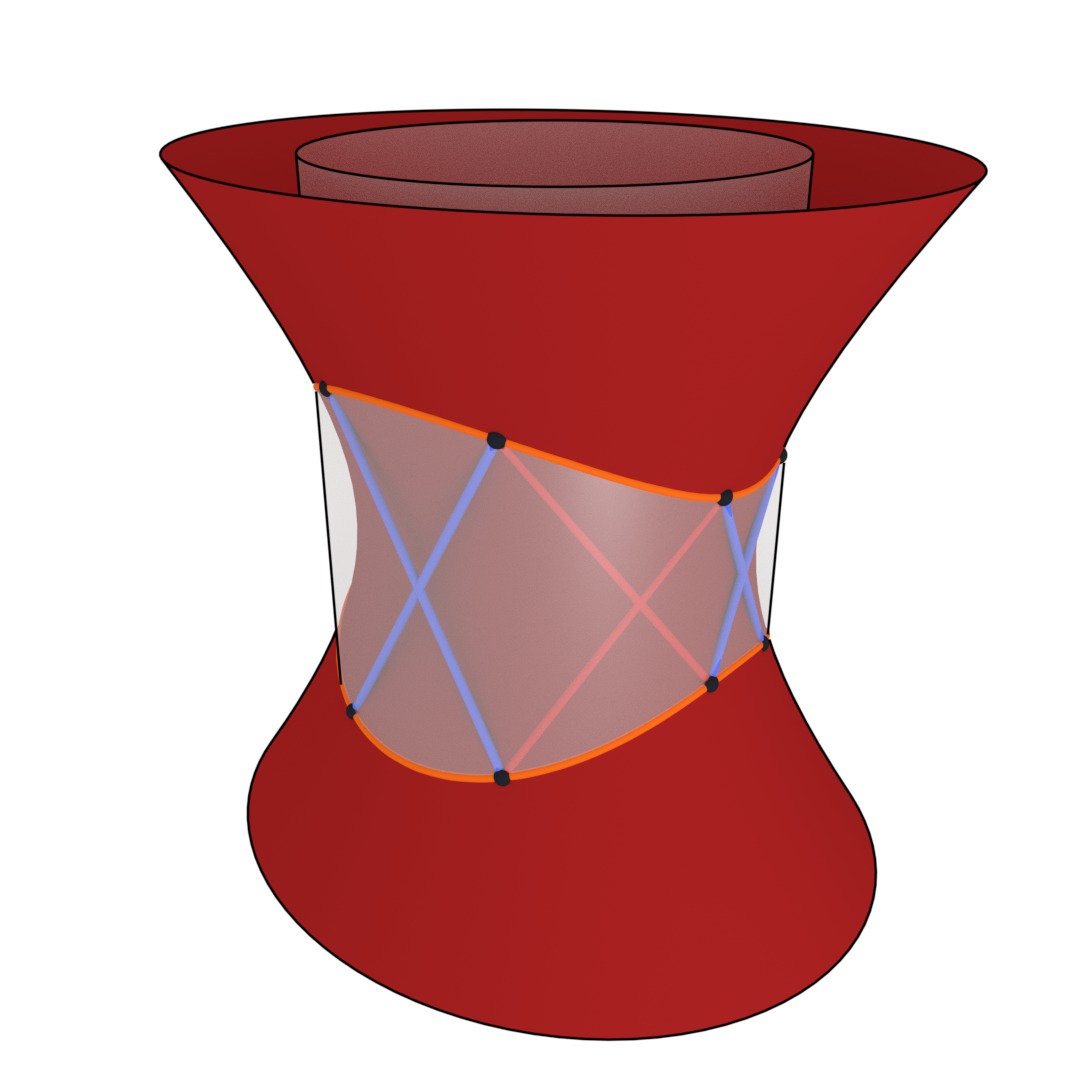}};
      \begin{scriptsize}
        \coordinate[label={[label distance=0.06cm]90:\textcolor{red}{\contour{white}{$\mathcal{\tilde{H}}$}}}] (l1) at (4.9,4.3);
        \coordinate[label={[label distance=0.06cm]225:\textcolor{white}{$\ell_1$}}] (l1) at (1.5,2.1);
        \coordinate[label={[label distance=0.06cm]90:\textcolor{white}{$\ell_2$}}] (l2) at (2.32,3.65);
        \coordinate[label={[label distance=0.06cm]-90:\textcolor{white}{$\ell_3$}}] (l3) at (3.55,2.25);
        \coordinate[label={[label distance=0.06cm]45:\contour{white}{$\ell_4$}}] (l4) at (3.95,3.55);
        \coordinate[label={[label distance=0.06cm]135:\contour{white}{$m_1$}}] (m1) at (1.35,3.95);
        \coordinate[label={[label distance=0.06cm]-90:\textcolor{white}{$m_2$}}] (m2) at (2.35,1.75);
        \coordinate[label={[label distance=0.06cm]90:\textcolor{white}{$m_3$}}] (m3) at (3.6,3.35);
        \coordinate[label={[label distance=0.06cm]-45:\contour{white}{$m_4$}}] (m4) at (3.85,2.5);
      \end{scriptsize}
    \end{tikzpicture}}
  \caption{
    Construction of checkerboard IC-nets in the Blaschke cylinder model.
    The lines  $L_{2k+1},M_{2k+1}$ (red) and $L_{2k},M_{2k}$ (blue) are generators of the quadrics $\mathcal{H}$
    and $\tilde{\mathcal{H}}$ respectively.
  }
\label{f.Blaschke_construction}
\end{figure}

\begin{theorem} {\bf (Construction of checkerboard IC-nets in the Blaschke model)}
\label{th.Blaschke_construction}
A checkerboard IC-net is uniquely determined by hyperboloids ${\cal H},\tilde{\cal H}$ (with marked $L$- and $M$-families of generators on each hyperboloid) of a pencil of quadrics containing $\cal Z$ and two points $\ell_1, m_1$ of the hypercycle base curve ${\cal C}={\cal H}\cap{\cal Z}$. 
\end{theorem}

We now briefly discuss some illustrative classes of checkerboard IC-nets.

\subsubsection{(Confocal checkerboard) IC-nets}

In \cite{AB}, confocal checkerboard IC-nets are characterised by the property that their lines are tangent to a conic. In this paper, we make the assumption that the conic is either an ellipse or a hyperbola so that, by means of a rotation and a translation (which constitute special Laguerre transformations, this conic may be brought into the form
\begin{equation}
\label{eq.conic_normalized}
\frac{x^2}{a}+\frac{y^2}{b}=1.
\end{equation}
Tangent lines to the conic are given by
\begin{equation}
\label{eq.tangent_line}
\frac{xx_0}{a}+\frac{yy_0}{b}=1
\end{equation}
with
\begin{equation}
\label{eq.tangent_line_point}
  \frac{x_0^2}{a}+\frac{y_0^2}{b}=1.
\end{equation}
If we set $\bv = (v,w)$ in the $(\bv,d)$ description employed at the beginning of Section \ref{s.laguerre}, this leads to $v=x_0d/a$ and $w=y_0d/b$ so that \eqref{eq.tangent_line_point} may be expressed in terms of the cone \begin{equation}
\label{eq.cone_confocal}
av^2+bw^2=d^2.
\end{equation}
We refer to the latter as ``elliptic'' if $a>0,b>0$ and ``hyperbolic'' if $ab<0$. The hypercycle base curve $\cal C$ is the intersection of the cone \eqref{eq.cone_confocal} with the Blaschke cylinder $v^2+w^2=1$. It has two connected components and is symmetric with respect to the change of orientation $(v,w,d)\to (-v,-w,-d)$. In the plane, the two components of the hypercycle are the conic~\eqref{eq.conic_normalized} equipped with two different orientations.
Confocal checkerboard IC-nets are parametrised explicitly in Section \ref{s.confocal}.
\begin{figure}
  \centering
  \includegraphics[width=0.45\textwidth]{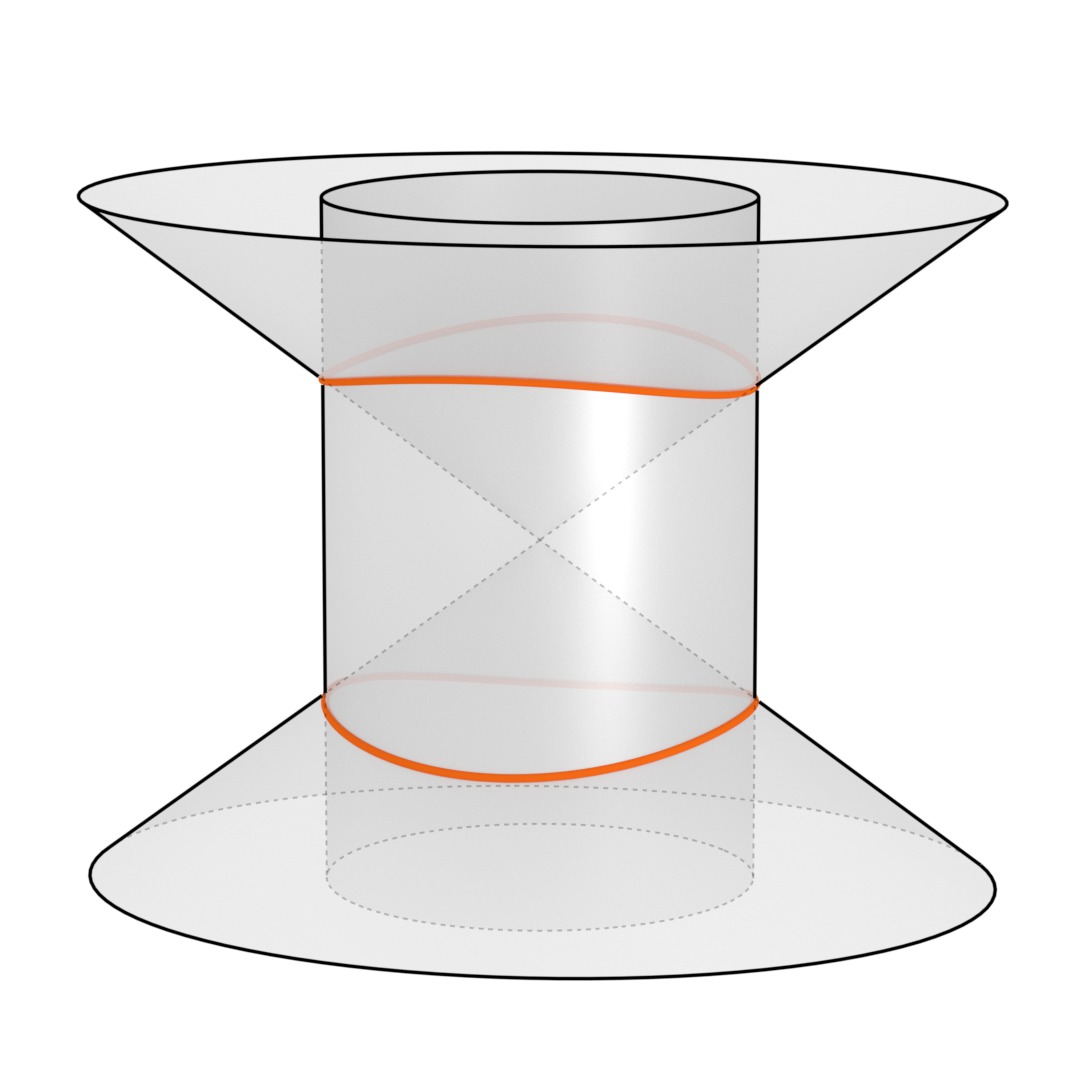}
  \includegraphics[width=0.45\textwidth]{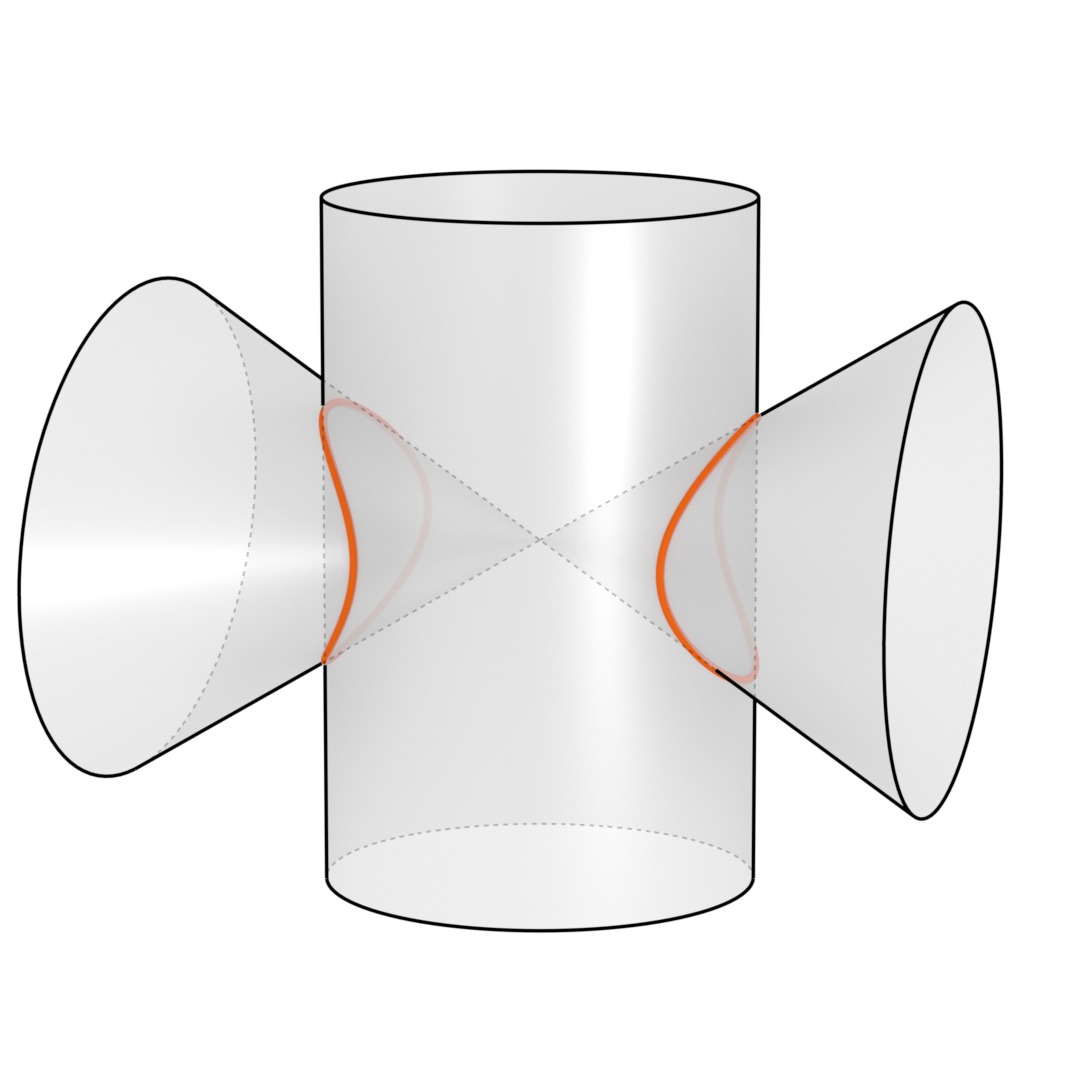}
  \caption{The elliptic and hyperbolic cones \protect\eqref{eq.cone_confocal} and the hypercycle base curves of confocal checkerboard IC-nets in the Blaschke cylinder model.}
\label{f.hyper_ellipse}
\end{figure}

As observed in Remark \ref{rem.IC}, any IC-net may be regarded as a (confocal) checkerboard IC-net by interpreting each line of the IC-net as a ``double line'', that is, two identical lines of opposite orientation represented by $\pm(v,w,d)$. Accordingly, one of the hyperboloids of the corresponding checkerboard IC-net constitutes a cone of the form \eqref{eq.cone_confocal}. Indeed, the latter may be regarded as a characterisation of IC-nets in the context of checkerboard IC-nets.

\subsubsection{Degeneration to rhombic checkerboard IC-nets}

Hyperbolic confocal checkerboard IC-nets may be regarded as deformations of ``rhombic'' checkerboard IC-nets, that is, checkerboard IC-nets composed of identical rhombi. In order to show this, let $\mathcal{H}$ and $\tilde{\mathcal{H}}$ be the two hyperboloids underlying a hyperbolic checkerboard IC-net $\mathcal{N}$. Since the two hyperboloids belong to a pencil of quadrics, $\tilde{{\cal H}}$ may be regarded as a deformation of $\cal H$ with the parameter of the pencil playing the role of the deformation parameter. This implies, in turn, that we may interpret $\mathcal{N}$ as a deformation of a confocal checkerboard IC-net $\mathcal{N}_c$ for which $\tilde{\mathcal{H}}=\mathcal{H} = \mathcal{H}_c$ coincide. According to the construction of checkerboard IC-nets summarised in Theorem \ref{th.Blaschke_construction}, $\mathcal{N}_c$ can only be non-trivial if the $L$-family on $\mathcal{H}$ coincides with the $M$-family on $\tilde{\mathcal{H}}$ (and vice versa), thereby representing the same family of generators on $\mathcal{H}_c$. Hence, we here assume that the choice of the $L$-and $M$-families on the hyperboloids associated with $\mathcal{N}$ has been made in such a manner that, in the limit $\tilde{\cal H}\rightarrow {\cal H}$, this non-triviality requirement is met.

\begin{figure}
  \centering
  \raisebox{-0.5\height}{\includegraphics[scale=0.06]{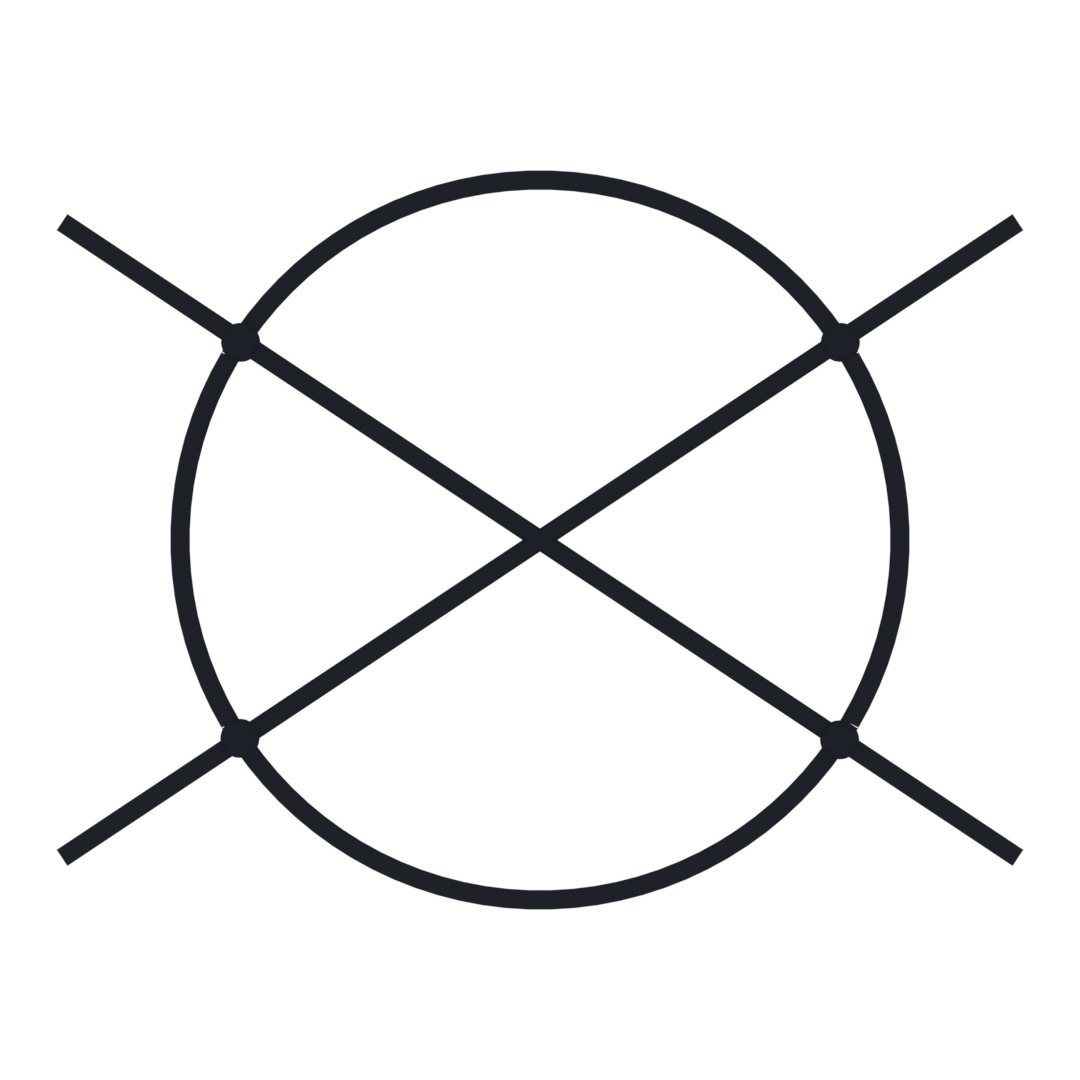}}
  \raisebox{-0.5\height}{\includegraphics[scale=0.17]{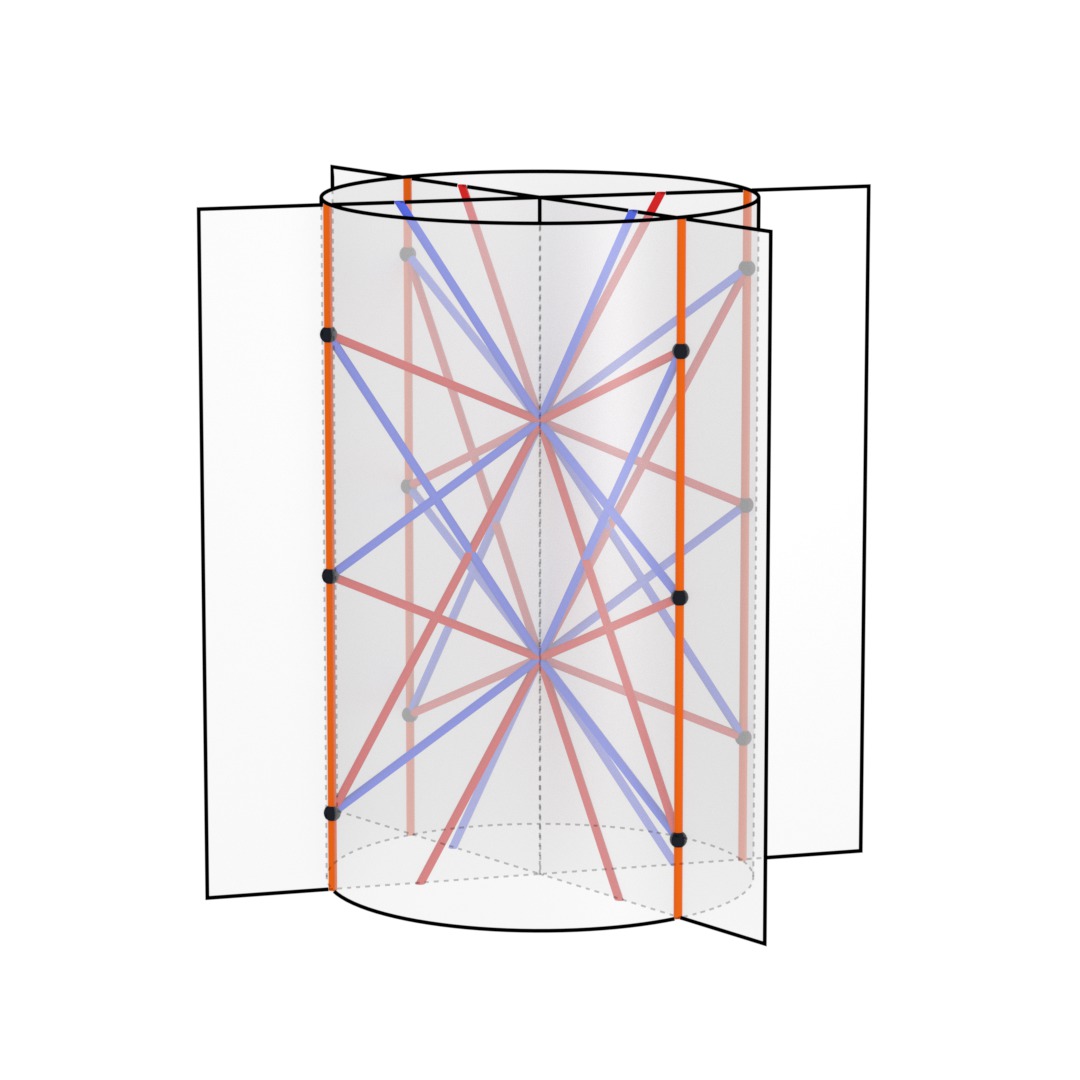}}
  \raisebox{-0.5\height}{\includegraphics[scale=0.13]{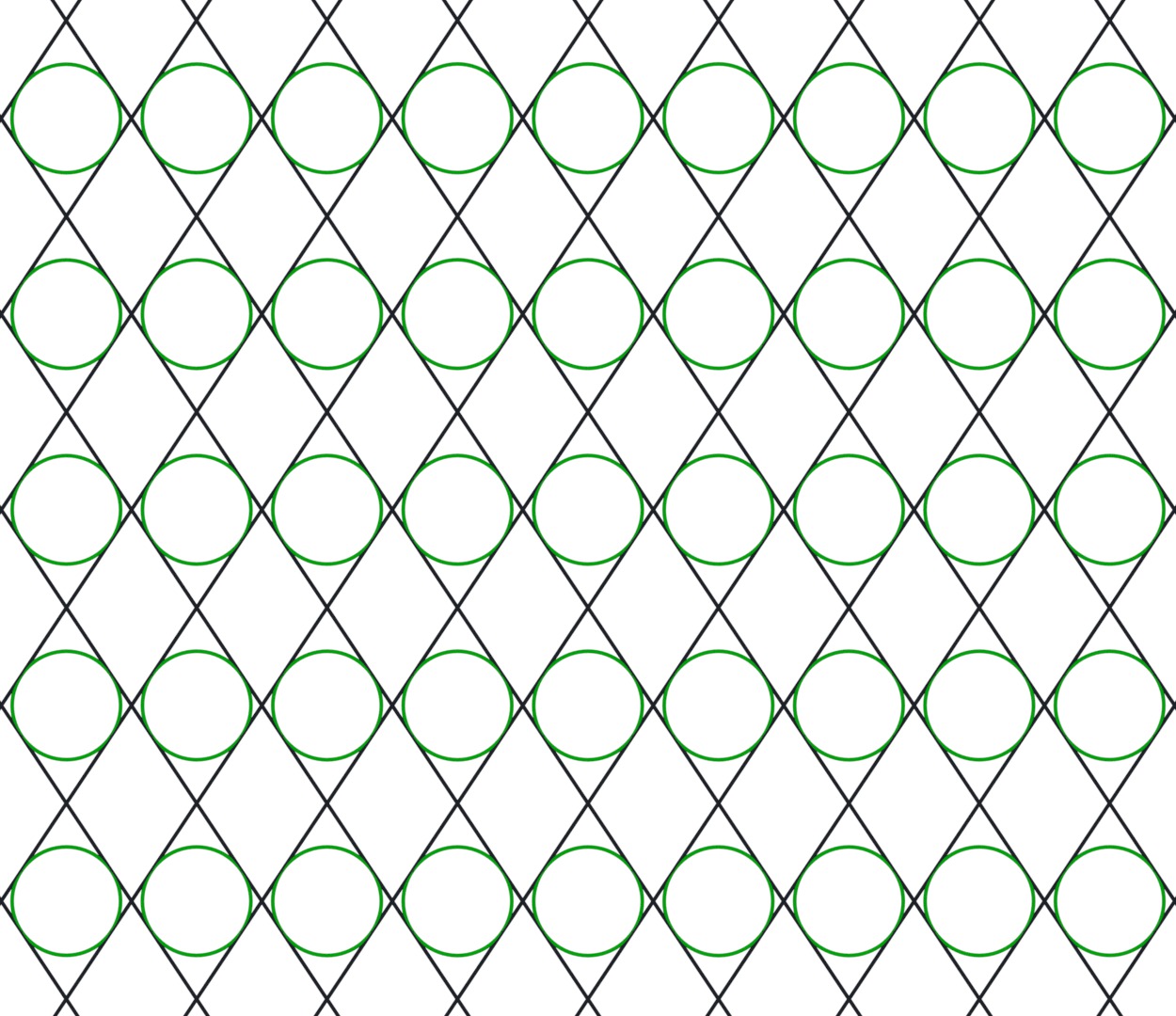}}
  \raisebox{-0.5\height}{\includegraphics[scale=0.06]{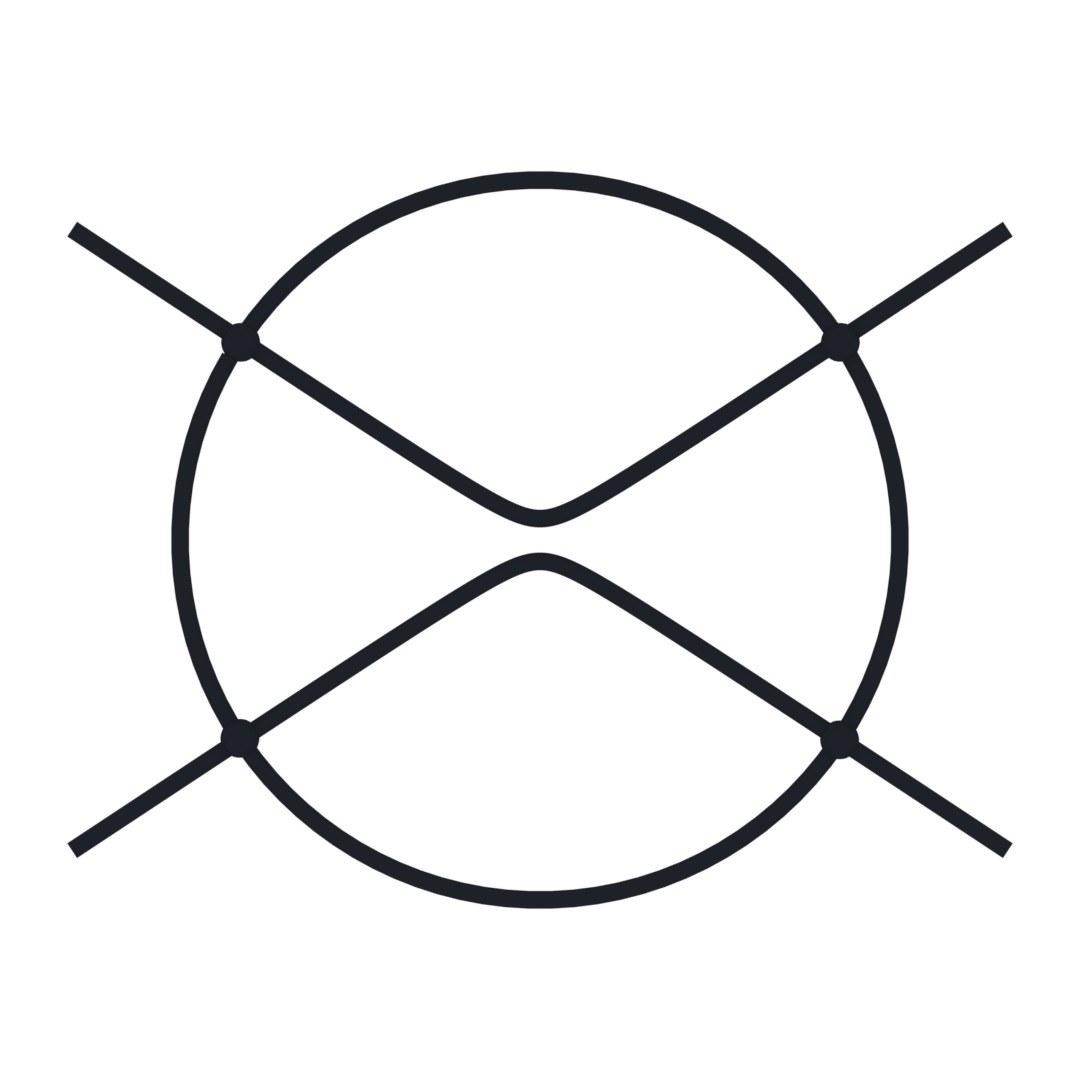}}
  \raisebox{-0.5\height}{\includegraphics[scale=0.17]{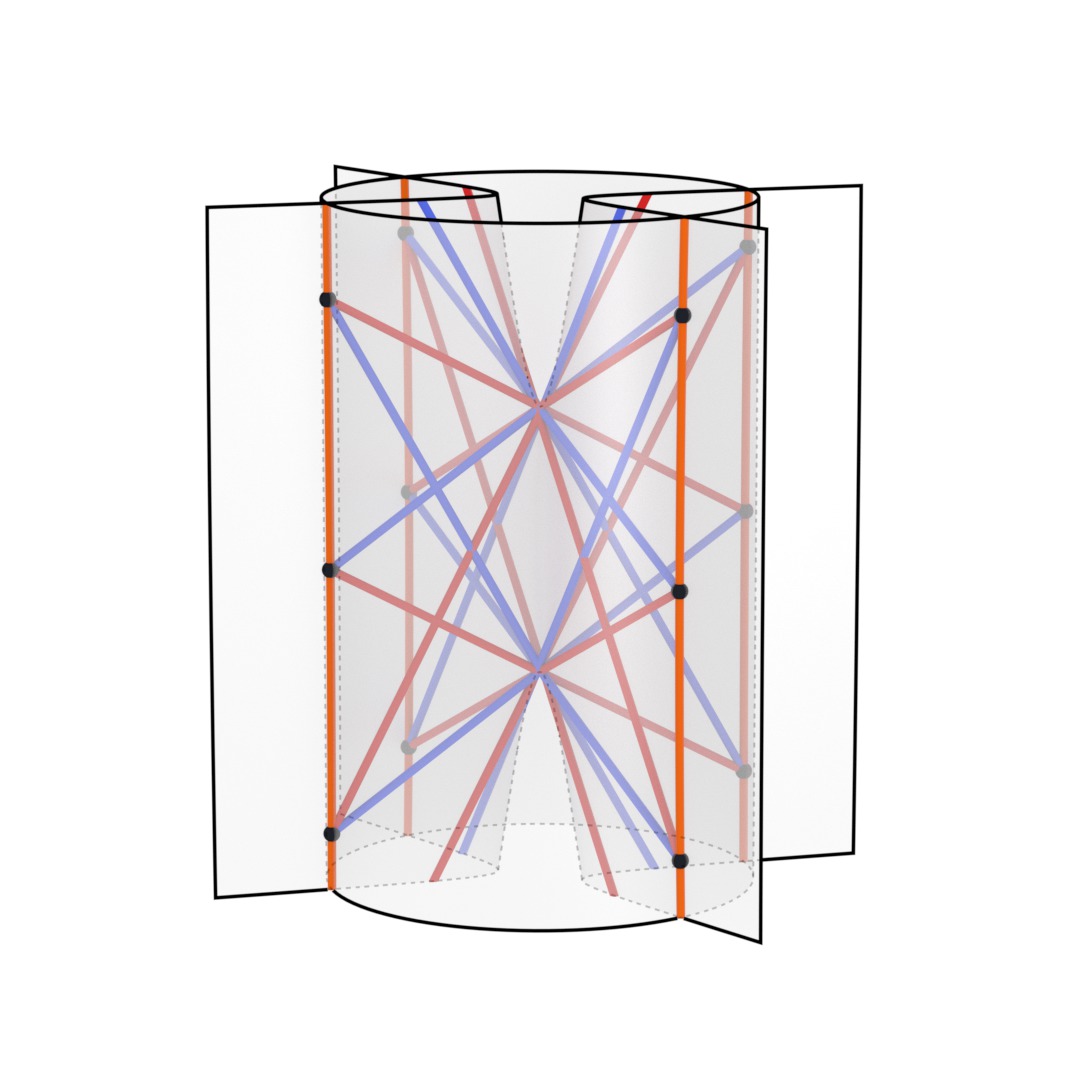}}
  \raisebox{-0.5\height}{\includegraphics[scale=0.13]{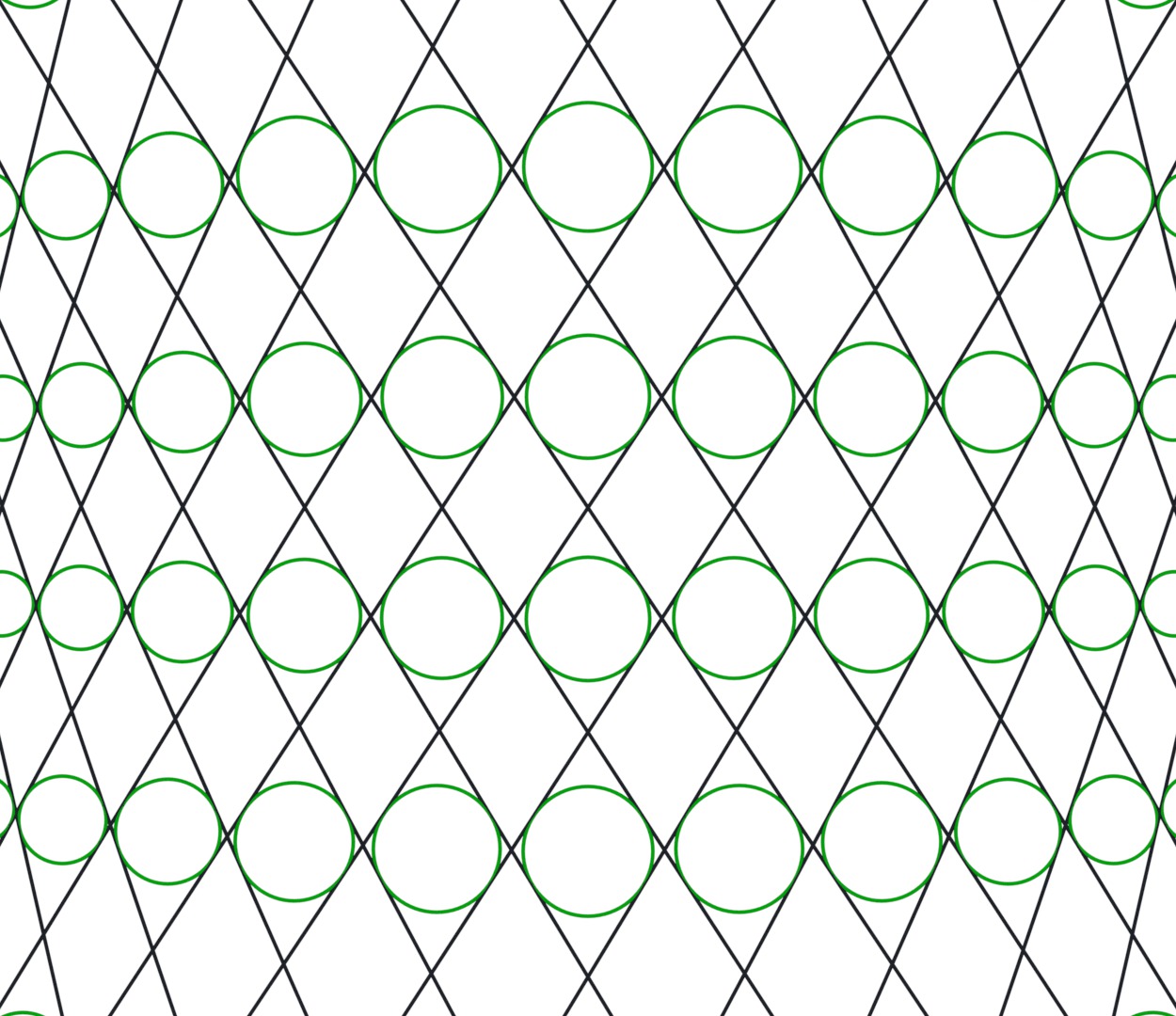}}
  \raisebox{-0.5\height}{\includegraphics[scale=0.06]{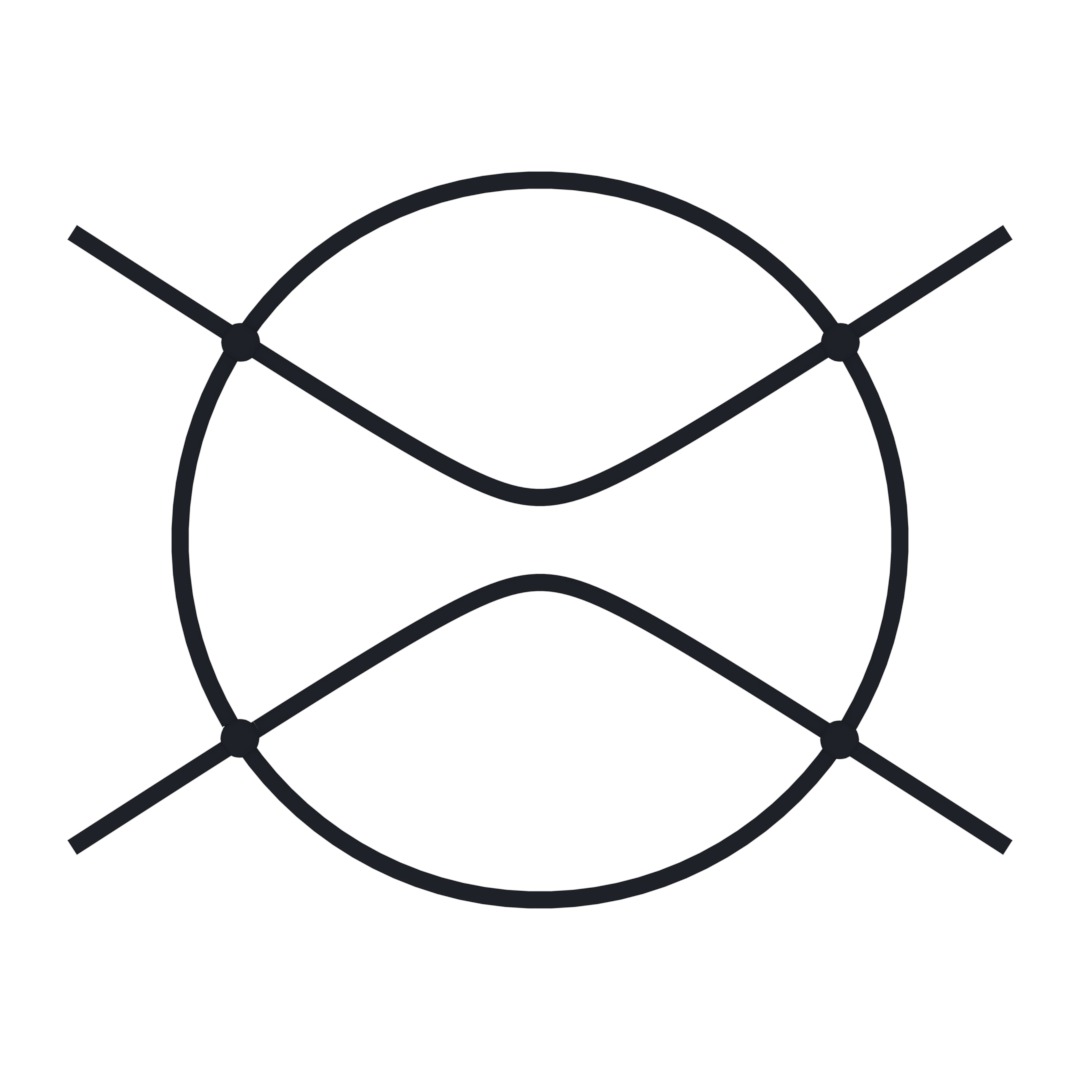}}
  \raisebox{-0.5\height}{\includegraphics[scale=0.17]{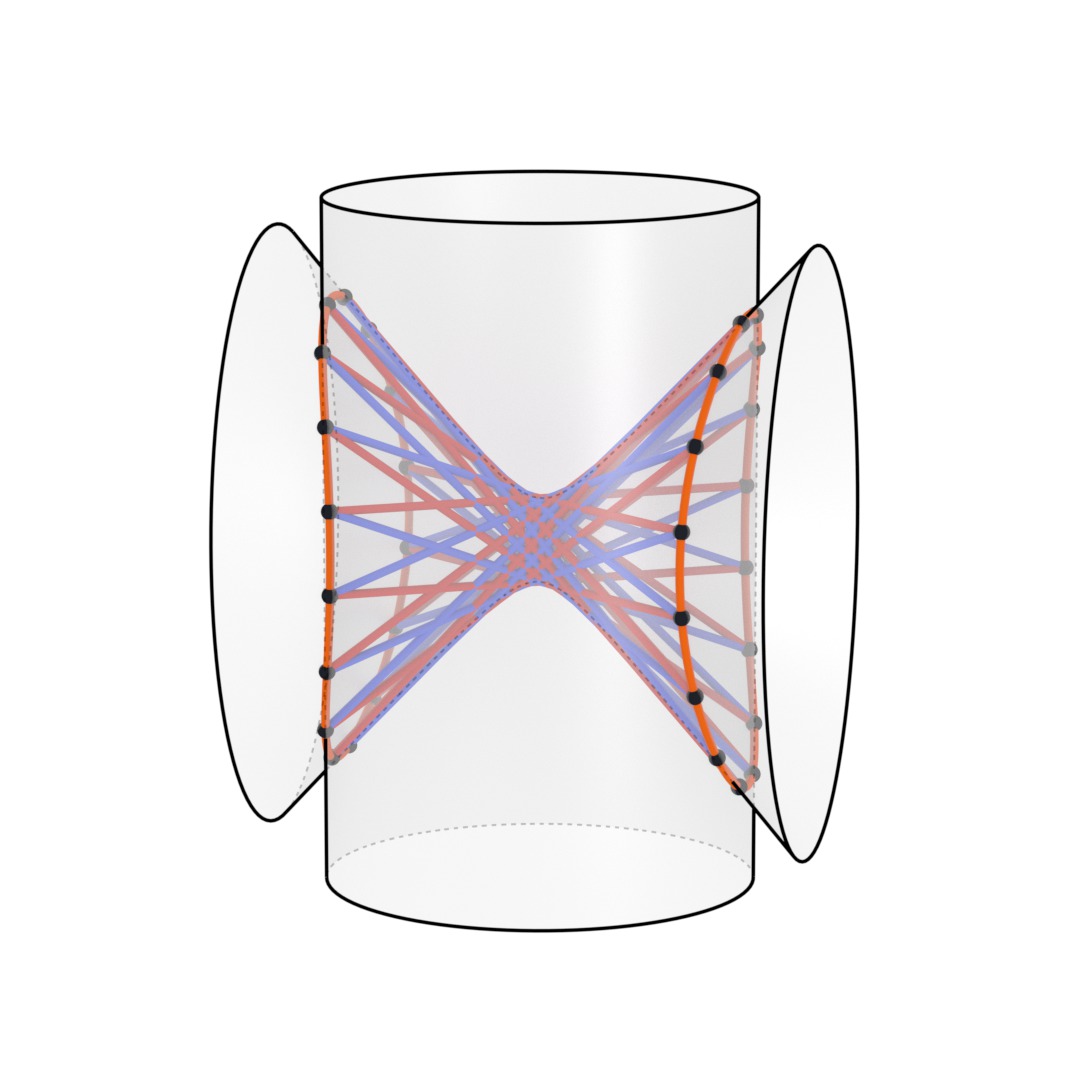}}
  \raisebox{-0.5\height}{\includegraphics[scale=0.13]{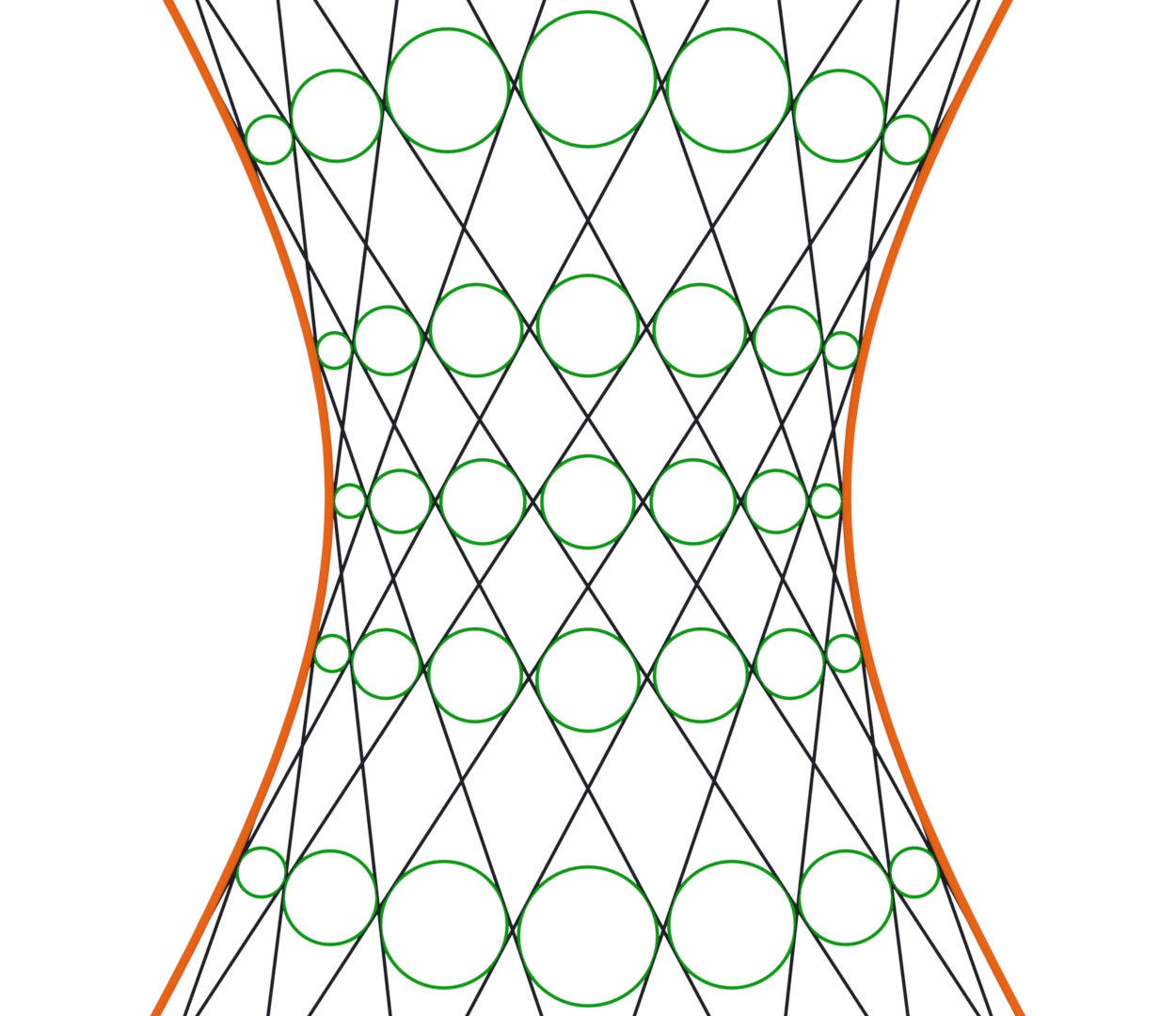}}
  \caption{Deformation of a rhombic checkerboard IC-net.
    \emph{Top to bottom:} A rhombic checkerboard IC-net.
    A slightly deformed ``almost rhombic'' checkerboard IC-net with equal hyperboloids $\mathcal{H} = \tilde{\mathcal{H}}$.
    A hyperbolic checkerboard IC-net  from a larger deformation with equal hyperboloids $\mathcal{H} = \tilde{\mathcal{H}}$.
    \emph{Left to right:} Conics in the $d=0$ plane (The corresponding pencil has 4 base points).
    The checkerboard IC-net in the Blaschke cylinder model.
    The checkerboard IC-net in the plane.
  }
  \label{f.equal_parallelogram}
\end{figure}

The next step (see Figure~\ref{f.equal_parallelogram}) is to consider a one-parameter family of confocal checkerboard IC-nets $\mathcal{N}_c(\epsilon)$ with underlying hyperboloids ${\cal H}_c(\epsilon)$ given by
\bela{C1}
  w_0^2v^2 - v_0^2w^2 = \epsilon^2(\Delta^2 - d^2),\qquad \epsilon\geq0,\quad\Delta>0
\ela
with $v_0^2 + w_0^2=1$, $v_0w_0\neq0$ and
\bela{C1a}
  w_0^2 - \epsilon^2\Delta^2 > 0
\ela
for hyperbolicity of $\mathcal{N}(\epsilon)$. In the limit $\epsilon\rightarrow0$, the hyperboloid $\mathcal{H}_c(0)$ degenerates to the union of the two planes
\bela{C2}
  \mathcal{P}_{\pm}:\quad w_0v \pm v_0w = 0
\ela
passing through the axis of the Blaschke cylinder $\mathcal{Z}$,
equipped with two special points on the axis given by
\bela{C3}
  P_{\pm} = (0,0,\pm\Delta).
\ela
In fact, it is easy to show that the generators of the hyperboloid $\mathcal{H}_c(\epsilon)$ become straight lines passing through either $P_+$ or $P_-$ as $\epsilon\rightarrow 0$. Specifically, all lines
$L_n$ associated with the degenerate checkerboard IC-net $\mathcal{N}_c(0)$ in the Blaschke cylinder model lie in the plane $\mathcal{P}_-$, whereby all $L_{2k}$ pass through $P_-$, while all $L_{2k+1}$ pass through $P_+$. Similarly, all lines $M_n$ lie in the plane ${\cal P}_+$ and all $M_{2k}$ pass through $P_-$, while all $M_{2k+1}$ pass through $P_+$. Moreover, without loss of generality, the lines of $\mathcal{N}_c(0)$ are given by $\ell_{2n}=(v_0,w_0,4n\Delta)$, $\ell_{2n+1}=(-v_0,-w_0,-2(2n+1)\Delta)$, $m_{2n}=(v_0,-w_0,4n\Delta)$, $m_{2n+1}=(-v_0,w_0,-2(2n+1)\Delta)$ so that $\mathcal{N}_c(0)$ is indeed of rhombic type.

It is observed that rhombic checkerboard IC-nets are non-generic since the corresponding hypercycle base curve $\cal C$ consists of four straight lines parallel to the axis of $\cal Z$. The simplest generic checkerboard IC-nets are obtained by ``switching on'' the parameter $\epsilon$. The conics obtained by intersecting the quadric $\mathcal{H}_c(\epsilon)$ with the $d=0$ plane for some ``small'' $\epsilon$ are displayed in Figure~\ref{f.equal_parallelogram} (left).
It is important to note that the number of real base points of the pencil of conics in the $d=0$ plane associated with the checkerboard IC-nets $\mathcal{N}_c(\epsilon)$, that is, the number of points common to the conics of the pencil in the $d=0$ plane is 4. This distinguishes hyperbolic confocal checkerboard IC-nets from elliptic confocal checkerboard IC-nets as discussed in detail in Section \ref{s.confocal}.


\section[Checkerboard IC-nets as Laguerre transforms of confocal checkerboard IC-nets]{Checkerboard IC-nets as Laguerre transforms of confocal\\ checkerboard IC-nets}

We will now examine under what circumstances checkerboard IC-nets may be regarded as Laguerre transforms of confocal checkerboard IC-nets. The associated analysis may naturally be split into two parts.  

\subsection{Pre-normalisation}
In homogeneous coordinates $\sim(v,w,1,d)$, the quadratic forms of the pencil of quadrics associated with a confocal checkerboard IC-net with normalised conics \eqref{eq.conic_normalized} are diagonal. Indeed, the cone \eqref{eq.cone_confocal} and the Blaschke cylinder $v^2+w^2=1$ generate the whole pencil. It is easy to see that the converse statement is also true if we include in our definition of confocal checkerboard IC-nets the case of the lines being ``tangent'' to a degenerate ellipse or hyperbola corresponding to $a>0$, $b\rightarrow0$ or $b>0$, $a\rightarrow0$. In this case, all lines pass through the two focal points. The cases $a<0$, $b\rightarrow0$ or $b<0$, $a\rightarrow0$ may be excluded since the hypercycle base curve consists of only two points corresponding to two lines which only differ in their orientation and which may not be used to construct a proper checkerboard IC-net.
\begin{theorem}\label{prop:confocal}
A generic checkerboard IC-net is confocal with its lines being tangent to a normalised conic of the form \eqref{eq.conic_normalized} if and only if the quadratic forms of the associated pencil of quadrics are diagonal in homogeneous coordinates $\sim(v,w,1,d)$.
\end{theorem}

The above theorem implies that a pencil of quadrics associated with a general (that is, non-normalised) confocal checkerboard IC-net is diagonalisable by Laguerre transformations of the form \eqref{eq.Euclidean_Laguerre} (Appendix) since Euclidean motions constitute particular Laguerre transformations. On the other hand, since the lines of confocal checkerboard IC-nets are tangent to two copies of the same conic which have two different orientations, a Laguerre transformation separates those two copies and one obtains a hypercycle which consists of two possibly intersecting pieces (see Figure~\ref{f.hypercycle}). It is therefore natural to investigate the question of the diagonalisability of generic pencils of quadrics which are not necessarily associated with confocal checkerboard IC-nets. Thus, we now consider a generic checkerboard IC-net with $\cal H$ being one of its corresponding generic hyperboloids so that the diagonalisability of the associated pencil of quadrics is equivalent to the diagonalisability of $\cal H$. Let ${(\tilde{Q}_{i,j})}_{ i,j=1,\ldots,4}$ be the symmetric matrix of its quadratic form in coordinates $(v,w,1,d)$. Our genericity assumption is equivalent to $\tilde{Q}_{4,4}\neq 0$. The normalisation $\tilde{Q}_{4,4}=-1$ therefore leads to
$$
\tilde{Q}=\left(\begin{matrix} \tilde{S} & a  \\[1mm]
                       a^T & -1 
  \end{matrix}\right),
$$ 
where $\tilde{S}$ is a symmetric matrix, and $a\in\R^3$. By means of a Laguerre transformation \eqref{eq.Laguerre_trafo_4x4} (Appendix) with the matrix
$$
A=\left(\begin{matrix} 1 & 0  \\[1mm]
                       a^T & 1 
  \end{matrix}\right),
$$
$\tilde{Q}$ is brought into the block diagonal form
\begin{equation}
\label{eq.Q-matrix}
Q=A^T\tilde{Q}A=\left(\begin{matrix} S & 0  \\[1mm]
                       0 & -1 
  \end{matrix}\right),
\end{equation}
where $S$ is symmetric. Furthermore, the matrix $Q$ may be diagonalised by a Laguerre transformation $A'$ only if it is of the form 
$$
A'=\left(\begin{matrix} B & 0  \\[1mm]
                       0 & 1 
  \end{matrix}\right)
$$ 
with $B\in O(2,1)$. Hence, we conclude that the pencil of quadrics associated with a generic checkerboard IC-net is diagonalisable by a Laguerre transformation if and only if there exists a $B\in O(2,1)$ such that 
$$
B^TSB
$$   
is diagonal, where $S$ is defined by \eqref{eq.Q-matrix}.

\subsection{Diagonalisation}
\label{s.diagonalisation}

In the preceding, it has been demonstrated that the question of whether or not a checkerboard IC-net may be Laguerre-transformed into a confocal checkerboard IC-net may be answered by determining the class of symmetric $3\times3$ matrices $S$ which may be diagonalised according to
\bela{E4.1}
  S \rightarrow B^TSB,\quad B\in O(2,1).
\ela
It is important to note that $S$ represents the conic intersection of the quadric
\bela{E4.1a}
  (v\,\,w\,\,1)S\left(\bear{c}v\\w\\1\ear\right) = d^2
\ela
and the plane $d=0$. Since the matrix $Z=\mbox{diag}(1,1,-1)$ representing the ``Blaschke circle'' $v^2 + w^2 = 1$ is invariant under the action of the group $O(2,1)$, the matrix $S$ is diagonalisable if and only if the one-parameter family of matrices
\bela{S1}
  S^{\lambda} = S + \lambda Z
\ela
encoding the pencil of conics $\mathcal{P}$ spanned by the conic associated with $S$ and the Blaschke circle is diagonalisable. It is noted that the roots of the characteristic cubic polynomial
\bela{S2}
  P(\lambda)= \det S^{\lambda}
\ela
correspond to the degenerate conics of $\mathcal{P}$. Furthermore, it is evident that $S$ is diagonal if and only if one conic in $\mathcal{P}$ different from the Blaschke circle (and therefore any conic) is symmetric with respect to the $v$- and $w$-axes.

We will now demonstrate that diagonalisability may be characterised in terms of the real base points of the pencil $\mathcal{P}$, that is, the points (on the Blaschke circle) common to all conics of $\mathcal{P}$ (or, equivalently, any particular pair of conics of $\mathcal{P}$). The proof of the following theorem is based on the standard classification of pencils of conics \cite{levy1964} in terms of the number and nature of the base points which are determined by the roots of the quartic equation, representing the common solutions of $v^2 + w^2=1$ and \eqref{E4.1a}$_{d=0}$. This classification is in one-to-one correspondence with the number and nature of roots of the characteristic polynomial $P(\lambda)$ associated with any pencil. The present theorem is a variant of a theorem \cite{levy1964} which states that a symmetric matrix $S$ is diagonalisable by means of a {\em projective} transformation if and only if the corresponding pencil is of type Ia, Ic, IIIa or IIIb in the classification tabled in Figure \ref{classification}. Normal forms of degenerate conics associated with this classification are depicted in Figure \ref{classification-degenerate-conics}.

\begin{figure}
\begin{center}
    \begin{tabular}{ | p{0.8cm} | p{1.9cm} | p{1cm} | p{2cm} |  p{1.5cm} |}
    \hline
    Type & Type and multiplicity of base points & \# of real base points & Type and multiplicity of degenerate conics & Type and multiplicity of roots\\ \hline\hline
      Ia &$ 1,1,1,1$ & $4$ &$\mbox{\boldmath$\times$},\mbox{\boldmath$\times$},\mbox{\boldmath$\times$}$ & $1,1,1$\\ 
      Ib &$ 1,1, (1,\bar{1})$ & $2$ &$\mbox{\boldmath$\times$},\circ,\bar{\circ}$ & $1,(1,\bar{1})$\\ 
      Ic &$ (1,\bar{1}),(1,\bar{1})$ & $0$ & $\mbox{\boldmath$\times$},\bullet,\bullet$ & $1,1,1$\\ \hline
     IIa &$ 2,1,1$ & $3$ &$2\mbox{\boldmath$\times$},\mbox{\boldmath$\times$}$ & $2,1$\\ 
     IIb &$ 2,(1,\bar{1})$ & $1$ & $2\bullet,\mbox{\boldmath$\times$}$ & $2,1$\\ \hline
     IIIa &$ 2,2$ & $2$ &$2\,\rotatebox[origin=t]{90}{\mbox{\boldmath$=$}},\mbox{\boldmath$\times$}$ & $2,1$\\ 
     IIIb &$ (2,\bar{2})$ & $0$ & $2\,\rotatebox[origin=t]{90}{\mbox{\boldmath$=$}},\bullet$ & $2,1$\\ \hline
     IV &$ 3,1$ & $2$ &$3\mbox{\boldmath$\times$}$& $3$\\ \hline
     V &$ 4$ & $1$ &$3\,\rotatebox[origin=t]{90}{\mbox{\boldmath$=$}}$ & $3$\\ \hline
    \end{tabular}
\end{center}
\caption{The classification of real pencils of conics. There exist four different types of degenerate conics. (\mbox{\boldmath$\times$}) Two real intersecting lines. ($\circ$) Two non-intersecting complex lines. ($\bullet$) Two complex conjugate lines which intersect in a real point. 
(\mbox{\boldmath$\|$}) A real double line.}
\label{classification}
\end{figure}

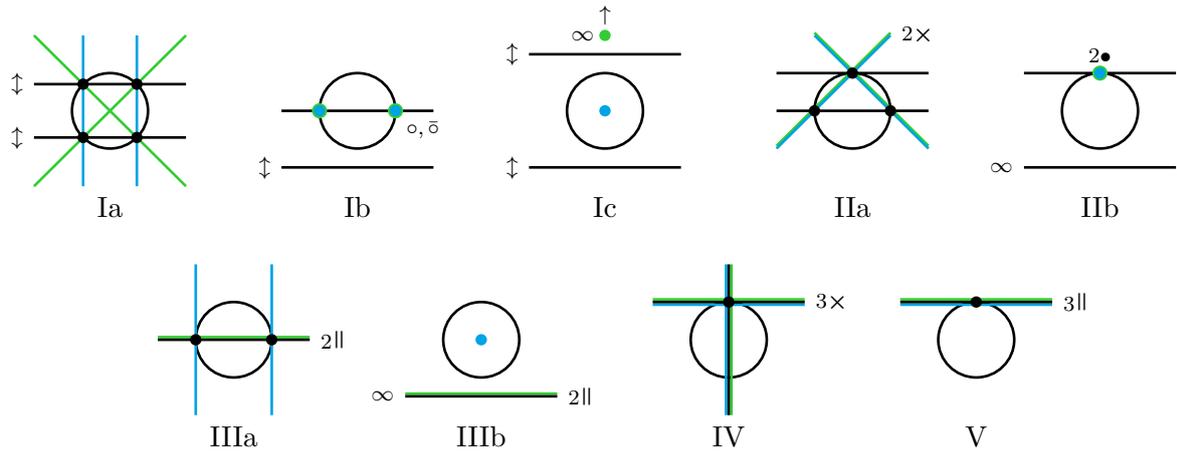
\begin{figure}
  \centering
  \begin{tikzpicture}[scale=0.5, line width=1pt]
    \useasboundingbox (-2.5,-2.5) rectangle (2.5,2.5);
    \draw (0,0) circle (1);
    \draw (-2,{1/sqrt(2)}) -- (2,{1/sqrt(2)});
    \coordinate[label=180:{\scriptsize $\updownarrow$}] (none) at (-2,{1/sqrt(2)});
    \draw (-2,{-1/sqrt(2)}) -- (2,{-1/sqrt(2)});
    \coordinate[label=180:{\scriptsize $\updownarrow$}] (none) at (-2,{-1/sqrt(2)});
    \draw [color=Cerulean] ({-1/sqrt(2)},-2) -- ({-1/sqrt(2)},2);
    \draw [color=Cerulean] ({1/sqrt(2)},-2) -- ({1/sqrt(2)},2);
    \draw [color=LimeGreen] (-2,-2) -- (2,2);
    \draw [color=LimeGreen] (2,-2) -- (-2,2);
    \fill ({1/sqrt(2)},{1/sqrt(2)}) circle (0.15);
    \fill (-{1/sqrt(2)},{1/sqrt(2)}) circle (0.15);
    \fill ({1/sqrt(2)},-{1/sqrt(2)}) circle (0.15);
    \fill (-{1/sqrt(2)},-{1/sqrt(2)}) circle (0.15);
    \coordinate[label=-90:{Ia}] (none) at (0,-2);
  \end{tikzpicture}
  \hspace{0.5cm}
  \begin{tikzpicture}[scale=0.5, line width=1pt]
    \useasboundingbox (-2.5,-2.5) rectangle (2.5,2.5);
    \draw (0,0) circle (1);
    \draw (-2,0) -- (2,0);
    \draw (-2,-1.5) -- (2,-1.5);
    \coordinate[label=180:{\scriptsize $\updownarrow$}] (none) at (-2,-1.5);
    \fill [color=LimeGreen] (-1,0) circle (0.2);
    \fill [color=Cerulean] (-1,0) circle (0.15);
    \fill [color=LimeGreen] (1,0) circle (0.2);
    \fill [color=Cerulean] (1,0) circle (0.15);
    \coordinate[label=-45:{\scriptsize $\circ,\bar{\circ}$}] (none) at (1,0);
    \coordinate[label=-90:{Ib}] (none) at (0,-2);
  \end{tikzpicture}
  \hspace{0.5cm}
  \begin{tikzpicture}[scale=0.5, line width=1pt]
    \useasboundingbox (-2.5,-2.5) rectangle (2.5,2.5);
    \draw (0,0) circle (1);
    \draw (-2,1.5) -- (2,1.5);
    \coordinate[label=180:{\scriptsize $\updownarrow$}] (none) at (-2,1.5);
    \draw (-2,-1.5) -- (2,-1.5);
    \coordinate[label=180:{\scriptsize $\updownarrow$}] (none) at (-2,-1.5);
    \fill [color=Cerulean] (0,0) circle (0.15);
    \fill [color=LimeGreen] (0,2) circle (0.15);
    \coordinate[label=180:{\scriptsize $\infty$}] (none) at (0,2);
    \coordinate[label=90:{\scriptsize $\uparrow$}] (none) at (0,2);
    \coordinate[label=-90:{Ic}] (none) at (0,-2);
  \end{tikzpicture}
  \hspace{0.5cm}
  \begin{tikzpicture}[scale=0.5, line width=1pt]
    \useasboundingbox (-2.5,-2.5) rectangle (2.5,2.5);
    \draw (0,0) circle (1);
    \draw (-2,1) -- (2,1);
    \draw (-2,0) -- (2,0);
    \draw [color=LimeGreen, transform canvas={yshift=1pt}] (-2,-1) -- (1,2);
    \draw [color=LimeGreen, transform canvas={yshift=1pt}] (2,-1) -- (-1,2);
    \draw [color=Cerulean] (-2,-1) -- (1,2);
    \draw [color=Cerulean] (2,-1) -- (-1,2);
    \coordinate[label=0:{\scriptsize $2\mbox{\boldmath$\times$}$}] (none) at (1,2);
    \fill (-1,0) circle (0.15);
    \fill (1,0) circle (0.15);
    \fill (0,1) circle (0.15);
    \coordinate[label=-90:{IIa}] (none) at (0,-2);
  \end{tikzpicture}
  \hspace{0.5cm}
  \begin{tikzpicture}[scale=0.5, line width=1pt]
    \useasboundingbox (-2.5,-2.5) rectangle (2.5,2.5);
    \draw (0,0) circle (1);
    \draw (-2,1) -- (2,1);
    \draw (-2,-1.5) -- (2,-1.5);
    \coordinate[label=180:{\scriptsize $\infty$}] (none) at (-2,-1.5);
    \fill [color=LimeGreen] (0,1) circle (0.2);
    \fill [color=Cerulean] (0,1) circle (0.15);
    \coordinate[label=90:{\scriptsize $2\bullet$}] (none) at (0,1);
    \coordinate[label=-90:{IIb}] (none) at (0,-2);
  \end{tikzpicture}
  
  \vspace{0.5cm}
  
  \begin{tikzpicture}[scale=0.5, line width=1pt]
    \useasboundingbox (-2.5,-2.5) rectangle (2.5,2.5);
    \draw (0,0) circle (1);
    \draw [color=LimeGreen, transform canvas={yshift=1pt}] (-2,0) -- (2,0);
    \draw (-2,0) -- (2,0);
    \coordinate[label=0:{\scriptsize $2\,\rotatebox[origin=t]{90}{\mbox{\boldmath$=$}}$}] (none) at (2,0);
    \draw [color=Cerulean] (-1,-2) -- (-1,2);
    \draw [color=Cerulean] (1,-2) -- (1,2);
    \fill (-1,0) circle (0.15);
    \fill (1,0) circle (0.15);
    \coordinate[label=-90:{IIIa}] (none) at (0,-2);
  \end{tikzpicture}
  \hspace{0.5cm}
  \begin{tikzpicture}[scale=0.5, line width=1pt]
    \useasboundingbox (-2.5,-2.5) rectangle (2.5,2.5);
    \draw (0,0) circle (1);
    \draw [color=LimeGreen, transform canvas={yshift=1pt}] (-2,-1.5) -- (2,-1.5);
    \draw (-2,-1.5) -- (2,-1.5);
    \coordinate[label=180:{\scriptsize $\infty$}] (none) at (-2,-1.5);
    \coordinate[label=0:{\scriptsize $2\,\rotatebox[origin=t]{90}{\mbox{\boldmath$=$}}$}] (none) at (2,-1.5);
    \fill [color=Cerulean] (0,0) circle (0.15);
    \coordinate[label=-90:{IIIb}] (none) at (0,-2);
  \end{tikzpicture}
  \hspace{0.5cm}
  \begin{tikzpicture}[scale=0.5, line width=1pt]
    \useasboundingbox (-2.5,-2.5) rectangle (2.5,2.5);
    \draw (0,0) circle (1);
    \draw [color=LimeGreen, transform canvas={yshift=1pt}] (-2,1) -- (2,1);
    \draw [color=Cerulean, transform canvas={yshift=-1pt}] (-2,1) -- (2,1);
    \draw (-2,1) -- (2,1);
    \coordinate[label=0:{\scriptsize $3\mbox{\boldmath$\times$}$}] (none) at (2,1);
    \draw [color=LimeGreen, transform canvas={xshift=1pt}] (0,-2) -- (0,2);
    \draw [color=Cerulean, transform canvas={xshift=-1pt}] (0,-2) -- (0,2);
    \draw (0,-2) -- (0,2);
    \fill (0,1) circle (0.15);
    \coordinate[label=-90:{IV}] (none) at (0,-2);
  \end{tikzpicture}
  \hspace{0.5cm}
  \begin{tikzpicture}[scale=0.5, line width=1pt]
    \useasboundingbox (-2.5,-2.5) rectangle (2.5,2.5);
    \draw (0,0) circle (1);
    \draw [color=LimeGreen, transform canvas={yshift=1pt}] (-2,1) -- (2,1);
    \draw [color=Cerulean, transform canvas={yshift=-1pt}] (-2,1) -- (2,1);
    \draw (-2,1) -- (2,1);
    \coordinate[label=0:{\scriptsize $3\,\rotatebox[origin=t]{90}{\mbox{\boldmath$=$}}$}] (none) at (2,1);
    \fill (0,1) circle (0.15);
    \coordinate[label=-90:{V}] (none) at (0,-2);
  \end{tikzpicture}
  \caption{
    Degenerate conics of real pencils of conics containing the Blaschke circle.
    For each type, a normal form is shown with respect to projective transformations which preserve the Blaschke circle.
    For types I, the normal form still depends on one parameter, which is indicated by $\updownarrow$.
    The $\infty$ symbol indicates that the line (point) is chosen to be the line at infinity (the point at infinity in the designated direction).
  }
\label{classification-degenerate-conics}
\end{figure}

\begin{theorem}
The matrix $S$ may be diagonalised if and only if the associated pencil of conics $\mathcal{P}$ has four, two double or no real base points, that is, if $\mathcal{P}$ is of type Ia, Ic, IIIa or IIIb.
\end{theorem}

\begin{proof}
If $S$ is diagonal then the associated pencil $\mathcal{P}$ is symmetric with respect to the $v$- and $w$-axes. This symmetry cannot be present in types other than Ia, Ic, IIIa or IIIb. Conversely, it is necessary to show that if the pencil is of any of those four types then $\mathcal{P}$ can be made symmetric or, equivalently, $S^{\lambda}$ is diagonalisable for one $\lambda\neq\infty$.
\medskip

\noindent
{\sc Ia) Four real base points.} In this case, we may apply a projective transformation which transforms the Blaschke circle into an ellipse and maps the four base points to the four vertices of a rectangle which is symmetric with respect to the $v$- and $w$-axes. An appropriate subsequent affine transformation then maps the ellipse to the Blaschke circle without affecting the symmetry of the rectangle. The composition of these two transformations constitutes an $O(2,1)$ transformations since it leaves the Blaschke circle invariant. This compound transformation results in a symmetric distribution of the base points and, hence, the transformed pencil is symmetric.  
\medskip

\noindent 
{\sc IIIa) Two real double base points.} In this case, there exists a degenerate conic consisting of two real lines which touch the Blaschke circle. A suitable combination of a projective and an affine transformation sends the vertex of this degenerate conic to infinity and maps the intermediate ellipse back to the Blaschke circle. The degenerate conic may therefore be transformed into $w^2=1$ which is symmetric.
\medskip

\noindent
{\sc Ic) Four complex base points.} In this case, there exists a degenerate conic consisting of two intersecting real lines which do not intersect the Blaschke circle. A suitable combination of a projective and an affine transformation sends the vertex of this degenerate conic to infinity and maps the intermediate ellipse back to the Blaschke circle. The degenerate conic may therefore be transformed into $(w-a)(w+b)=0$, where $a,b>1$. It is not difficult to show that a suitable hyperbolic rotation in the $(\tilde{w},\tilde{z})$-plane (with $(\tilde{w},\tilde{z})\sim(w,1)$) leads to $a=b$ so that the degenerate conic simplifies to $w^2=a^2$ which is symmetric.
\medskip

\noindent
{\sc IIIb) Two complex double base points.} In this case, there exists a degenerate conic consisting of two coinciding real lines which do not intersect the Blaschke circle. A suitable combination of a projective and an affine transformation sends this double line to infinity and maps the intermediate ellipse back to the Blaschke circle. The transformed double line is therefore given by $\tilde{z}^2=0$ which is symmetric.
\end{proof}

In summary, it has been established that the matrix $S$ may be diagonalised if and only if the associated pencil of conics $\mathcal{P}$ has four, two double or no real base points. This leads to the following characterisation.

\begin{theorem}
A generic checkerboard IC-net is Laguerre-equivalent to a confocal checkerboard IC-net if and only if the hypercycle base curve of the associated pencil of quadrics consists of two non-degenerate loops on the Blaschke cylinder which are either disjoint or transversal.
\end{theorem}

\begin{proof}
Since the statement to be proven is invariant under Laguerre transformations, it suffices to examine the nature of the base curves of ``normal forms'' of generic pencils of quadrics. These are determined by the normal forms of the associated pencils of conics. The normal forms of degenerate conics which together with the Blaschke circle span pencils of conics of types Ia, Ic and  IIIa, IIIb
have been derived in the proof of the preceding theorem. In a similar manner, the remaining types of pencils may be treated. Accordingly, one obtains the classification of pencils of quadrics displayed in Figure \ref{normal_forms}. The associated base curves are also depicted in Figure \ref{normal_forms}. It is noted that the base curves of types Ic$_+$ and IIIb$_-$ are empty and, therefore, do not correspond to a hypercycle. The types Ia and Ic$_-$ are associated with a double hyperbola and a double ellipse respectively, while the types IIIa$_+$, IIIa$_-$ and IIIb$_+$ correspond to the special cases of the hypercycle consisting of two points, a double line and a double circle respectively. As pointed out in connection with Theorem \ref{prop:confocal}, the case of a double line does not give rise to a proper checkerboard IC-net. Accordingly, we find that confocal checkerboard IC-nets are captured by the types Ia, Ic$_-$, IIIa$_+$ and IIIb$_+$ which confirms the assertion of the theorem.
\end{proof}

\begin{remark}
A stronger notion of genericity is obtained by considering only those checkerboard IC-nets which are generically generated by means of the iterative geometric construction of checkerboard IC-nets described in Section~4. In this case, generic checkerboard IC-nets are of types Ia, Ib$_{\pm}$ and Ic$_-$ so that the case of two transversal components of the hypercycle base curve cannot occur and must be removed from the above theorem.
\end{remark}

\newpage

\def\picwidth{0.12\textwidth}%
\def\tw{2.1}
\def\hw{1.6cm}
\begin{center}
\begin{longtable}[b]{ | m{0.8cm} | m{5.5cm} | l |}
  \hline
  Type & Normal form of a degenerate quadric & Type and multiplicity of degenerate quadrics\\ \hline\hline
  \vspace{0.5cm}&\vspace{0.5cm} & 
  \multirow{2}{*}{
    \begin{tikzpicture}
      \useasboundingbox ({-0.54*\tw},{-0.55*\tw}) rectangle ({3.5*\tw},{0.45*\tw});
      \node[inner sep=0pt] (none) at (0,  0)     {\includegraphics[width=\picwidth]{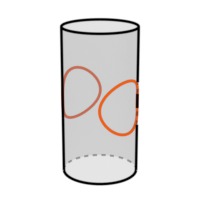}};
      \node[inner sep=0pt] (none) at (\tw,0)     {\includegraphics[width=\picwidth]{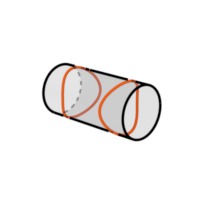}};
      \node[inner sep=0pt] (none) at ({2*\tw},0) {\includegraphics[width=\picwidth]{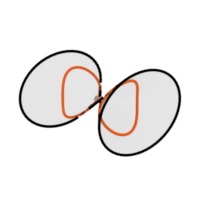}};
      \node[inner sep=0pt] (none) at ({3*\tw},0) {\includegraphics[width=\picwidth]{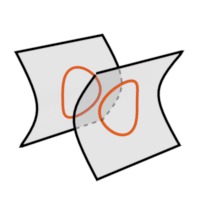}};
      \coordinate[label=-135:{\scriptsize\textcolor{gray}{$\star$}}] (none) at ({3.45*\tw},{0.45*\tw});
    \end{tikzpicture}}\\
  Ia\vspace{0.9cm} & $w^2-a^2= \varpm d^2,\quad 0<a<1$\vspace{0.9cm} & \\
  \hdashline
  \multirow{2}{*}{Ib$_{\pm}$} & \multirow{2}{*}{$w(bw-1)=\pm d^2,\quad |b|<1$} &
  \begin{tikzpicture}
    \useasboundingbox ({-0.6*\tw},{-0.5*\tw}) rectangle ({3.5*\tw},{0.52*\tw});
    \node[inner sep=0pt] (none) at (0,  0)     {\includegraphics[width=\picwidth]{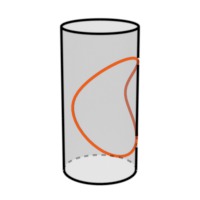}};
    \node[inner sep=0pt] (none) at (\tw,0)     {\includegraphics[width=\picwidth]{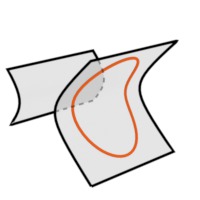}};
    \node[inner sep=0pt] (none) at ({2*\tw},0) {\includegraphics[width=\picwidth]{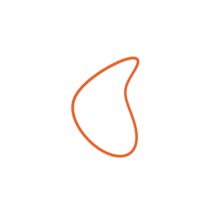}};
    \node[inner sep=0pt] (none) at ({3*\tw},0) {\includegraphics[width=\picwidth]{pics/classification/Ib_plus_2}};
    \draw[thick, <->] ({2.35*\tw},0) -- ({2.65*\tw},0);
    \node[align=center, font=\tiny] at ({2*\tw},{-0.38*\tw}) {complex\\ cone};
    \node[align=center, font=\tiny] at ({3*\tw},{-0.38*\tw}) {complex\\ cone};
    \coordinate[label=90:{\scriptsize c.c.}] (none) at ({2.5*\tw},0);
    \coordinate[label=-135:{\scriptsize\textcolor{gray}{$\star$}}] (none) at ({1.45*\tw},{0.45*\tw});
    \node[font=\scriptsize] at ({-0.4*\tw},0.) {$(+)$};
  \end{tikzpicture}\\
  & &
  \begin{tikzpicture}
    \useasboundingbox ({-0.6*\tw},{-0.51*\tw}) rectangle ({3.5*\tw},{0.5*\tw});
    \node[inner sep=0pt] (none) at (0,0)       {\includegraphics[width=\picwidth]{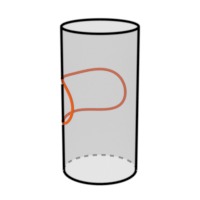}};
    \node[inner sep=0pt] (none) at (\tw,0)     {\includegraphics[width=\picwidth]{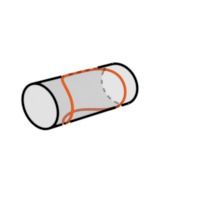}};
    \node[inner sep=0pt] (none) at ({2*\tw},0) {\includegraphics[width=\picwidth]{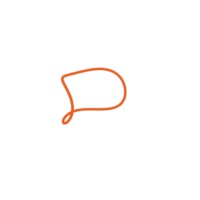}};
    \node[inner sep=0pt] (none) at ({3*\tw},0) {\includegraphics[width=\picwidth]{pics/classification/Ib_minus_2}};
    \draw[thick, <->] ({2.35*\tw},0) -- ({2.65*\tw},0);
    \node[align=center, font=\tiny] at ({2*\tw},{-0.26*\tw}) {complex\\ cone};
    \node[align=center, font=\tiny] at ({3*\tw},{-0.26*\tw}) {complex\\ cone};
    \coordinate[label=90:{\scriptsize c.c.}] (none) at ({2.5*\tw},0);
    \coordinate[label=-135:{\scriptsize\textcolor{gray}{$\star$}}] (none) at ({1.45*\tw},{0.45*\tw});
    \node[font=\scriptsize] at ({-0.4*\tw},0.) {$(-)$};
  \end{tikzpicture}\\
  \hdashline
  \multirow{2}{*}{Ic$_{\pm}$} & \multirow{2}{*}{$w^2-a^2=\pm d^2,\quad a>1$} &
  \begin{tikzpicture}
    \useasboundingbox ({-0.6*\tw},{-0.5*\tw}) rectangle ({3.5*\tw},{0.52*\tw});
    \node[inner sep=0pt] (none) at (0,  0)     {\includegraphics[width=\picwidth]{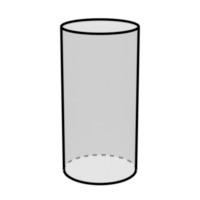}};
    \node[inner sep=0pt] (none) at (\tw,0)     {\includegraphics[width=\picwidth]{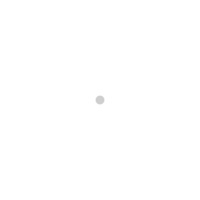}};
    \node[inner sep=0pt] (none) at ({2*\tw},0) {\includegraphics[width=\picwidth]{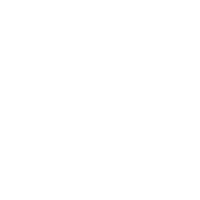}};
    \node[inner sep=0pt] (none) at ({3*\tw},0) {\includegraphics[width=\picwidth]{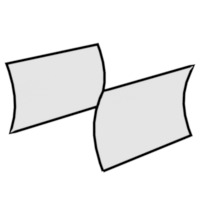}};
    \node[align=center, font=\scriptsize] at (\tw,{-0.2*\tw}) {imag.\\ cone};
    \node[align=center, font=\scriptsize] at ({2*\tw},{-0.2*\tw}) {imag.\\ cylinder};
    \coordinate[label=-135:{\scriptsize\textcolor{gray}{$\star$}}] (none) at ({3.45*\tw},{0.45*\tw});
    \node[font=\scriptsize] at ({-0.4*\tw},0.) {$(+)$};
  \end{tikzpicture}\\
  & &
  \begin{tikzpicture}
    \useasboundingbox ({-0.6*\tw},{-0.51*\tw}) rectangle ({3.5*\tw},{0.5*\tw});
    \node[inner sep=0pt] (none) at (0,0)       {\includegraphics[width=\picwidth]{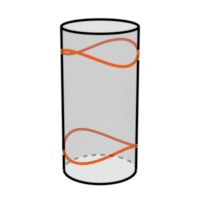}};
    \node[inner sep=0pt] (none) at (\tw,0)     {\includegraphics[width=\picwidth]{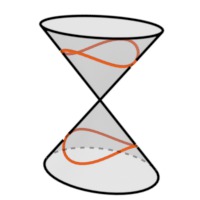}};
    \node[inner sep=0pt] (none) at ({2*\tw},0) {\includegraphics[width=\picwidth]{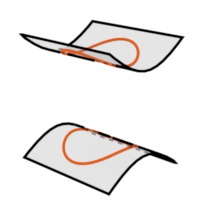}};
    \node[inner sep=0pt] (none) at ({3*\tw},0) {\includegraphics[width=\picwidth]{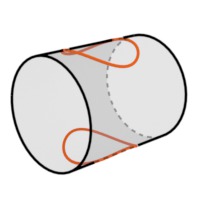}};
    \coordinate[label=-135:{\scriptsize\textcolor{gray}{$\star$}}] (none) at ({3.45*\tw},{0.45*\tw});
    \node[font=\scriptsize] at ({-0.4*\tw},0.) {$(-)$};
  \end{tikzpicture}\\
  \hline
  \multirow{2}{*}{IIa$_{\pm}$} & \multirow{2}{*}{$w(w-1) = \pm d^2$} &
  \begin{tikzpicture}
    \useasboundingbox ({-0.6*\tw},{-0.5*\tw}) rectangle ({3.5*\tw},{0.52*\tw});
    \node[inner sep=0pt] (none) at (0,  0)     {\includegraphics[width=\picwidth]{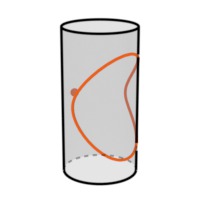}};
    \node[inner sep=0pt] (none) at (\tw,0)     {\includegraphics[width=\picwidth]{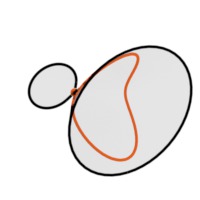}};
    \node[inner sep=0pt] (none) at ({2*\tw},0) {\includegraphics[width=\picwidth]{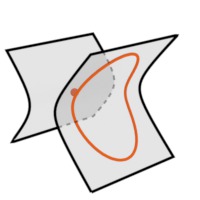}};
    \coordinate[label=135:{\scriptsize (2)}] (none) at ({1.5*\tw},{-0.5*\tw});
    \coordinate[label=-135:{\scriptsize\textcolor{gray}{$\star$}}] (none) at ({2.45*\tw},{0.45*\tw});
    \node[font=\scriptsize] at ({-0.4*\tw},0.) {$(+)$};
  \end{tikzpicture}\\
  & &
  \begin{tikzpicture}
    \useasboundingbox ({-0.6*\tw},{-0.51*\tw}) rectangle ({3.5*\tw},{0.5*\tw});
    \node[inner sep=0pt] (none) at (0,0)       {\includegraphics[width=\picwidth]{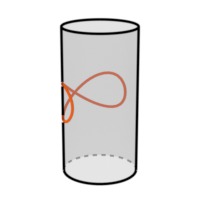}};
    \node[inner sep=0pt] (none) at (\tw,0)     {\includegraphics[width=\picwidth]{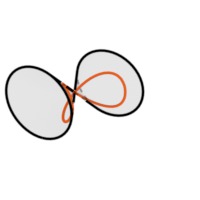}};
    \node[inner sep=0pt] (none) at ({2*\tw},0) {\includegraphics[width=\picwidth]{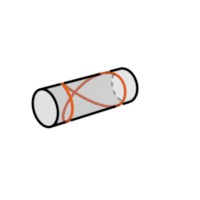}};
    \coordinate[label=135:{\scriptsize (2)}] (none) at ({1.5*\tw},{-0.5*\tw});
    \coordinate[label=-135:{\scriptsize\textcolor{gray}{$\star$}}] (none) at ({2.45*\tw},{0.45*\tw});
    \node[font=\scriptsize] at ({-0.4*\tw},0.) {$(-)$};
  \end{tikzpicture}\\
  \hdashline
  \multirow{2}{*}{IIb$_{\pm}$} & \multirow{2}{*}{$(w-1)=\pm d^2$} &
  \begin{tikzpicture}
    \useasboundingbox ({-0.6*\tw},{-0.5*\tw}) rectangle ({3.5*\tw},{0.52*\tw});
    \node[inner sep=0pt] (none) at (0,  0)     {\includegraphics[width=\picwidth]{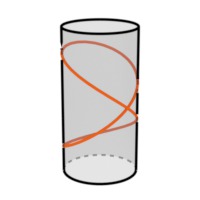}};
    \node[inner sep=0pt] (none) at (\tw,0)     {\includegraphics[width=\picwidth]{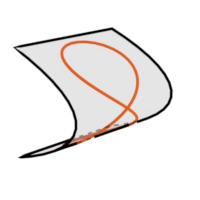}};
    \node[inner sep=0pt] (none) at ({2*\tw},0) {\includegraphics[width=\picwidth]{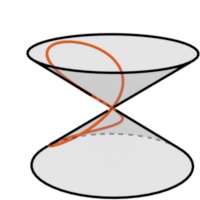}};
    \coordinate[label=135:{\scriptsize (2)}] (none) at ({2.7*\tw},{-0.5*\tw});
    \coordinate[label=-135:{\scriptsize\textcolor{gray}{$\star$}}] (none) at ({1.45*\tw},{0.45*\tw});
    \node[font=\scriptsize] at ({-0.4*\tw},0.) {$(+)$};
  \end{tikzpicture}\\
  & &
  \begin{tikzpicture}
    \useasboundingbox ({-0.6*\tw},{-0.51*\tw}) rectangle ({3.5*\tw},{0.5*\tw});
    \node[inner sep=0pt] (none) at (0,0)       {\includegraphics[width=\picwidth]{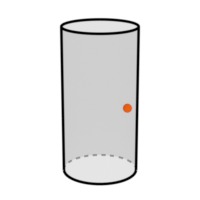}};
    \node[inner sep=0pt] (none) at (\tw,0)     {\includegraphics[width=\picwidth]{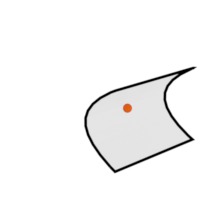}};
    \node[inner sep=0pt] (none) at ({2*\tw},0) {\includegraphics[width=\picwidth]{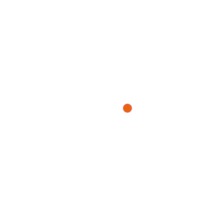}};
    \node[align=center, font=\scriptsize] at ({2*\tw},{-0.2*\tw}) {imag.\\ cone};
    \coordinate[label=135:{\scriptsize (2)}] (none) at ({2.5*\tw},{-0.5*\tw});
    \coordinate[label=-135:{\scriptsize\textcolor{gray}{$\star$}}] (none) at ({1.45*\tw},{0.45*\tw});
    \node[font=\scriptsize] at ({-0.4*\tw},0.) {$(-)$};
  \end{tikzpicture}\\
  \hline
  \pagebreak
  \hline
  \multirow{2}{*}{IIIa$_{\pm}$} & \multirow{2}{*}{$w^2=\pm d^2$} &
  \begin{tikzpicture}
    \useasboundingbox ({-0.6*\tw},{-0.5*\tw}) rectangle ({3.5*\tw},{0.52*\tw});
    \node[inner sep=0pt] (none) at (0,  0)     {\includegraphics[width=\picwidth]{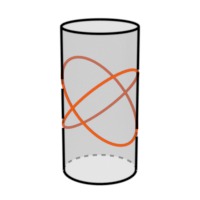}};
    \node[inner sep=0pt] (none) at (\tw,0)     {\includegraphics[width=\picwidth]{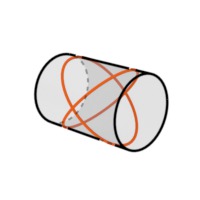}};
    \node[inner sep=0pt] (none) at ({2*\tw},0) {\includegraphics[width=\picwidth]{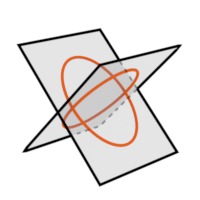}};
    \coordinate[label=135:{\scriptsize (2)}] (none) at ({2.5*\tw},{-0.5*\tw});
    \coordinate[label=-135:{\scriptsize\textcolor{gray}{$\star$}}] (none) at ({2.45*\tw},{0.45*\tw});
    \node[font=\scriptsize] at ({-0.4*\tw},0.) {$(+)$};
  \end{tikzpicture}\\
  & &
  \begin{tikzpicture}
    \useasboundingbox ({-0.6*\tw},{-0.51*\tw}) rectangle ({3.5*\tw},{0.5*\tw});
    \node[inner sep=0pt] (none) at (0,0)       {\includegraphics[width=\picwidth]{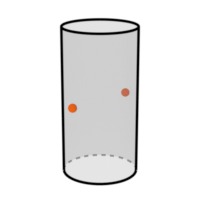}};
    \node[inner sep=0pt] (none) at (\tw,0)     {\includegraphics[width=\picwidth]{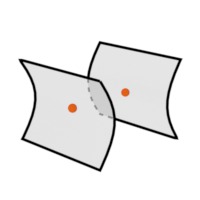}};
    \node[inner sep=0pt] (none) at ({2*\tw},0) {\includegraphics[width=\picwidth]{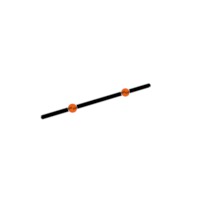}};
    \node[align=center, font=\scriptsize] at ({2*\tw},{-0.2*\tw}) {c.c.\\ planes};
    \coordinate[label=135:{\scriptsize (2)}] (none) at ({2.5*\tw},{-0.5*\tw});
    \coordinate[label=-135:{\scriptsize\textcolor{gray}{$\star$}}] (none) at ({2.45*\tw},{0.45*\tw});
    \node[font=\scriptsize] at ({-0.4*\tw},0.) {$(-)$};
  \end{tikzpicture}\\
  \hdashline
  \multirow{2}{*}{IIIb$_{\pm}$} & \multirow{2}{*}{$1= \pm d^2$} &
  \begin{tikzpicture}
    \useasboundingbox ({-0.6*\tw},{-0.5*\tw}) rectangle ({3.5*\tw},{0.52*\tw});
    \node[inner sep=0pt] (none) at (0,  0)     {\includegraphics[width=\picwidth]{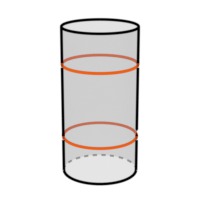}};
    \node[inner sep=0pt] (none) at (\tw,0)     {\includegraphics[width=\picwidth]{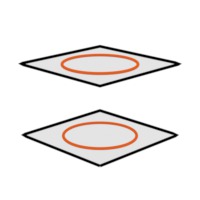}};
    \node[inner sep=0pt] (none) at ({2*\tw},0) {\includegraphics[width=\picwidth]{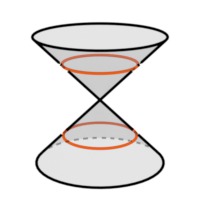}};
    \coordinate[label=135:{\scriptsize (2)}] (none) at ({1.5*\tw},{-0.5*\tw});
    \coordinate[label=-135:{\scriptsize\textcolor{gray}{$\star$}}] (none) at ({1.45*\tw},{0.45*\tw});
    \node[font=\scriptsize] at ({-0.4*\tw},0.) {$(+)$};
  \end{tikzpicture}\\
  & &
  \begin{tikzpicture}
    \useasboundingbox ({-0.6*\tw},{-0.51*\tw}) rectangle ({3.5*\tw},{0.5*\tw});
    \node[inner sep=0pt] (none) at (0,0)       {\includegraphics[width=\picwidth]{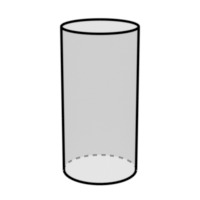}};
    \node[inner sep=0pt] (none) at (\tw,0)     {\includegraphics[width=\picwidth]{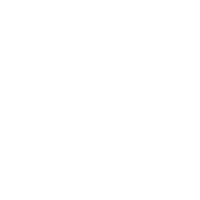}};
    \node[inner sep=0pt] (none) at ({2*\tw},0) {\includegraphics[width=\picwidth]{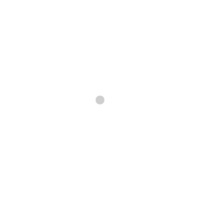}};
    \node[align=center, font=\scriptsize] at ({1*\tw},{-0.2*\tw}) {c.c.\\ parallel planes};
    \node[align=center, font=\scriptsize] at ({2*\tw},{-0.2*\tw}) {imag.\\ cone};
    \coordinate[label=135:{\scriptsize (2)}] (none) at ({1.5*\tw},{-0.52*\tw});
    \coordinate[label=-135:{\scriptsize\textcolor{gray}{$\star$}}] (none) at ({1.45*\tw},{0.45*\tw});
    \node[font=\scriptsize] at ({-0.4*\tw},0.) {$(-)$};
  \end{tikzpicture}\\
  \hline
  \vspace{0.5cm}&\vspace{0.5cm} & 
  \multirow{2}{*}{
    \begin{tikzpicture}
      \useasboundingbox ({-0.54*\tw},{-0.55*\tw}) rectangle ({3.5*\tw},{0.45*\tw});
      \node[inner sep=0pt] (none) at (0,0)       {\includegraphics[width=\picwidth]{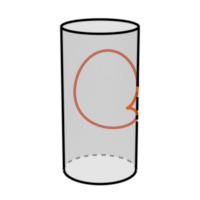}};
      \node[inner sep=0pt] (none) at (\tw,0)     {\includegraphics[width=\picwidth]{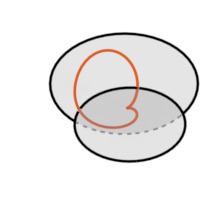}};
      \coordinate[label=135:{\scriptsize (3)}] (none) at ({1.5*\tw},{-0.5*\tw});
      \coordinate[label=-135:{\scriptsize\textcolor{gray}{$\star$}}] (none) at ({1.45*\tw},{0.45*\tw});
    \end{tikzpicture}}\\
  IV\vspace{0.9cm} & $(w-1)v = \varpm d^2$\vspace{0.9cm} & \\
  \hline
  \multirow{2}{*}{V$_{\pm}$} & \multirow{2}{*}{$(w-1)^2 = \pm d^2$} &
  \begin{tikzpicture}
    \useasboundingbox ({-0.6*\tw},{-0.5*\tw}) rectangle ({3.5*\tw},{0.52*\tw});
    \node[inner sep=0pt] (none) at (0,  0)     {\includegraphics[width=\picwidth]{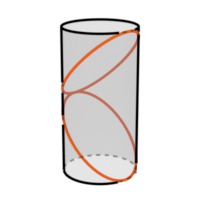}};
    \node[inner sep=0pt] (none) at (\tw,0)     {\includegraphics[width=\picwidth]{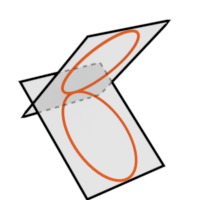}};
    \coordinate[label=135:{\scriptsize (3)}] (none) at ({1.5*\tw},{-0.52*\tw});
    \coordinate[label=-135:{\scriptsize\textcolor{gray}{$\star$}}] (none) at ({1.45*\tw},{0.45*\tw});
    \node[font=\scriptsize] at ({-0.4*\tw},0.) {$(+)$};
  \end{tikzpicture}\\
  & &
  \begin{tikzpicture}
    \useasboundingbox ({-0.6*\tw},{-0.51*\tw}) rectangle ({3.5*\tw},{0.5*\tw});
    \node[inner sep=0pt] (none) at (0,0)       {\includegraphics[width=\picwidth]{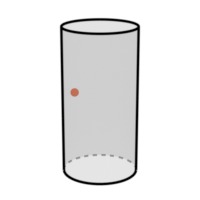}};
    \node[inner sep=0pt] (none) at (\tw,0)     {\includegraphics[width=\picwidth]{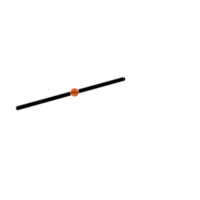}};
    \node[align=center, font=\scriptsize] at ({1*\tw},{-0.2*\tw}) {c.c.\\ planes};
    \coordinate[label=135:{\scriptsize (3)}] (none) at ({1.5*\tw},{-0.5*\tw});
    \coordinate[label=-135:{\scriptsize\textcolor{gray}{$\star$}}] (none) at ({1.45*\tw},{0.45*\tw});
    \node[font=\scriptsize] at ({-0.4*\tw},0.) {$(-)$};
  \end{tikzpicture}\\
  \hline
  \vspace{0.5cm}&\vspace{0.5cm} & 
  \multirow{2}{*}{
    \begin{tikzpicture}
      \useasboundingbox ({-0.54*\tw},{-0.55*\tw}) rectangle ({3.5*\tw},{0.45*\tw});
      \node[inner sep=0pt] (none) at (0,  0)     {\includegraphics[width=\picwidth]{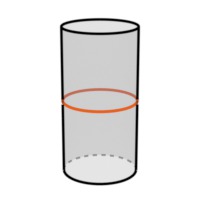}};
      \node[inner sep=0pt] (none) at (\tw,0)     {\includegraphics[width=\picwidth]{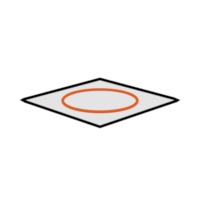}};
      \node[align=center, font=\scriptsize] at ({1*\tw},{-0.2*\tw}) {double plane};
      \coordinate[label=135:{\scriptsize (3)}] (none) at ({1.5*\tw},{-0.5*\tw});
      \coordinate[label=-135:{\scriptsize\textcolor{gray}{$\star$}}] (none) at ({1.45*\tw},{0.45*\tw});
    \end{tikzpicture}}\\
  O\vspace{0.9cm} & $0 = \varpm d^2$\vspace{0.9cm} & \\
  \hline
\end{longtable}
\end{center}
\begingroup
\captionof{figure}{
  Normal forms of degenerate quadrics of generic pencils, corresponding to the classification of different types of planar pencils of conics listed and illustrated in Figures \ref{classification} and \ref{classification-degenerate-conics}. These are obtained by ``adding'' a $\pm d^2$ term to the algebraic representation of the degenerate conics depicted in Figure \ref{classification-degenerate-conics}. However, 
  type O corresponds to the case $S=0$ which does not encode a pencil of conics in the $(v,w)$-plane.
  Note that types Ia, IV, and O each generate only one class, that is, different signs in $\pm d^2$ lead to equivalent pencils.
  Types Ic$_+$ and IIIb$_-$ correspond to empty hypercycles.
  For each type, a normal form of one degenerate quadric is given which spans the generic pencil together with the Blaschke cylinder.
  Furthermore, for each type, all degenerate quadrics are shown together with the hypercycle base curve.
  The multiplicity is given if greater than 1, and a $\star$ indicates that the degenerate quadric corresponds to the normal form recorded in the second column of the table.
  \label{normal_forms}
}
\endgroup

\addtocounter{table}{-1}%


\section{An elliptic function representation of confocal checkerboard IC-nets}
\label{s.confocal}

Explicit parametrisations of confocal checkerboard IC-nets and their Laguerre transforms may now be obtained by parametrising the hypercycle base curves  associated with a pencil of quadrics in terms of elliptic functions. 

\subsection{Elliptic confocal checkerboard IC-nets}

It is recalled that ``elliptic'' confocal checkerboard IC-nets, that is checkerboard IC-nets the lines of which are tangent to an ellipse
\bela{E6.1}
  \frac{x^2}{\alpha^2} + \frac{y^2}{\beta^2} = 1,
\ela
correspond to a pencil of quadrics 
\bela{E6.7}
  (\alpha^2 + \lambda)v^2 + (\beta^2  + \lambda)w^2 = d^2 + \lambda
\ela
generated by an elliptic cone and the Blaschke cylinder, namely
\bela{E6.2}
  \alpha^2v^2 + \beta^2w^2 = d^2,\quad v^2 + w^2 = 1.
\ela
In the following, we assume that $\alpha^2\geq\beta^2$ without loss of generality. The associated base curve is the set of all points $(v,w,d)$ obeying the pair \eqref{E6.2}. 
This corresponds to types Ic$_-$ or IIIb$_+$ of the classification of pencils of quadrics with the two components of the base curve being mapped into each other by $d\rightarrow -d$.
The nature of any quadric $\mathcal{H}$ in the pencil \eqref{E6.7} depends on the value of the associated parameter $\lambda$. Accordingly, there exist two cases.

\subsubsection{The case \boldmath $\lambda\geq 0$}
\label{s.ellipticg0}

In this case, the quadric $\mathcal{H}$ constitutes a one-sheeted hyperboloid (or a cone for $\lambda=0$) which is aligned with the Blaschke cylinder. 
If we now parametrise $v$ and $w$ in terms of Jacobi elliptic functions \cite{NIST} $\jac{cn}$ and $\jac{sn}$ respectively then the general solution of \eqref{E6.2} is given by
\bela{E6.4}
  \bv_{\pm}(\psi) =  \begin{pmatrix} v(\psi)\\ w(\psi)\\ d_\pm(\psi) \end{pmatrix} = \begin{pmatrix} \jac{cn}(\psi,k)\\ \jac{sn}(\psi,k)\\ \pm\alpha\jac{dn}(\psi,k)\end{pmatrix},\quad
  k = \sqrt{1-\frac{\beta^2}{\alpha^2}},
\ela
where $\psi$ constitutes the parameter along the (two components of the) base curve. Any pair of  points on the two components of the base curve may be represented by
\bela{E6.5}
  \bv_\pm(\psi_0),\quad \bv_\mp(\psi_1),\quad \psi_1 = s + \psi_0
\ela
for any fixed choice of the above signs. If we demand that, for fixed $s$, the one-parameter family of lines
\bela{E6.6}
  \bl(\psi_0,t) = \bv_\pm(\psi_0) + t [\bv_\mp(\psi_1) - \bv_\pm(\psi_0)],\quad t\in\R
\ela
consist of generators of the quadric $\mathcal{H}$ then we obtain a relationship between the parameters $s$ and $\lambda$ which is to be independent of $\psi_0$. Indeed, insertion of $\bl$ into \eqref{E6.7} produces
\bela{E6.8}
  v(\psi_0)v(\psi_1)(\alpha^2 + \lambda) + w(\psi_0)w(\psi_1)(\beta^2 + \lambda) = d_\pm(\psi_0)d_\mp(\psi_1) + \lambda.
\ela
It is observed that, geometrically, the latter merely represents the fact that the points $\bv_\pm(\psi_0)$ and $\bv_\mp(\psi_1)$ are required to lie in the tangent planes to the hyperboloid \eqref{E6.7} at those two points. Now, comparison with the general identity \cite{NIST}
 \begin{gather}\label{E6.9}
  c_s\jac{sn}(\psi_0,k)\jac{sn}(\psi_1,k) + c_c\jac{cn}(\psi_0,k)\jac{cn}(\psi_1,k) = c_d\jac{dn}(\psi_0,k)\jac{dn}(\psi_1,k) + 1\\
    c_s = \jac{dc}(s,k) + c_d(1-k^2)\jac{nc}(s,k), \quad c_c = \jac{nc}(s,k) + c_d\jac{dc}(s,k)
 \end{gather}
for elliptic functions with $\psi_1 = \psi_0 + s$ shows that it is required that
\bela{E6.10}
 \begin{split}
  \alpha^2 + \lambda & = \lambda\jac{nc}(s,k) - \alpha^2\jac{dc}(s,k)\\
  \beta^2 + \lambda & = \lambda\jac{dc}(s,k) - \alpha^2(1-k^2)\jac{nc}(s,k).
 \end{split}
\ela
Since the latter two conditions coincide, we conclude that
\bela{E6.11}
  \lambda = \alpha^2\frac{\jac{dc}(s,k) + 1}{\jac{nc}(s,k) - 1} = \alpha^2\jac{cs}^2\left(\frac{s}{2},k\right) \geq 0
\ela
so that any family of generators of the quadric \eqref{E6.7} for $\lambda\geq0$ is encoded via the parametrisation \eqref{E6.4}-\eqref{E6.6} in an appropriately chosen parameter $s$. In other words, a translation of the argument $\psi$ in the parametrisation \eqref{E6.4} by some fixed quantity $s$ together with a change of the component of the base curve gives rise to a family of generators of a unique quadric $\mathcal{H}$ of the pencil. The second family of generators is obtained by letting $s\rightarrow -s$.

\subsubsection{The case \boldmath $-\alpha^2 \leq \lambda \leq -\beta^2$}

This case corresponds to the remaining one-sheeted hyperboloids (and an elliptic and hyperbolic cylinder for $\lambda=-\alpha^2$ and $\lambda=-\beta^2$ respectively) of the pencil \eqref{E6.7} which are aligned with the $v$-axis. 
It is now convenient to introduce an additional parameter $\epsilon$ in the parametrisation \eqref{E6.4} according to
\bela{E6.12}
  \bv^\epsilon_{\pm}(\psi) = \begin{pmatrix} \jac{cn}(\psi,k)\\ \epsilon\jac{sn}(\psi,k)\\ \pm\alpha\jac{dn}(\psi,k)\end{pmatrix},\quad \epsilon^2 = 1
\ela
and consider two points
\bela{E6.13}
  \bv^\epsilon_\pm(\psi_0),\quad \bv^{-\epsilon}_\pm(\psi_1),\quad \psi_1 = s + \psi_0
\ela
on any fixed component of the base curve. Then, the lines
\bela{E6.14}
  \bl(\psi_0,t) = \bv^\epsilon_\pm(\psi_0) + t [\bv^{-\epsilon}_\pm(\psi_1) - \bv^\epsilon_\pm(\psi_0)],\quad t\in\R
\ela
turn out to be generators on any hyperboloid $\mathcal{H}$ given by \eqref{E6.7} for
\bela{E6.15}
  \lambda = -\beta^2\jac{nd}^2\left(\frac{s}{2},k\right).
\ela
Once again, for any fixed $\lambda$ in the current range, there exists an $|s|$ such that the generators of the associated quadric $\mathcal{H}$ which pass through the base curve are parametrised by \eqref{E6.14}. The two signs of $s$ correspond to the two families of generators of $\mathcal{H}$.

\subsubsection{Construction of elliptic confocal checkerboard IC-nets}

In order to illustrate the construction of checkerboard IC-nets from pencils of quadrics presented in Section 3, we consider two quadrics $\calH$ and $\tilde{\calH}$ of the pencil \eqref{E6.7} corresponding to a pair of parameters $s$  and $\tilde{s}$ which are related to $\lambda,\tilde{\lambda}>0$ by \eqref{E6.11}. It is recalled that any point $(v,w,d)$ of the base curve is in one-to-one correspondence with a line
\bela{E6.16}
  vx + wy = d.
\ela
In this sense, we refer to $(v,w,d)$  as a point on the Blaschke cylinder or a line in the $(x,y)$-plane.
The first (``vertical'') family of lines of the elliptic confocal checkerboard is then given by
\bela{E6.17}
 \begin{split}
  \bv_{2n}^v & = \bv_+(\psi_0^v + n(s + \tilde{s})),\\
   \bv_{2n+1}^v & = \bv_-(\psi_0^v + n(s + \tilde{s}) + s),
 \end{split}
 \quad n\in\Z,
\ela
where $\psi_0^v$ is arbitrary and corresponds to one of the ``initial conditions'' of the construction. The second (``horizontal'') family of lines is associated with the two other families of generators of $\calH$ and $\tilde{\calH}$ encoded in the parameters $-s$ and $-\tilde{s}$ respectively.
Accordingly, the construction of confocal checkerboard IC-nets in this case may be summarised as follows.
\begin{theorem}\label{ellipticcheckerboards}
For any pairs of parameters $\alpha\geq\beta>0$ and $\psi_0^v,\psi_0^h$ and $s,\tilde{s}$, the lines 
\bela{EE6.17}
 \begin{split}
  \bv_{2n}^v & = \bv_+(\psi_0^v + n(s + \tilde{s})),\\
   \bv_{2n+1}^v & = \bv_-(\psi_0^v + n(s + \tilde{s}) + s),\\
   \bv_{2n}^h & = \bv_+(\psi_0^h - n(s + \tilde{s})),\\
   \bv_{2n+1}^h & = \bv_-(\psi_0^h - n(s + \tilde{s}) - s),
 \end{split}
 \qquad
  \begin{split}
  \bv_{\pm}(\psi) &= \begin{pmatrix} \jac{cn}(\psi,k)\\ \jac{sn}(\psi,k)\\ \pm\alpha\jac{dn}(\psi,k)\end{pmatrix}\\
  k &= \sqrt{1-\frac{\beta^2}{\alpha^2}}.
  \end{split}
\ela
form a (confocal) checkerboard IC-net and are tangent to the ellipse
\bela{EE6.1}
  \frac{x^2}{\alpha^2} + \frac{y^2}{\beta^2} = 1.
\ela
The parameters $s,\tilde{s}$ determine the associated hyperboloids $\mathcal{H},\tilde{\mathcal{H}}$ of the pencil
\bela{EE6.7}
  (\alpha^2 + \lambda)v^2 + (\beta^2  + \lambda)w^2 = d^2 + \lambda
\ela
according to
\bela{EE6.11}
  \lambda = \alpha^2\jac{cs}^2\left(\frac{s}{2},k\right).
\ela
\end{theorem}

``Embedded'' elliptic confocal checkerboard IC-nets are obtained by requiring periodicity, that is,
\bela{E6.19}
  s + \tilde{s} = 4\sfK + \frac{4\sfK}{N}, \quad N\in\N,
 \ela
where the quarter-period $\sfK(k)$ of the Jacobi elliptic functions is given by the complete elliptic integral of the first kind, and demanding that the two families of lines $\bv^v$ and $\bv^h$ coincide up to their orientation. The latter may be achieved by relating the parameters $\psi_0^v$ and $\psi_0^h$ according to
\bela{E6.21}
  \psi_0^v = 2\sfK + \psi_0^h - s
\ela
so that
\bela{E6.20}
  \bv_{2n}^v = -\bv_{-2n+1}^h,\quad \bv_{2n+1}^v = -\bv_{-2n}^h.
\ela
If we now parametrise the constraint \eqref{E6.19} by
\bela{E6.22}
  s = 2\sfK + \frac{4\sfK}{N} - \kappa,\quad \tilde{s} = 2\sfK + \kappa,
\ela
where $\kappa$ is the arbitrary parameter, then
\bela{E6.23}
  \psi_0^h = \psi_0^v + \frac{4\sfK}{N} - \kappa.
\ela
For $\kappa=0$, the lines $\bv_{2n}^v$ and $\bv_{2n-1}^v$ coincide up to their orientation and the quadric $\tilde{\mathcal{H}}$ becomes the cone \eqref{E6.2}$_1$ since $\tilde{s}=2\sfK$ so that $\tilde{\lambda}=0$. Hence, as discussed in Section 3, an elliptic IC-net is obtained as depicted in Figure~\ref{global1} (left) for $N=32$ and $\psi_0^v=0.2$. Here, $\alpha=2$ and $\beta=1$.
As $\kappa$ increases, the coinciding lines separate and the circles of zero radius between the coinciding lines enlarge so that a non-degenerate confocal checkerboard emerges with $\tilde{\mathcal{H}}$ being a proper hyperboloid. 
An elliptic confocal checkerboard IC-net for $\kappa=0.1$ is displayed in Figure \ref{global1} (right).
\begin{figure}
  \centering
  \raisebox{-0.5\height}{\includegraphics[width=0.49\textwidth]{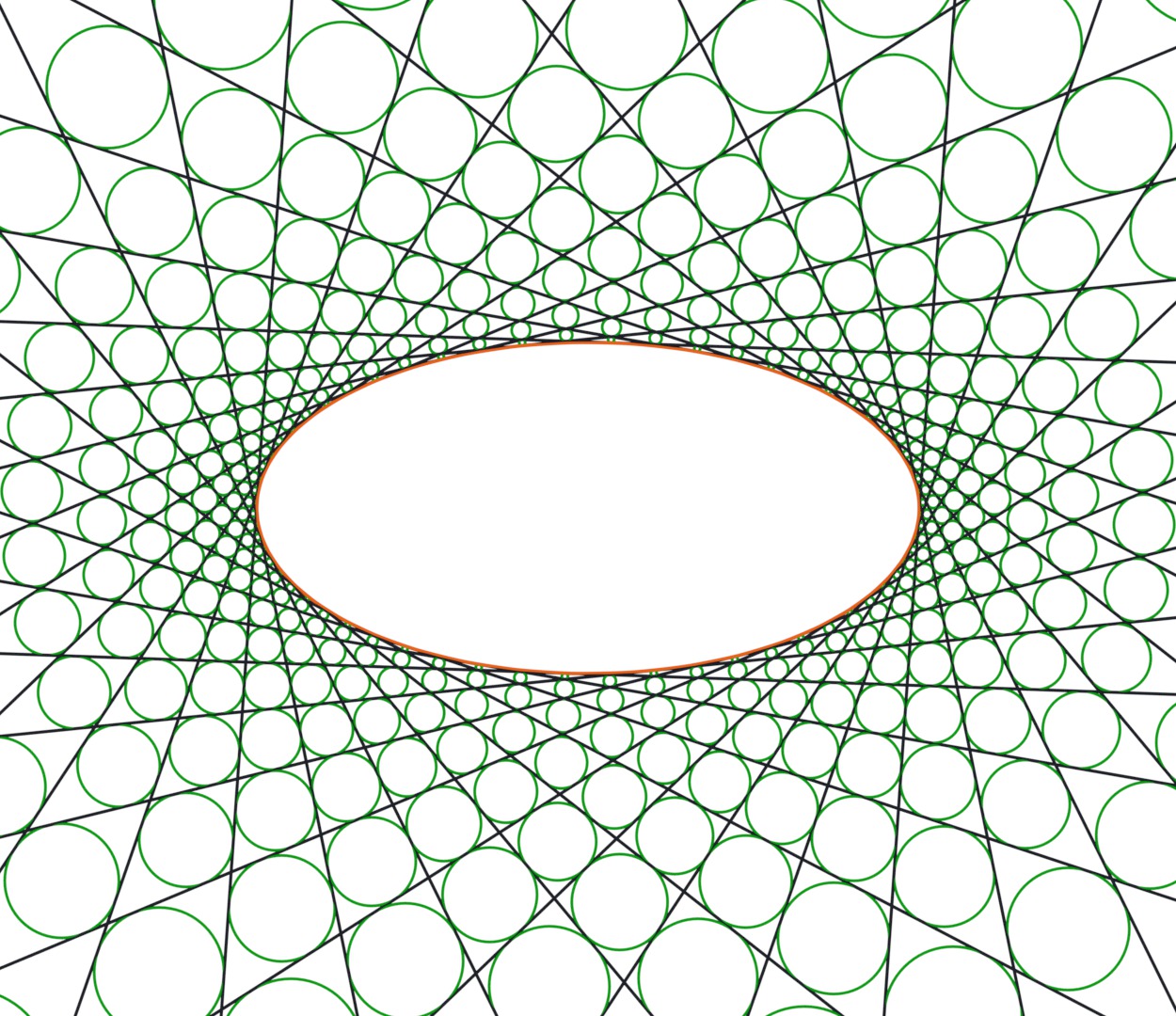}}
  \raisebox{-0.5\height}{\includegraphics[width=0.49\textwidth]{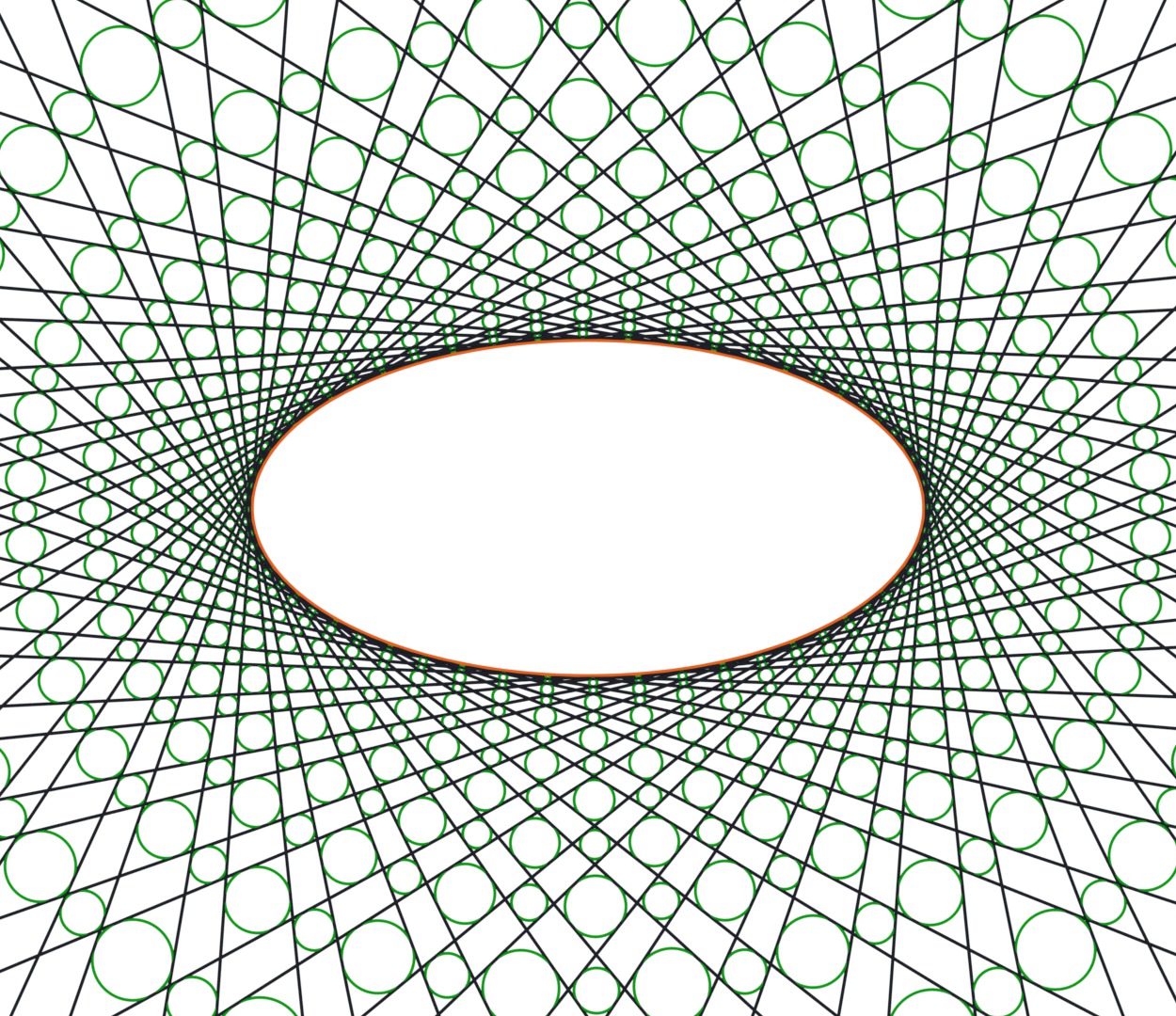}}
  \caption{
    Periodic elliptic confocal checkerboard IC-nets for $\alpha=2$, $\beta=1$, $N=32$.
    \emph{Left:} $\kappa = 0$, corresponding to the degenerate case of an IC-net.
    \emph{Right:} $\kappa = 0.1$.
  }
  \label{global1}
\end{figure}

\subsection{Hyperbolic confocal checkerboard IC-nets}

The lines of ``hyperbolic'' confocal checkerboard IC-nets are tangent to a hyperbola given by, without loss of generality,
\bela{E6.24}
  \frac{x^2}{\alpha^2} - \frac{y^2}{\beta^2} = 1.
\ela
In terms of the Blaschke model, the associated pencil of quadrics 
\bela{E6.27}
  (\alpha^2 + \lambda)v^2 - (\beta^2 - \lambda)w^2 = d^2 + \lambda
\ela
is generated by the pair of quadrics
\bela{E6.25}
  \alpha^2v^2 - \beta^2w^2 = d^2,\quad v^2 + w^2 = 1.
\ela
The base curve is the intersection of these two quadrics (and all members of the pencil), corresponding to type Ia$_-$ of the classification of pencils of quadrics. The two components of the base curve are mapped into each other via $v\rightarrow -v$. As in the elliptic case, the base curve may be parametrised in terms of elliptic functions and one has to distinguish between two cases. 

\blank{In fact, we observe that the elliptic and hyperbolic cases are formally linked by letting $\beta \rightarrow i\beta$, so that $ k<1\rightarrow k >1$, and applying the identities
\begin{gather}\label{E6.30a}
  \jac{sn}(z,1/k) = k\jac{sn}(z/k,k),\quad \jac{cn}(z,1/k) =\jac{dn}(z/k,k)\\ \jac{dn}(z,1/k) = \jac{cn}(z/k,k).
\end{gather}
An appropriate complex shift of the argument $\psi$ then establishes the formal equivalence.}

\subsubsection{The case \boldmath $-\alpha^2 \leq \lambda\leq 0$}

In this case, any quadric $\mathcal{H}$ constitutes a one-sheeted hyperboloid (or a cone for $\lambda=0$ and an elliptic cylinder for $\lambda=-\alpha^2$) which is aligned with the $v$-axis.
It is then readily verified that 
\bela{E6.26}
  \bv_{\pm}(\psi) =  \begin{pmatrix} v_\pm(\psi)\\ w(\psi)\\ d(\psi) \end{pmatrix} = \begin{pmatrix} \pm\jac{dn}(\psi,k)\\ k\jac{sn}(\psi,k)\\ \alpha\jac{cn}(\psi,k)\end{pmatrix},\quad
  k = \sqrt{\frac{\alpha^2}{\alpha^2 + \beta^2}}
\ela
covers all points of these two components. Given any two points of the form \eqref{E6.5} on the two components of the base curve, one may now determine the corresponding quadric $\mathcal{H}$ of the pencil which contains the lines \eqref{E6.6} as generators for fixed $s$ and all $\psi_0$. A calculation along the lines of the previous subsection reveals that the pencil parameter $\lambda$ linked to the parameter $s$ is given by
\bela{E6.28}
  \lambda = -\alpha^2\jac{cn}^2\left(\frac{s}{2},k\right).
\ela

\subsubsection{The case \boldmath $\lambda\geq\beta^2$}

This case corresponds to the remaining two-sheeted hyperboloids $\mathcal{H}$ (or a hyperbolic cylinder for $\lambda=\beta^2$) of the pencil \eqref{E6.27} which are aligned with the Blaschke cylinder.
In analogy with the elliptic case, it is convenient to introduce a parameter $\epsilon$ in the parametrisation
\bela{E6.29}
  \bv^\epsilon_{\pm}(\psi) = \begin{pmatrix} \pm\jac{dn}(\psi,k)\\ k\jac{sn}(\psi,k)\\ \epsilon \alpha\jac{cn}(\psi,k)\end{pmatrix},\quad \epsilon^2 = 1
\ela
of the base curve. Then, any pair of points of the type \eqref{E6.13} on any fixed component of the base curve may be connected by a line \eqref{E6.14} which constitutes a generator of the quadric \eqref{E6.27} for
\bela{E6.30}
  \lambda =  (\alpha^2 + \beta^2)\jac{ds}^2\left(\frac{s}{2},k\right)
\ela
independently of the value of $\psi$. 

\subsubsection{Construction of hyperbolic confocal checkerboard IC-nets}

Once again, as an illustration, we now consider two quadrics $\calH$ and $\tilde{\calH}$ of the pencil \eqref{E6.27} defined via the relation \eqref{E6.28} by given parameters $s$ and $\tilde{s}$. The formulae \eqref{EE6.17} for the two families of lines of the corresponding confocal checkerboard IC-nets remain valid in the current hyperbolic case  but $\bv_{\pm}$ and $k$ are now defined by \eqref{E6.26}. In fact, the conditions \eqref{E6.19}-\eqref{E6.23} for embeddedness are likewise applicable. Thus, for $\kappa=0$, one obtains hyperbolic IC-nets as illustrated in Figure \ref{global4} (left) for $N=32$, $\psi_0^v=0.2$ and $\alpha=\beta=1$.
Furthermore, a hyperbolic confocal checkerboard IC-net for $\kappa=0.1$ is depicted in Figure \ref{global4} (right).
\begin{figure}
  \centering
  \raisebox{-0.5\height}{\includegraphics[width=0.49\textwidth]{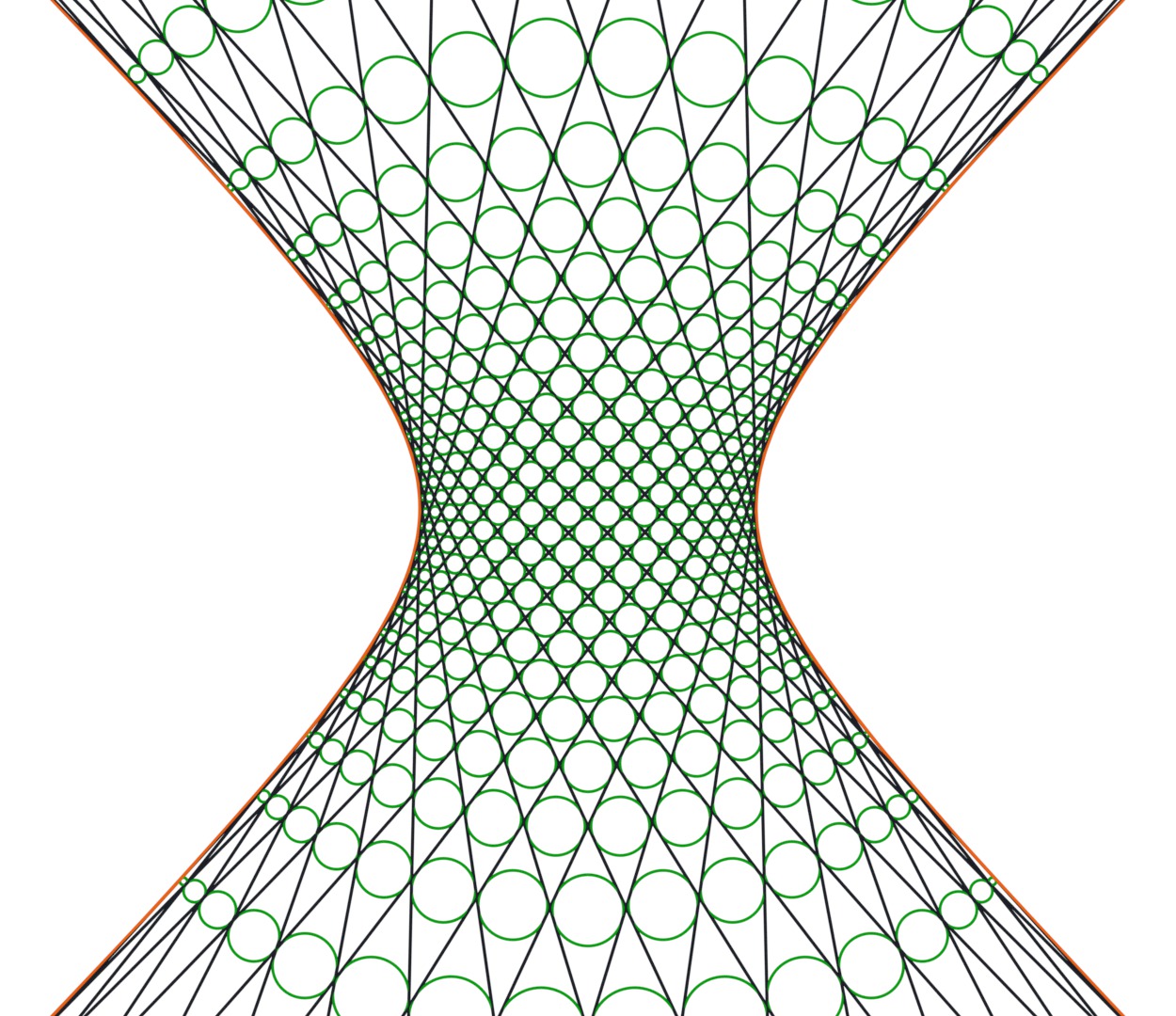}}
  \raisebox{-0.5\height}{\includegraphics[width=0.49\textwidth]{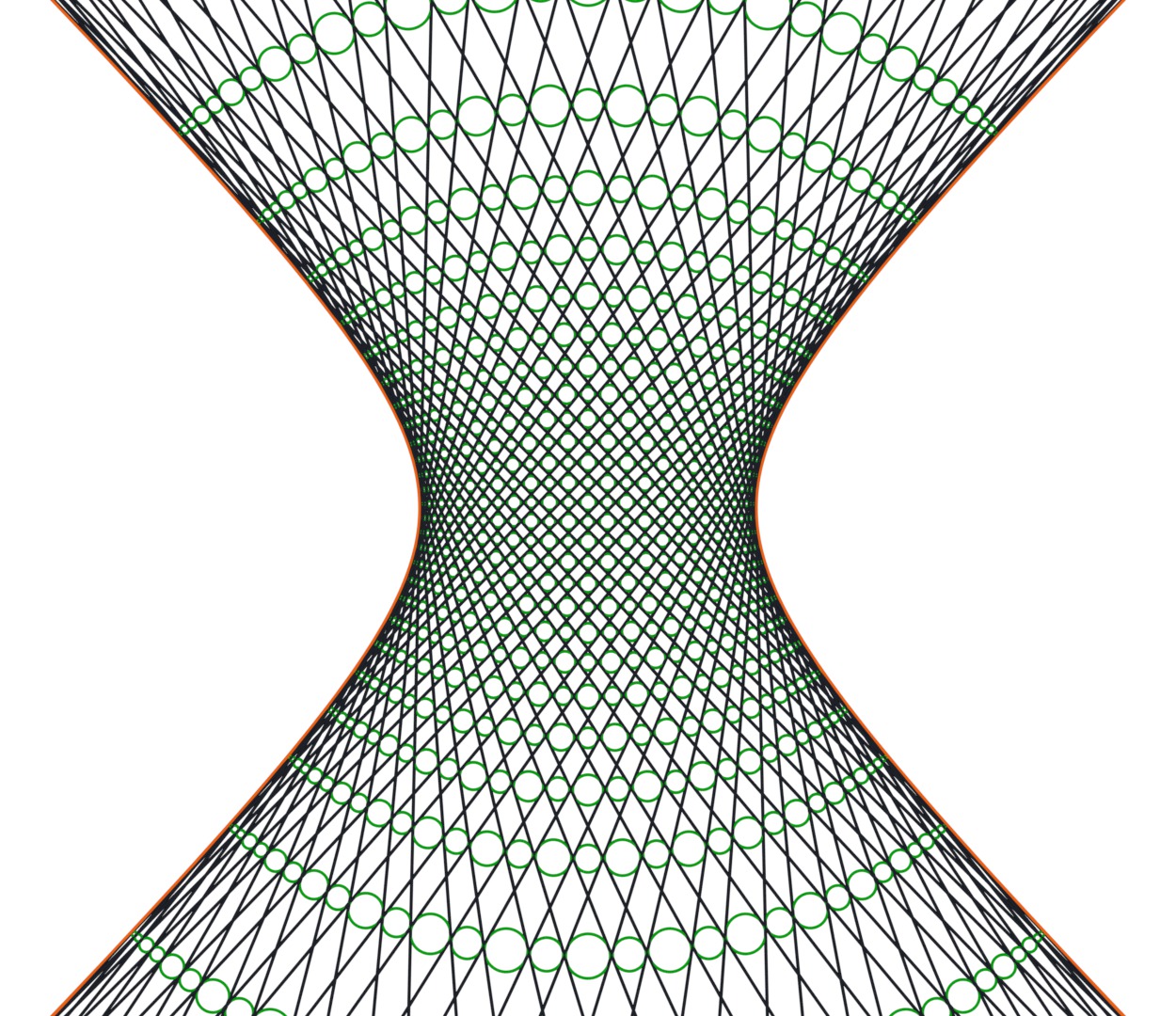}}
  \caption{
    Periodic hyperbolic confocal checkerboard IC-nets for $\alpha=\beta=1$, $N=32$.
    \emph{Left:} $\kappa = 0$, corresponding to the degenerate case of an IC-net.
    \emph{Right:} $\kappa = 0.1$.
  }
\label{global4}
\end{figure}

\subsection{IC-nets and discrete confocal conics}

We conclude this section by relating IC-nets to the discrete confocal conics proposed in \cite{BSST16,BSST17}. For brevity, we focus on the class of IC-nets which is subsumed by the class of elliptic checkerboard IC-nets captured by Theorem~\ref{ellipticcheckerboards}. These IC-nets are associated with the choice $\tilde{s}=2\sfK$, corresponding to a cone as the corresponding quadric $\tilde{\cal{H}}$. Thus, if we set $s=2\sfK + \delta$ and $\psi_0^{v/h} = \delta n_0^{v/h}$ then the explicit parametrisation \eqref{EE6.17} leads to the following corollary.

\begin{corollary}
The (coinciding pairs of non-oriented) lines of an elliptic IC-net of the type captured by Theorem \ref{ellipticcheckerboards} may be represented by the lines
\bela{E6.31}
  \bv_{n_1}^v = \bv(\delta(n_0^v +  n_1)),\quad \bv_{n_2}^h = \bv(\delta(n_0^h - n_2)),\quad n_i\in\Z,
\ela
where 
\bela{E6.32}
  \bv(\psi) =  \begin{pmatrix} v(\psi)\\ w(\psi)\\ d(\psi) \end{pmatrix} = \begin{pmatrix} \jac{cn}(\psi,k)\\ \jac{sn}(\psi,k)\\ \alpha\jac{dn}(\psi,k)\end{pmatrix},\quad
  k = \sqrt{1-\frac{\beta^2}{\alpha^2}}.
\ela
\end{corollary}

In the sense of Laguerre geometry, the quadruples of oriented lines
\bela{E6.33}
  \bv_{n_1}^v,\quad \bv_{n_2}^h,\quad -\bv_{n_1+1}^v,\quad -\bv_{n_2+1}^h
\ela
are tangent to circles so that, by construction,
\bela{E6.34}
  \begin{vmatrix} 1 & v_{n_1}^v & w_{n_1}^v & d_{n_1}^v\\[1mm]
                         1 & v_{n_2}^h & w_{n_2}^h & d_{n_2}^h\\[1mm]
                         1 & -v_{n_1+1}^v & -w_{n_1+1}^v & -d_{n_1+1}^v\\[1mm]
                         1 & -v_{n_2+1}^h & -w_{n_2+1}^h & -d_{n_2+1}^h\\[1mm]
  \end{vmatrix} = 0
\ela
which coincides with the known identity \cite{NIST}
\bela{E6.35}
  \begin{vmatrix} 1 & \jac{sn}(z_1,k) &\jac{cn}(z_1,k) & \jac{dn}(z_1,k)\\
                         1 & \jac{sn}(z_2,k) &\jac{cn}(z_2,k) & \jac{dn}(z_2,k)\\
                         1 & \jac{sn}(z_3,k) &\jac{cn}(z_3,k) & \jac{dn}(z_3,k)\\
                         1 & \jac{sn}(z_4,k) &\jac{cn}(z_4,k) & \jac{dn}(z_4,k)                     
  \end{vmatrix} = 0,\quad z_1+z_2+z_3+z_4 = 0
\ela
for Jacobi elliptic functions if one identifies the arguments in \eqref{E6.34} and \eqref{E6.35} appropriately. The point of intersection $(x_\times,y_\times)$ of two lines $\bv_{n_1}^v$ and $\bv_{n_2}^h$ is given by the solution of the two linear equations
\bela{E6.36}
  v_{n_1}^v x_\times + w_{n_1}^v y_\times = d_{n_1}^v,\quad  v_{n_2}^h x_\times + w_{n_2}^h y_\times = d_{n_2}^h.
\ela
If we now make the change of variables
\bela{E6.37}
 \begin{split}
  n_1 = m_2 + m_1,\quad \xi_1 & = \delta\left[m_1 + \frac{1}{2}(n_0^v + n_0^h)\right]\\
  n_2 = m_2 - m_1,\quad \xi_2 & = \delta\left[m_2 + \frac{1}{2}(n_0^v - n_0^h)\right],
 \end{split}
\ela
leading to
\bela{E6.38}
  \bv^v = \bv(\xi_1 + \xi_2),\quad \bv^h = \bv(\xi_1-\xi_2),
\ela
where we have suppressed the dependence on $m_i\in\frac{1}{2}\Z$, then consideration of the sum and the difference of the linear equations \eqref{E6.36} and application of the addition theorems for Jacobi elliptic functions produces the following result.

\begin{theorem}
The points of intersection of the pairs of lines $(\bv_{m_2+m_1}^v$, $\bv_{m_2-m_1}^h)$ of the elliptic IC-nets \eqref{E6.31}, \eqref{E6.32} are given by the compact formulae
\bela{E6.39}
  x_\times = \alpha\jac{cd}(\xi_1,k)\jac{dc}(\xi_2,k),\quad y_\times = \alpha (1-k^2)\jac{sd}(\xi_1,k)\jac{nc}(\xi_2,k)
\ela
with
\bela{EE6.37}
 \xi_1 = \delta\left[m_1 + \frac{1}{2}(n_0^v + n_0^h)\right],\quad
 \xi_2 = \delta\left[m_2 + \frac{1}{2}(n_0^v - n_0^h)\right].
\ela
These lie on the conics
\bela{E6.40}
  \frac{x_\times^2}{\lambda^e(\xi_2)} + \frac{y_\times^2}{\mu^e(\xi_2)} = 1, \quad \frac{x_\times^2}{\lambda^h(\xi_1)} + \frac{y_\times^2}{\mu^h(\xi_1)} = 1,
\ela
where
\bela{E6.41}
 \begin{split}
   \lambda^e = \alpha^2\jac{dc}^2(\xi_2,k), \phantom{k^2}\quad \mu^e &= \alpha^2(1 - k^2)\jac{nc}^2(\xi_2,k)\\
    \lambda^h = \alpha^2k^2\jac{cd}^2(\xi_1,k),\quad \mu^h &= -\alpha^2k^2(1 - k^2)\jac{sd}^2(\xi_1,k),
 \end{split}
\ela
which are confocal to the ellipse of contact \eqref{E6.1}.
\end{theorem}

\begin{proof}
It is straightforward to verify that $\xi_1$ and $\xi_2$ as given by \eqref{EE6.37} obey the quadratic relations \eqref{E6.40}. Moreover, since
\bela{E6.42}
  \lambda^e - \mu^e = \lambda^h - \mu^h = \alpha^2k^2 = \alpha^2 - \beta^2,
\ela
the conics \eqref{E6.40} are indeed in the set of confocal conics defined by the ellipse \eqref{E6.1}.
\end{proof}

As proven in \cite{AB}, the centres of the circles of IC-nets lie on affine transforms of confocal conics. Hence, the algebraic structure of their coordinates should coincide with that of the points of intersection of pairs of lines as given by \eqref{E6.39}. In order to confirm this assertion, it is observed that the centre $(x_\odot,y_\odot)$ and the radius $r_\odot$ of any particular circle are determined by solving any three equations of the linear system
\bela{E6.43}
  v x_\odot + w y_\odot - r_\odot = d,\quad \bv \in \{\bv_{n_1}^v,\bv_{n_2}^h,-\bv_{n_1+1}^v,-\bv_{n_2+1}^h\}.
\ela
Elimination of $r_\odot$ leads to the pair of equations
\bela{E6.44}
 \begin{split}
  (v_{n_1}^v - v_{n_2}^h) x_\odot + (w_{n_1}^v - w_{n_2}^h) y_\odot & = (d_{n_1}^v - d_{n_2}^h)\\
  (v_{n_1+1}^v - v_{n_2+1}^h) x_\odot + (w_{n_1+1}^v - w_{n_2+1}^h) y_\odot & = (d_{n_1+1}^v - d_{n_2+1}^h).
 \end{split}
\ela 
Once again, the addition theorems for Jacobi elliptic functions and the double- and half-``angle'' formulae \cite{NIST}
\bela{E6.46}
 \begin{split}
  \jac{sn}(2z,k) & = \frac{2\jac{sn}(z,k)\jac{cn}(z,k)\jac{dn}(z,k)}{1-k^2\jac{sn}^4(z,k)}\\
  \jac{sn}^2(\tfrac{1}{2}z,k) & = \frac{1-\jac{cn}(z,k)}{1+\jac{dn}(z,k)}
 \end{split}
\ela
give rise to a compact form of its solution.

\begin{theorem}
The centres of the circles of the elliptic IC-nets \eqref{E6.31}, \eqref{E6.32} are given by
\bela{E6.45}
 \begin{split}
  x_\odot & = \alpha\jac{dc}(\tfrac{\delta}{2},k)\jac{cd}(\xi_1,k)\jac{dc}(\xi_2+\tfrac{\delta}{2},k)\\
  y_\odot & = \alpha (1-k^2)\jac{nc}(\tfrac{\delta}{2},k)\jac{sd}(\xi_1,k)\jac{nc}(\xi_2+\tfrac{\delta}{2},k).
 \end{split}
\ela
These constitute the vertices of a discrete confocal coordinate system on the plane, that is,
there exist functions $f(m_1),g(m_1)$ and $\tilde{f}(m_2),\tilde{g}(m_2)$ such that
\bela{E6.47}
  \begin{pmatrix} x_\odot\\ y_\odot\end{pmatrix} = \frac{1}{\sqrt{a-b}}\begin{pmatrix}f(m_1)\tilde{f}(m_2)\\ g(m_1)\tilde{g}(m_2)\end{pmatrix}
\ela
and
\bela{E6.51}
 \begin{split}
  f(m_1)f(m_1+\tfrac{1}{2}) + g(m_1)g(m_1+\tfrac{1}{2}) & = a-b\\  
  \tilde{f}(m_2)\tilde{f}(m_2+\tfrac{1}{2}) - \tilde{g}(m_2)\tilde{g}(m_2+\tfrac{1}{2}) & = a-b,
 \end{split}
\ela
where $a - b = \alpha^2k^2$.
\end{theorem}

\begin{proof}
The structure of the solution \eqref{E6.45} of the linear system \eqref{E6.44} shows that it factorises according \eqref{E6.47}. If we choose the scaling of the functions $f,g$ and $\tilde{f},\tilde{g}$ in such a manner that
\bela{E6.48}
 \begin{split} 
  f & = |\alpha k| \sqrt{\jac{dc}(\tfrac{\delta}{2},k)}\jac{cd}(\xi_1,k)\\ \tilde{f}&=\alpha\sqrt{\jac{dc}(\tfrac{\delta}{2},k)}\jac{dc}(\xi_2+\tfrac{\delta}{2})\\
 g &= |\alpha k|\sqrt{1-k^2}\sqrt{\jac{nc}(\tfrac{\delta}{2},k)}\jac{sd}(\xi_1,k)\\ \tilde{g}&=\alpha\sqrt{1-k^2}\sqrt{\jac{nc}(\tfrac{\delta}{2},k)}\jac{nc}(\xi_2+\tfrac{\delta}{2})
  \end{split}
\ela
then it may be verified that the difference equations \eqref{E6.51} are indeed satisfied. Here, we have assumed that $\jac{cn}(\tfrac{\delta}{2},k)>0$, which is compatible with the continuum limit $\delta\rightarrow0$. The other case may be dealt with in a similar manner but requires the introduction of factors of the type $(-1)^{m_1}$ and $(-1)^{m_2}$ in the definitions of $f,g$ and $\tilde{f},\tilde{g}$ respectively. This corresponds to ``superdiscrete'' IC-nets which do not admit a continuum limit. Finally, since the pair \eqref{E6.47}, \eqref{E6.51} characterises discrete confocal coordinate systems on the plane \cite{BSST17}, the proof is complete.
\end{proof}

\begin{remark}
Up to the shift of the argument $\xi_2\rightarrow\xi_2 + \frac{\delta}{2}$ and the affine transformation
\bela{E6.54}
(x_\odot,y_\odot) \rightarrow (Ax_\odot,By_\odot),\qquad A = \jac{cd}(\tfrac{\delta}{2},k),\quad  B = \jac{cn}(\tfrac{\delta}{2},k).
\ela
the formulae \eqref{E6.39} and \eqref{E6.45} coincide. This confirms that the centres $(x_\odot,y_\odot)$ lie on affine transforms of the confocal conics associated with the ellipse \eqref{E6.1}. Equivalently, this implies that the affine transforms $(A^{-1}x_\times,B^{-1}y_\times)$ are likewise vertices of a discrete confocal coordinate system. In fact, the points $(x_\odot,y_\odot)$ and $(A^{-1}x_\times,B^{-1}y_\times)$ are part of the same (extended) discrete confocal coordinate system of ${(\frac{1}{2}\mathbb{Z})}^2$ combinatorics. Another implication of this connection is that the functions $f,g$ and $\tilde{f},\tilde{g}$ as given by \eqref{E6.48} satisfy the algebraic identities
\bela{E6.49}
 Af^2 + Bg^2 = a-b,\quad A\tilde{f}^2 - B\tilde{g}^2 = a-b.
\ela
The latter constitute the algebraic constraints on discrete confocal coordinate systems as defined by \eqref{E6.47}, \eqref{E6.51} which give rise to the privileged IC-net-related discrete confocal coordinate systems touched upon in the preceding. These have been discussed in detail in \cite{BSST17}.
\end{remark}

\begin{remark}
If we eliminate, for instance, $g$ between \eqref{E6.51}$_1$ and \eqref{E6.49}$_1$ then we obtain a first-order difference equation for $f$, namely
\bela{QRT1}
  (A^2-B^2)f_{\frac{1}{2}}^2f^2 + 2(a-b)B^2f_{\frac{1}{2}}f - (a-b)A(f_{\frac{1}{2}}^2+f^2) + (a-b)^2(1-B^2) = 0,
\ela
where $f = f(m_1),\,f_{\frac{1}{2}} = f(m_1+\frac{1}{2})$. The latter may be regarded as a first integral of a difference equation of second order. Indeed, if we regard $B^2$ as the associated constant of integration then elimination of $B$ leads to 
\bela{QRT2}
  f_1 = \frac{F^1(f_{\frac{1}{2}}) - f F^2(f_{\frac{1}{2}})}{F^2(f_{\frac{1}{2}}) - f F^3(f_{\frac{1}{2}})}
\ela
with
\bela{QRT3}
  F^1(f_{\frac{1}{2}}) = 2(a-b)f_{\frac{1}{2}},\quad F^2(f_{\frac{1}{2}}) = f_{\frac{1}{2}}^2 + (a-b)A,\quad F^3(f_{\frac{1}{2}}) = 2Af_{\frac{1}{2}}
\ela
and $f_1 = f(m_1+1)$. Remarkably, \eqref{QRT2}, \eqref{QRT3} constitutes a particular symmetric case of an 18-parameter family of integrable reversible mappings of the plane known as QRT maps \cite{QRT89}. These play a fundamental role in the theory of discrete integrable systems and are known to be parametrisable in terms of elliptic functions, which is in agreement with the parametrisation of IC-nets presented in the preceding.
\end{remark}


\section{Generalised checkerboard IC-nets}

The construction of checkerboard IC-nets in terms of the Blaschke cylinder model as described in Section \ref{s.laguerre} may immediately be generalised in a natural manner. Thus, for any given pencil of quadrics which contains the Blaschke cylinder $\cal Z$, we first select two (``horizontal'' and ``vertical'') sequences of hyperboloids $\mathcal{H}^h_n$ and $\mathcal{H}_n^v$ belonging to this pencil. We then choose two points $\ell_1$ and $m_1$ of the associated hypercycle base curve and iteratively construct two sequences of points $\ell_n$ and $m_n$ on the hypercycle base curve by ``moving along'' generators $L_n$ and $M_n$ of the corresponding hyperboloids $\mathcal{H}^h_n$ and $\mathcal{H}_n^v$ respectively, that is,
\bela{E7.1}
  L_n = (\ell_n,\ell_{n+1})\subset \mathcal{H}^h_n,\quad  M_n = (m_n,m_{n+1})\subset \mathcal{H}^v_n.
\ela
If, for any $i,k$, the two hyperboloids $\mathcal{H}^h_i$ and $\mathcal{H}_k^v$ coincide and the corresponding generators $L_i$ and $M_k$ have been chosen in such a manner that they are not in the same family of generators of the common hyperboloid then the lines $\ell_i,\ell_{i+1}$ and $m_k,m_{k+1}$ circumscribe an oriented circle. In particular, if $\mathcal{H}_n :=\mathcal{H}^h_n=\mathcal{H}_n^v$ and $\mathcal{H}_{n+2}=\mathcal{H}_n$ then standard checkerboard IC-nets are retrieved. An example of a generalised checkerboard IC-net in the case of period 4, that is, $\mathcal{H}_{n+4}=\mathcal{H}_n$ is displayed in Figure \ref{f.IC3} (left). Another example of period 4 which involves only three hyperboloids with $\mathcal{H}_{4n+1} = \mathcal{H}_1$, $\mathcal{H}_{2n+2} = \mathcal{H}_2$, $\mathcal{H}_{4n+3} = \mathcal{H}_3$ is also depicted in Figure \ref{f.IC3} (right).

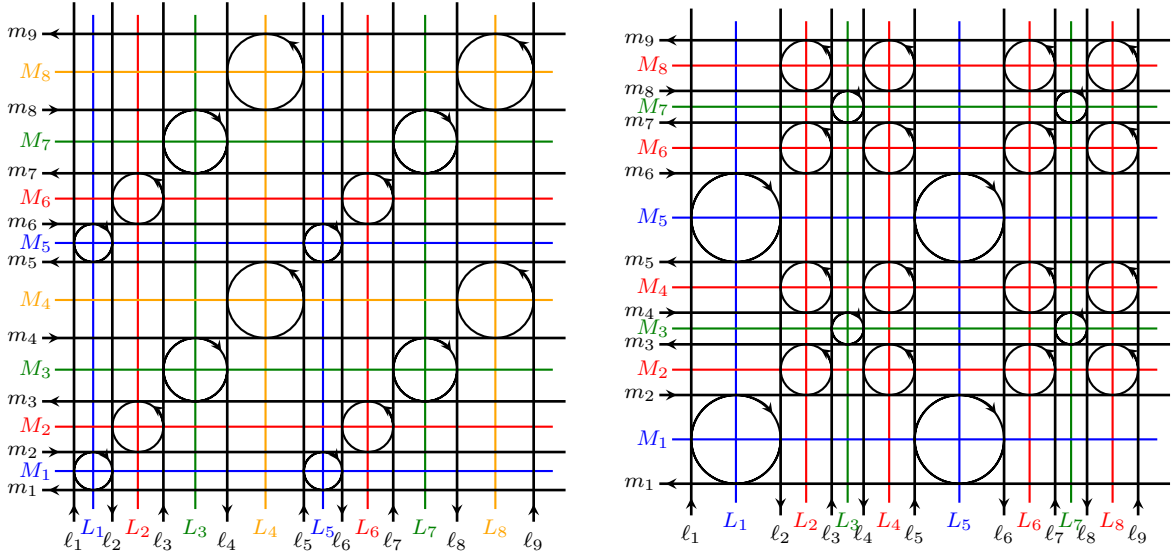
\begin{figure}
  \begin{center}
    \raisebox{-0.5\height}{
      \begin{tikzpicture}[line cap=line join=round,>=stealth,x=1.0cm,y=1.0cm, scale=0.84]
        \def\ra{0.6}%
        \def\rb{0.8}%
        \def\rc{1.0}%
        \def\rd{1.2}%
        \begin{scriptsize}
          \draw (0, -0.8) node {$\ell_1$};
          \draw (\ra, -0.8) node {$\ell_2$};
          \draw (\ra+\rb, -0.8) node {$\ell_3$};
          \draw (\ra+\rb+\rc, -0.8) node {$\ell_4$};
          \draw (\ra+\rb+\rc+\rd, -0.8) node {$\ell_5$};
          \draw (2*\ra+\rb+\rc+\rd, -0.8) node {$\ell_6$};
          \draw (2*\ra+2*\rb+\rc+\rd, -0.8) node {$\ell_7$};
          \draw (2*\ra+2*\rb+2*\rc+\rd, -0.8) node {$\ell_8$};
          \draw (2*\ra+2*\rb+2*\rc+2*\rd, -0.8) node {$\ell_9$};
          \draw (-0.8, 0) node {$m_1$};
          \draw (-0.8, \ra) node {$m_2$};
          \draw (-0.8, \ra+\rb) node {$m_3$};
          \draw (-0.8, \ra+\rb+\rc) node {$m_4$};
          \draw (-0.8, \ra+\rb+\rc+\rd) node {$m_5$};
          \draw (-0.8, 2*\ra+\rb+\rc+\rd) node {$m_6$};
          \draw (-0.8, 2*\ra+2*\rb+\rc+\rd) node {$m_7$};
          \draw (-0.8, 2*\ra+2*\rb+2*\rc+\rd) node {$m_8$};
          \draw (-0.8, 2*\ra+2*\rb+2*\rc+2*\rd) node {$m_9$};
          \draw [color=Blue]  (\ra/2, -0.6) node {$L_1$};
          \draw [color=Red]   (\ra+\rb/2, -0.6) node {$L_2$};
          \draw [color=Green] (\ra+\rb+\rc/2, -0.6) node {$L_3$};
          \draw [color=Orange]  (\ra+\rb+\rc+\rd/2, -0.6) node {$L_4$};
          \draw [color=Blue]   (\ra+\rb+\rc+\rd+\ra/2, -0.6) node {$L_5$};
          \draw [color=Red] (2*\ra+\rb+\rc+\rd+\rb/2, -0.6) node {$L_6$};
          \draw [color=Green] (2*\ra+2*\rb+\rc+\rd+\rc/2, -0.6) node {$L_7$};
          \draw [color=Orange] (2*\ra+2*\rb+2*\rc+\rd+\rd/2, -0.6) node {$L_8$};
          \draw [color=Blue]  (-0.6, \ra/2) node {$M_1$};
          \draw [color=Red]   (-0.6, \ra+\rb/2) node {$M_2$};
          \draw [color=Green] (-0.6, \ra+\rb+\rc/2) node {$M_3$};
          \draw [color=Orange]  (-0.6, \ra+\rb+\rc+\rd/2) node {$M_4$};
          \draw [color=Blue]   (-0.6, \ra+\rb+\rc+\rd+\ra/2) node {$M_5$};
          \draw [color=Red] (-0.6, 2*\ra+\rb+\rc+\rd+\rb/2) node {$M_6$};
          \draw [color=Green] (-0.6, 2*\ra+2*\rb+\rc+\rd+\rc/2) node {$M_7$};
          \draw [color=Orange] (-0.6, 2*\ra+2*\rb+2*\rc+\rd+\rd/2) node {$M_8$};
        \end{scriptsize}
        \clip(-0.5,-0.5) rectangle ({2*(\ra+\rb+\rc+\rd)+1.0}, {2*(\ra+\rb+\rc+\rd)+1.0});
        \draw [line width=0.8pt, color=Blue, domain=-0.3:2*(\ra+\rb+\rc+\rd)+0.3] plot(\ra/2,\x);
        \draw [line width=0.8pt, color=Red,  domain=-0.3:2*(\ra+\rb+\rc+\rd)+0.3] plot(\ra+\rb/2,\x);
        \draw [line width=0.8pt, color=Green, domain=-0.3:2*(\ra+\rb+\rc+\rd)+0.3] plot(\ra+\rb+\rc/2,\x);
        \draw [line width=0.8pt, color=Orange,  domain=-0.3:2*(\ra+\rb+\rc+\rd)+0.3] plot(\ra+\rb+\rc+\rd/2,\x);
        \draw [line width=0.8pt, color=Blue, domain=-0.3:2*(\ra+\rb+\rc+\rd)+0.3] plot(\ra+\rb+\rc+\rd+\ra/2,\x);
        \draw [line width=0.8pt, color=Red, domain=-0.3:2*(\ra+\rb+\rc+\rd)+0.3] plot(2*\ra+\rb+\rc+\rd+\rb/2,\x);
        \draw [line width=0.8pt, color=Green, domain=-0.3:2*(\ra+\rb+\rc+\rd)+0.3] plot(2*\ra+2*\rb+\rc+\rd+\rc/2,\x);
        \draw [line width=0.8pt, color=Orange, domain=-0.3:2*(\ra+\rb+\rc+\rd)+0.3] plot(2*\ra+2*\rb+2*\rc+\rd+\rd/2,\x);
        \draw [line width=0.8pt, color=Blue, domain=-0.3:2*(\ra+\rb+\rc+\rd)+0.3] plot(\x, \ra/2);
        \draw [line width=0.8pt, color=Red,  domain=-0.3:2*(\ra+\rb+\rc+\rd)+0.3] plot(\x, \ra+\rb/2);
        \draw [line width=0.8pt, color=Green, domain=-0.3:2*(\ra+\rb+\rc+\rd)+0.3] plot(\x, \ra+\rb+\rc/2);
        \draw [line width=0.8pt, color=Orange,  domain=-0.3:2*(\ra+\rb+\rc+\rd)+0.3] plot(\x, \ra+\rb+\rc+\rd/2);
        \draw [line width=0.8pt, color=Blue, domain=-0.3:2*(\ra+\rb+\rc+\rd)+0.3] plot(\x, \ra+\rb+\rc+\rd+\ra/2);
        \draw [line width=0.8pt, color=Red, domain=-0.3:2*(\ra+\rb+\rc+\rd)+0.3] plot(\x, 2*\ra+\rb+\rc+\rd+\rb/2);
        \draw [line width=0.8pt, color=Green, domain=-0.3:2*(\ra+\rb+\rc+\rd)+0.3] plot(\x, 2*\ra+2*\rb+\rc+\rd+\rc/2);
        \draw [line width=0.8pt, color=Orange, domain=-0.3:2*(\ra+\rb+\rc+\rd)+0.3] plot(\x, 2*\ra+2*\rb+2*\rc+\rd+\rd/2);

        \draw [line width=1pt, domain=-0.5:2*(\ra+\rb+\rc+\rd)+0.5] plot(0,\x);
        \draw [>-, thick] (0,-0.4) -- (0,-0.3);
        \draw [line width=1pt, domain=-0.5:2*(\ra+\rb+\rc+\rd)+0.5] plot(\ra,\x);
        \draw [<-, thick] (\ra,-0.4) -- (\ra,-0.3);
        \draw [line width=1pt, domain=-0.5:2*(\ra+\rb+\rc+\rd)+0.5] plot(\ra+\rb,\x);
        \draw [>-, thick] (\ra+\rb,-0.4) -- (\ra+\rb,-0.3);
        \draw [line width=1pt, domain=-0.5:2*(\ra+\rb+\rc+\rd)+0.5] plot(\ra+\rb+\rc,\x);
        \draw [<-, thick] (\ra+\rb+\rc,-0.4) -- (\ra+\rb+\rc,-0.3);
        \draw [line width=1pt, domain=-0.5:2*(\ra+\rb+\rc+\rd)+0.5] plot(\ra+\rb+\rc+\rd,\x);
        \draw [>-, thick] (\ra+\rb+\rc+\rd,-0.4) -- (\ra+\rb+\rc+\rd,-0.3);
        \draw [line width=1pt, domain=-0.5:2*(\ra+\rb+\rc+\rd)+0.5] plot(2*\ra+\rb+\rc+\rd,\x);
        \draw [<-, thick] (2*\ra+\rb+\rc+\rd,-0.4) -- (2*\ra+\rb+\rc+\rd,-0.3);
        \draw [line width=1pt, domain=-0.5:2*(\ra+\rb+\rc+\rd)+0.5] plot(2*\ra+2*\rb+\rc+\rd,\x);
        \draw [>-, thick] (2*\ra+2*\rb+\rc+\rd,-0.4) -- (2*\ra+2*\rb+\rc+\rd,-0.3);
        \draw [line width=1pt, domain=-0.5:2*(\ra+\rb+\rc+\rd)+0.5] plot(2*\ra+2*\rb+2*\rc+\rd,\x);
        \draw [<-, thick] (2*\ra+2*\rb+2*\rc+\rd,-0.4) -- (2*\ra+2*\rb+2*\rc+\rd,-0.3);
        \draw [line width=1pt, domain=-0.5:2*(\ra+\rb+\rc+\rd)+0.5] plot(2*\ra+2*\rb+2*\rc+2*\rd,\x);
        \draw [>-, thick] (2*\ra+2*\rb+2*\rc+2*\rd,-0.4) -- (2*\ra+2*\rb+2*\rc+2*\rd,-0.3);
        \draw [line width=1pt, domain=-0.5:2*(\ra+\rb+\rc+\rd)+0.5] plot(\x, 0);
        \draw [<-, thick] (-0.4,0) -- (-0.3,0);
        \draw [line width=1pt, domain=-0.5:2*(\ra+\rb+\rc+\rd)+0.5] plot(\x,\ra);
        \draw [>-, thick] (-0.4,\ra) -- (-0.3,\ra);
        \draw [line width=1pt, domain=-0.5:2*(\ra+\rb+\rc+\rd)+0.5] plot(\x, \ra+\rb);
        \draw [<-, thick] (-0.4,\ra+\rb) -- (-0.3,\ra+\rb);
        \draw [line width=1pt, domain=-0.5:2*(\ra+\rb+\rc+\rd)+0.5] plot(\x,\ra+\rb+\rc);
        \draw [>-, thick] (-0.4,\ra+\rb+\rc) -- (-0.3,\ra+\rb+\rc);
        \draw [line width=1pt, domain=-0.5:2*(\ra+\rb+\rc+\rd)+0.5] plot(\x, \ra+\rb+\rc+\rd);
        \draw [<-, thick] (-0.4,\ra+\rb+\rc+\rd) -- (-0.3,\ra+\rb+\rc+\rd);
        \draw [line width=1pt, domain=-0.5:2*(\ra+\rb+\rc+\rd)+0.5] plot(\x, 2*\ra+\rb+\rc+\rd);
        \draw [>-, thick] (-0.4,2*\ra+\rb+\rc+\rd) -- (-0.3,2*\ra+\rb+\rc+\rd);
        \draw [line width=1pt, domain=-0.5:2*(\ra+\rb+\rc+\rd)+0.5] plot(\x,2*\ra+2*\rb+\rc+\rd);
        \draw [<-, thick] (-0.4,2*\ra+2*\rb+\rc+\rd) -- (-0.3,2*\ra+2*\rb+\rc+\rd);
        \draw [line width=1pt, domain=-0.5:2*(\ra+\rb+\rc+\rd)+0.5] plot(\x,2*\ra+2*\rb+2*\rc+\rd);
        \draw [>-, thick] (-0.4,2*\ra+2*\rb+2*\rc+\rd) -- (-0.3,2*\ra+2*\rb+2*\rc+\rd);
        \draw [line width=1pt, domain=-0.5:2*(\ra+\rb+\rc+\rd)+0.5] plot(\x,2*\ra+2*\rb+2*\rc+2*\rd);
        \draw [<-, thick] (-0.4,2*\ra+2*\rb+2*\rc+2*\rd) -- (-0.3,2*\ra+2*\rb+2*\rc+2*\rd);
        \draw [->, thick] (\ra,\ra/2) arc (0:-685:\ra/2);
        \draw [->, thick] (2*\ra+\rb+\rc+\rd,\ra/2) arc (0:-685:\ra/2);
        \draw [->, thick] (\ra+\rb,\ra+\rb/2) arc (0:415:\rb/2);
        \draw [->, thick] (2*\ra+2*\rb+\rc+\rd,\ra+\rb/2) arc (0:415:\rb/2);
        \draw [->, thick] (\ra+\rb+\rc,\ra+\rb+\rc/2) arc (0:-685:\rc/2);
        \draw [->, thick] (2*\ra+2*\rb+2*\rc+\rd,\ra+\rb+\rc/2) arc (0:-685:\rc/2);
        \draw [->, thick] (\ra+\rb+\rc+\rd,\ra+\rb+\rc+\rd/2) arc (0:415:\rd/2);
        \draw [->, thick] (2*\ra+2*\rb+2*\rc+2*\rd,\ra+\rb+\rc+\rd/2) arc (0:415:\rd/2);
        \draw [->, thick] (\ra,\ra+\rb+\rc+\rd+\ra/2) arc (0:-685:\ra/2);
        \draw [->, thick] (2*\ra+\rb+\rc+\rd,\ra+\rb+\rc+\rd+\ra/2) arc (0:-685:\ra/2);
        \draw [->, thick] (\ra+\rb,2*\ra+\rb+\rc+\rd+\rb/2) arc (0:415:\rb/2);
        \draw [->, thick] (2*\ra+2*\rb+\rc+\rd,2*\ra+\rb+\rc+\rd+\rb/2) arc (0:415:\rb/2);
        \draw [->, thick] (\ra+\rb+\rc,2*\ra+2*\rb+\rc+\rd+\rc/2) arc (0:-685:\rc/2);
        \draw [->, thick] (2*\ra+2*\rb+2*\rc+\rd,2*\ra+2*\rb+\rc+\rd+\rc/2) arc (0:-685:\rc/2);
        \draw [->, thick] (\ra+\rb+\rc+\rd,2*\ra+2*\rb+2*\rc+\rd+\rd/2) arc (0:415:\rd/2);
        \draw [->, thick] (2*\ra+2*\rb+2*\rc+2*\rd,2*\ra+2*\rb+2*\rc+\rd+\rd/2) arc (0:415:\rd/2);
      \end{tikzpicture}}
     \raisebox{-0.5\height}{
      \begin{tikzpicture}[line cap=line join=round,>=stealth,x=1.0cm,y=1.0cm, scale=0.84]
        \def\ra{1.4}%
        \def\rb{0.8}%
        \def\rc{0.5}%
        \begin{scriptsize}
          \draw (0, -0.8) node {$\ell_1$};
          \draw (\ra, -0.8) node {$\ell_2$};
          \draw (\ra+\rb, -0.8) node {$\ell_3$};
          \draw (\ra+\rb+\rc, -0.8) node {$\ell_4$};
          \draw (\ra+\rb+\rc+\rb, -0.8) node {$\ell_5$};
          \draw (2*\ra+\rb+\rc+\rb, -0.8) node {$\ell_6$};
          \draw (2*\ra+2*\rb+\rc+\rb, -0.8) node {$\ell_7$};
          \draw (2*\ra+2*\rb+2*\rc+\rb, -0.8) node {$\ell_8$};
          \draw (2*\ra+2*\rb+2*\rc+2*\rb, -0.8) node {$\ell_9$};
          \draw (-0.8, 0) node {$m_1$};
          \draw (-0.8, \ra) node {$m_2$};
          \draw (-0.8, \ra+\rb) node {$m_3$};
          \draw (-0.8, \ra+\rb+\rc) node {$m_4$};
          \draw (-0.8, \ra+\rb+\rc+\rb) node {$m_5$};
          \draw (-0.8, 2*\ra+\rb+\rc+\rb) node {$m_6$};
          \draw (-0.8, 2*\ra+2*\rb+\rc+\rb) node {$m_7$};
          \draw (-0.8, 2*\ra+2*\rb+2*\rc+\rb) node {$m_8$};
          \draw (-0.8, 2*\ra+2*\rb+2*\rc+2*\rb) node {$m_9$};
          \draw [color=Blue]  (\ra/2, -0.6) node {$L_1$};
          \draw [color=Red]   (\ra+\rb/2, -0.6) node {$L_2$};
          \draw [color=Green] (\ra+\rb+\rc/2, -0.6) node {$L_3$};
          \draw [color=Red]  (\ra+\rb+\rc+\rb/2, -0.6) node {$L_4$};
          \draw [color=Blue]   (\ra+\rb+\rc+\rb+\ra/2, -0.6) node {$L_5$};
          \draw [color=Red] (2*\ra+\rb+\rc+\rb+\rb/2, -0.6) node {$L_6$};
          \draw [color=Green] (2*\ra+2*\rb+\rc+\rb+\rc/2, -0.6) node {$L_7$};
          \draw [color=Red] (2*\ra+2*\rb+2*\rc+\rb+\rb/2, -0.6) node {$L_8$};
          \draw [color=Blue]  (-0.6, \ra/2) node {$M_1$};
          \draw [color=Red]   (-0.6, \ra+\rb/2) node {$M_2$};
          \draw [color=Green] (-0.6, \ra+\rb+\rc/2) node {$M_3$};
          \draw [color=Red]  (-0.6, \ra+\rb+\rc+\rb/2) node {$M_4$};
          \draw [color=Blue]   (-0.6, \ra+\rb+\rc+\rb+\ra/2) node {$M_5$};
          \draw [color=Red] (-0.6, 2*\ra+\rb+\rc+\rb+\rb/2) node {$M_6$};
          \draw [color=Green] (-0.6, 2*\ra+2*\rb+\rc+\rb+\rc/2) node {$M_7$};
          \draw [color=Red] (-0.6, 2*\ra+2*\rb+2*\rc+\rb+\rb/2) node {$M_8$};
        \end{scriptsize}
        \clip(-0.5,-0.5) rectangle ({2*(\ra+\rb+\rc+\rb)+1.0}, {2*(\ra+\rb+\rc+\rb)+1.0});
        \draw [line width=0.8pt, color=Blue, domain=-0.3:2*(\ra+\rb+\rc+\rb)+0.3] plot(\ra/2,\x);
        \draw [line width=0.8pt, color=Red,  domain=-0.3:2*(\ra+\rb+\rc+\rb)+0.3] plot(\ra+\rb/2,\x);
        \draw [line width=0.8pt, color=Green, domain=-0.3:2*(\ra+\rb+\rc+\rb)+0.3] plot(\ra+\rb+\rc/2,\x);
        \draw [line width=0.8pt, color=Red,  domain=-0.3:2*(\ra+\rb+\rc+\rb)+0.3] plot(\ra+\rb+\rc+\rb/2,\x);
        \draw [line width=0.8pt, color=Blue, domain=-0.3:2*(\ra+\rb+\rc+\rb)+0.3] plot(\ra+\rb+\rc+\rb+\ra/2,\x);
        \draw [line width=0.8pt, color=Red, domain=-0.3:2*(\ra+\rb+\rc+\rb)+0.3] plot(2*\ra+\rb+\rc+\rb+\rb/2,\x);
        \draw [line width=0.8pt, color=Green, domain=-0.3:2*(\ra+\rb+\rc+\rb)+0.3] plot(2*\ra+2*\rb+\rc+\rb+\rc/2,\x);
        \draw [line width=0.8pt, color=Red, domain=-0.3:2*(\ra+\rb+\rc+\rb)+0.3] plot(2*\ra+2*\rb+2*\rc+\rb+\rb/2,\x);
        \draw [line width=0.8pt, color=Blue, domain=-0.3:2*(\ra+\rb+\rc+\rb)+0.3] plot(\x, \ra/2);
        \draw [line width=0.8pt, color=Red,  domain=-0.3:2*(\ra+\rb+\rc+\rb)+0.3] plot(\x, \ra+\rb/2);
        \draw [line width=0.8pt, color=Green, domain=-0.3:2*(\ra+\rb+\rc+\rb)+0.3] plot(\x, \ra+\rb+\rc/2);
        \draw [line width=0.8pt, color=Red,  domain=-0.3:2*(\ra+\rb+\rc+\rb)+0.3] plot(\x, \ra+\rb+\rc+\rb/2);
        \draw [line width=0.8pt, color=Blue, domain=-0.3:2*(\ra+\rb+\rc+\rb)+0.3] plot(\x, \ra+\rb+\rc+\rb+\ra/2);
        \draw [line width=0.8pt, color=Red, domain=-0.3:2*(\ra+\rb+\rc+\rb)+0.3] plot(\x, 2*\ra+\rb+\rc+\rb+\rb/2);
        \draw [line width=0.8pt, color=Green, domain=-0.3:2*(\ra+\rb+\rc+\rb)+0.3] plot(\x, 2*\ra+2*\rb+\rc+\rb+\rc/2);
        \draw [line width=0.8pt, color=Red, domain=-0.3:2*(\ra+\rb+\rc+\rb)+0.3] plot(\x, 2*\ra+2*\rb+2*\rc+\rb+\rb/2);

        \draw [line width=1pt, domain=-0.5:2*(\ra+\rb+\rc+\rb)+0.5] plot(0,\x);
        \draw [>-, thick] (0,-0.4) -- (0,-0.3);
        \draw [line width=1pt, domain=-0.5:2*(\ra+\rb+\rc+\rb)+0.5] plot(\ra,\x);
        \draw [<-, thick] (\ra,-0.4) -- (\ra,-0.3);
        \draw [line width=1pt, domain=-0.5:2*(\ra+\rb+\rc+\rb)+0.5] plot(\ra+\rb,\x);
        \draw [>-, thick] (\ra+\rb,-0.4) -- (\ra+\rb,-0.3);
        \draw [line width=1pt, domain=-0.5:2*(\ra+\rb+\rc+\rb)+0.5] plot(\ra+\rb+\rc,\x);
        \draw [<-, thick] (\ra+\rb+\rc,-0.4) -- (\ra+\rb+\rc,-0.3);
        \draw [line width=1pt, domain=-0.5:2*(\ra+\rb+\rc+\rb)+0.5] plot(\ra+\rb+\rc+\rb,\x);
        \draw [>-, thick] (\ra+\rb+\rc+\rb,-0.4) -- (\ra+\rb+\rc+\rb,-0.3);
        \draw [line width=1pt, domain=-0.5:2*(\ra+\rb+\rc+\rb)+0.5] plot(2*\ra+\rb+\rc+\rb,\x);
        \draw [<-, thick] (2*\ra+\rb+\rc+\rb,-0.4) -- (2*\ra+\rb+\rc+\rb,-0.3);
        \draw [line width=1pt, domain=-0.5:2*(\ra+\rb+\rc+\rb)+0.5] plot(2*\ra+2*\rb+\rc+\rb,\x);
        \draw [>-, thick] (2*\ra+2*\rb+\rc+\rb,-0.4) -- (2*\ra+2*\rb+\rc+\rb,-0.3);
        \draw [line width=1pt, domain=-0.5:2*(\ra+\rb+\rc+\rb)+0.5] plot(2*\ra+2*\rb+2*\rc+\rb,\x);
        \draw [<-, thick] (2*\ra+2*\rb+2*\rc+\rb,-0.4) -- (2*\ra+2*\rb+2*\rc+\rb,-0.3);
        \draw [line width=1pt, domain=-0.5:2*(\ra+\rb+\rc+\rb)+0.5] plot(2*\ra+2*\rb+2*\rc+2*\rb,\x);
        \draw [>-, thick] (2*\ra+2*\rb+2*\rc+2*\rb,-0.4) -- (2*\ra+2*\rb+2*\rc+2*\rb,-0.3);
        \draw [line width=1pt, domain=-0.5:2*(\ra+\rb+\rc+\rb)+0.5] plot(\x, 0);
        \draw [<-, thick] (-0.4,0) -- (-0.3,0);
        \draw [line width=1pt, domain=-0.5:2*(\ra+\rb+\rc+\rb)+0.5] plot(\x,\ra);
        \draw [>-, thick] (-0.4,\ra) -- (-0.3,\ra);
        \draw [line width=1pt, domain=-0.5:2*(\ra+\rb+\rc+\rb)+0.5] plot(\x, \ra+\rb);
        \draw [<-, thick] (-0.4,\ra+\rb) -- (-0.3,\ra+\rb);
        \draw [line width=1pt, domain=-0.5:2*(\ra+\rb+\rc+\rb)+0.5] plot(\x,\ra+\rb+\rc);
        \draw [>-, thick] (-0.4,\ra+\rb+\rc) -- (-0.3,\ra+\rb+\rc);
        \draw [line width=1pt, domain=-0.5:2*(\ra+\rb+\rc+\rb)+0.5] plot(\x, \ra+\rb+\rc+\rb);
        \draw [<-, thick] (-0.4,\ra+\rb+\rc+\rb) -- (-0.3,\ra+\rb+\rc+\rb);
        \draw [line width=1pt, domain=-0.5:2*(\ra+\rb+\rc+\rb)+0.5] plot(\x, 2*\ra+\rb+\rc+\rb);
        \draw [>-, thick] (-0.4,2*\ra+\rb+\rc+\rb) -- (-0.3,2*\ra+\rb+\rc+\rb);
        \draw [line width=1pt, domain=-0.5:2*(\ra+\rb+\rc+\rb)+0.5] plot(\x,2*\ra+2*\rb+\rc+\rb);
        \draw [<-, thick] (-0.4,2*\ra+2*\rb+\rc+\rb) -- (-0.3,2*\ra+2*\rb+\rc+\rb);
        \draw [line width=1pt, domain=-0.5:2*(\ra+\rb+\rc+\rb)+0.5] plot(\x,2*\ra+2*\rb+2*\rc+\rb);
        \draw [>-, thick] (-0.4,2*\ra+2*\rb+2*\rc+\rb) -- (-0.3,2*\ra+2*\rb+2*\rc+\rb);
        \draw [line width=1pt, domain=-0.5:2*(\ra+\rb+\rc+\rb)+0.5] plot(\x,2*\ra+2*\rb+2*\rc+2*\rb);
        \draw [<-, thick] (-0.4,2*\ra+2*\rb+2*\rc+2*\rb) -- (-0.3,2*\ra+2*\rb+2*\rc+2*\rb);
        \draw [->, thick] (\ra,\ra/2) arc (0:-685:\ra/2);
        \draw [->, thick] (2*\ra+\rb+\rc+\rb,\ra/2) arc (0:-685:\ra/2);
        \draw [->, thick] (\ra+\rb,\ra+\rb/2) arc (0:415:\rb/2);
        \draw [->, thick] (\ra+\rb+\rc+\rb,\ra+\rb/2) arc (0:415:\rb/2);
        \draw [->, thick] (2*\ra+2*\rb+\rc+\rb,\ra+\rb/2) arc (0:415:\rb/2);
        \draw [->, thick] (2*\ra+2*\rb+2*\rc+2*\rb,\ra+\rb/2) arc (0:415:\rb/2);
        \draw [->, thick] (\ra+\rb+\rc,\ra+\rb+\rc/2) arc (0:-685:\rc/2);
        \draw [->, thick] (2*\ra+2*\rb+2*\rc+\rb,\ra+\rb+\rc/2) arc (0:-685:\rc/2);
        \draw [->, thick] (\ra+\rb+\rc+\rb,\ra+\rb+\rc+\rb/2) arc (0:415:\rb/2);
        \draw [->, thick] (\ra+\rb,\ra+\rb+\rc+\rb/2) arc (0:415:\rb/2);
        \draw [->, thick] (2*\ra+2*\rb+2*\rc+2*\rb,\ra+\rb+\rc+\rb/2) arc (0:415:\rb/2);
        \draw [->, thick] (2*\ra+2*\rb+\rc+\rb,\ra+\rb+\rc+\rb/2) arc (0:415:\rb/2);
        \draw [->, thick] (\ra,\ra+\rb+\rc+\rb+\ra/2) arc (0:-685:\ra/2);
        \draw [->, thick] (2*\ra+\rb+\rc+\rb,\ra+\rb+\rc+\rb+\ra/2) arc (0:-685:\ra/2);
        \draw [->, thick] (\ra+\rb,2*\ra+\rb+\rc+\rb+\rb/2) arc (0:415:\rb/2);
        \draw [->, thick] (\ra+\rb+\rc+\rb,2*\ra+\rb+\rc+\rb+\rb/2) arc (0:415:\rb/2);
        \draw [->, thick] (2*\ra+2*\rb+\rc+\rb,2*\ra+\rb+\rc+\rb+\rb/2) arc (0:415:\rb/2);
        \draw [->, thick] (2*\ra+2*\rb+2*\rc+2*\rb,2*\ra+\rb+\rc+\rb+\rb/2) arc (0:415:\rb/2);
        \draw [->, thick] (\ra+\rb+\rc,2*\ra+2*\rb+\rc+\rb+\rc/2) arc (0:-685:\rc/2);
        \draw [->, thick] (2*\ra+2*\rb+2*\rc+\rb,2*\ra+2*\rb+\rc+\rb+\rc/2) arc (0:-685:\rc/2);
        \draw [->, thick] (\ra+\rb+\rc+\rb,2*\ra+2*\rb+2*\rc+\rb+\rb/2) arc (0:415:\rb/2);
        \draw [->, thick] (\ra+\rb,2*\ra+2*\rb+2*\rc+\rb+\rb/2) arc (0:415:\rb/2);
        \draw [->, thick] (2*\ra+2*\rb+2*\rc+2*\rb,2*\ra+2*\rb+2*\rc+\rb+\rb/2) arc (0:415:\rb/2);
        \draw [->, thick] (2*\ra+2*\rb+\rc+\rb,2*\ra+2*\rb+2*\rc+\rb+\rb/2) arc (0:415:\rb/2);
      \end{tikzpicture}}
  \end{center}
  \caption{Examples of generalised checkerboard IC nets involving three (right) and four (left)  different hyperboloids.}
  \label{f.IC3} 
\end{figure}

In algebraic terms, the construction of generalised checkerboard IC-nets may be implemented as follows. Here, we focus on a pencil of quadrics which has already been normalised so that
\bela{E7.2}
  (a + \lambda)v^2 + (b + \lambda)w^2 = d^2 + \lambda
\ela
with underlying normalised cone and Blaschke cylinder
\bela{E7.8}
  av^2 + bw^2 = d^2,\quad v^2 + w^2 = 1.
\ela
Any prescribed sequence of (suitably constrained) pencil parameters $\lambda_n$ then corresponds to a sequence of hyperboloids $\mathcal{H}_n$ which, in the following, represents one of  $\mathcal{H}^h_n$ or $\mathcal{H}_n^v$. The procedure described below may then also be applied to the other sequence of hyperboloids. Now, given a point $\bv_n=(v_n,w_n,d_n)$ (that is, either $\ell_n$ or $m_n$) on the hypercycle base curve, the two choices for the point $\bv_{n+1}$ corresponding to the pair of generators of the hyperboloid $\mathcal{H}_n$ passing through $\bv_n$ are obtained by intersecting the hypercycle base curve with the tangent plane of $\mathcal{H}_n$ at $\bv_n$. Algebraically, this is expressed by
\bela{E7.9}
  (a + \lambda)v_nv_{n+1} + (b + \lambda)w_nw_{n+1} = d_nd_{n+1} + \lambda.
\ela
If we eliminate $v_{n+1}$ and $w_{n+1}$ between this tangency condition and the hypercycle base curve constraints 
\bela{E7.10}
  av_{n+1}^2 + bw_{n+1}^2 = d_{n+1}^2,\quad v_{n+1}^2 + w_{n+1}^2 = 1
\ela
then we obtain a quartic in $d_{n+1}$ which, by construction, contains the factor $(d_{n+1}-d_n)^2$. Accordingly, we are left with a symmetric and biquadratic relation between $d_{n+1}$ and $d_n$ which reads
\bela{E7.11}
  \kappa_d(d_n^2d_{n+1}^2 + ab) + d_n^2 + d_{n+1}^2 + 2\kappa_v\kappa_wd_nd_{n+1} = 0,
\ela
where
\bela{E7.12}
  \kappa_v = \frac{\lambda^2 + 2a\lambda + ab}{\lambda^2 - ab},\quad
  \kappa_w = \frac{\lambda^2 + 2b\lambda + ab}{\lambda^2 - ab},\quad
  \kappa_d = 4\frac{\lambda(\lambda + a)(\lambda + b)}{(\lambda^2 - ab)^2}.
\ela
For a given point $\bv_n$, the two choices for the point $\bv_{n+1}$ therefore correspond to the two roots of the quadratic \eqref{E7.11}. In order to verify algebraically the uniqueness of the components $v_{n+1}$ and $w_{n+1}$ once $d_{n+1}$ has been fixed, it is convenient to be aware of the pair of linear equations (in $v_{n+1}$ and $w_{n+1}$) 
\bela{E7.13}
  \kappa_w a v_nv_{n+1} + \kappa_v b w_nw_{n+1} + d_nd_{n+1} = 0,\quad \kappa_v v_nv_{n+1} + \kappa_w w_nw_{n+1} = 1 
\ela
which may be extracted from the compatible system \eqref{E7.9}, \eqref{E7.10}. One may directly verify that the pair \eqref{E7.13} may be combined to reproduce the tangency condition \eqref{E7.9}. 

We observe in passing that for any given sequence $\lambda_n$, the biquadratic equation \eqref{E7.11} may be regarded as a non-autonomous extension of the first integral of a particular member of the symmetric class of QRT maps alluded to at the end of the previous section. In the case of standard confocal checkerboard IC-nets, the coefficients of the biquadratic  are of period 2. Non-autonomous QRT maps with periodic coefficients and their relation to discrete Painlev\'e equations have been discussed in detail in \cite{RGW11}. 

As an illustration of the above formalism, we consider the ``elliptic'' case analogous to that discussed in Section \ref{s.ellipticg0}, that is,
\bela{E7.14}
  a = \alpha^2,\quad b = \beta^2,\quad \lambda \geq 0.
\ela
Then, any given sequence $s_n$ determines both the hyperboloids $\mathcal{H}_n$ according to
\bela{E7.15}
  \lambda_n = \alpha^2\jac{cs^2}\left(\frac{s_n}{2},k\right),\quad k = \sqrt{1 - \frac{\beta^2}{\alpha^2}}
\ela
and, via $\operatorname{sgn}(s_n)$, to which families the generators ($L_n$ or $M_n$) belong. The solution of the system \eqref{E7.11}, \eqref{E7.13} is given by
\bela{E7.16}
  \bv_n = \begin{pmatrix} \jac{cn}(\psi_n,k)\\ \jac{sn}(\psi_n,k)\\ (-1)^n\alpha\jac{dn}(\psi_n,k)\end{pmatrix},\quad
  \psi_{n+1} = \psi_n + s_n
\ela
and the coefficients $\kappa_v$, $\kappa_w$ and $\kappa_d$ read
\bela{E7.17}
  \kappa_v = \jac{nc}(s_n,k),\quad \kappa_w = \jac{dc}(s_n,k),\quad \kappa_d = \alpha^{-2}\jac{sc}^2(s_n,k)
\ela
so that the relations \eqref{E7.13} are seen to encode the classical ``pencil'' of addition theorems \eqref{E6.9} for Jacobi elliptic functions. 

We conclude by translating the construction of ``periodic'' generalised checkerboard IC-nets in terms of the Blaschke model into a direct geometric construction in the plane with suitably prescribed Cauchy data. Here, we consider the case $\mathcal{H}^h_n=\mathcal{H}_n^v=\mathcal{H}_n$. We first observe that if we prescribe the line $m_1$ and the lines $\ell_n$ which are in oriented contact with a given hypercycle then a corresponding generalised checkerboard IC-net for which each quadruplet of lines $\ell_n,\ell_{n+1}$, $m_n,m_{n+1}$ is in oriented contact with a circle is uniquely determined. Indeed, there exists a unique line $m_2$ which is in oriented contact with the hypercycle and the unique circle in oriented contact with the lines $\ell_1,\ell_2$ and $m_1$. In this manner, all lines $m_n$ may be constructed iteratively. In the periodic case $\mathcal{H}_{n+N} = \mathcal{H}_n$, $N\geq2$ for which the circles inscribed in the quadruples of lines $\ell_n,\ell_{n+1}$, $m_{n + kN},m_{n+kN+1}$ are required to exist (cf.\ Figure \ref{f.IC3} (left) for $N=4$), it is sufficient to prescribe the lines $m_1$ and $\ell_1,\ldots,\ell_{N+1}$. In order to make good this assertion, we first construct the lines $m_2,\ldots,m_{N+1}$ in the manner described above. The triples of lines $m_1,m_2,\ell_{N+1}$ and $\ell_1,\ell_2,m_{N+1}$ then give rise to associated circles in oriented contact which, in turn, determine the lines $\ell_{N+2}$ and $m_{N+2}$ via oriented contact with the respective circle and the hypercycle. By virtue of Lemma \ref{l.quadric_in_pencil_through_line}, the existence of these two circles and the circle circumscribed by the lines $\ell_1,\ell_2, m_1, m_2$ now implies that the lines $L_1=(\ell_1,\ell_2),M_1=(m_1,m_2)$ and $L_{N+1}=(\ell_{N+1},\ell_{N+2}),M_{N+1}=(m_{N+1},m_{N+2})$ are generators of the same hyperboloid which one may denote by $\mathcal{H}_1=\mathcal{H}_{N+1}$ and, moreover, that $L_{N+1}$ and $M_{N+1}$ are not in the same family of generators of $\mathcal{H}_1$. This guarantees the existence of a circle which is in oriented contact with the quadruple of lines $\ell_{N+1},\ell_{N+2},m_{N+1},m_{N+2}$. Iterative application of this procedure now generates the entire generalised checkerboard IC-net of period $N$. Finally, we merely mention that generalised checkerboard IC-nets of the type displayed in Figure \ref{f.IC3} (right) are determined by lines $m_1$ and $\ell_1,\ell_2,\ell_3,\ell_4$ which are in oriented contact with a hypercycle.

\paragraph{Acknowledgement}

This research was supported by the DFG Collaborative Research Center TRR 109 ``Discretization in Geometry and Dynamics''. W.K.S. was also supported by the Australian Research Council (DP1401000851).

\begin{appendix}
\section{Laguerre geometry}
\label{s.Laguerre}
Here, we present the basic facts about Laguerre geometry, focussing on the Blaschke cylinder model employed in this paper for studying checkerboard IC-nets. We begin with the more fundamental Lie sphere geometry. Lie sphere geometry in the plane is the geometry of oriented circles and lines. These are described as elements of the Lie quadric 
$$
{\cal L}=P(\L^{3,2}), \quad \L^{3,2}=\{x\in \R^{3,2}| <x,x>_{\R^{3,2}}=0\}.
$$
 Let $e_1,e_2,e_3,e_4,e_5$ be an orthonormal basis with signature $(+++--)$. For our purposes, another basis $e_1, e_2, e_5, e_\infty, e_0$ defined by
 $$
 e_0=\frac{1}{2}(e_4-e_3),\quad e_\infty=\frac{1}{2}(e_4+e_3),\quad <e_0,e_\infty>=-\frac{1}{2} 
 $$
turns out to be more convenient. Elements of $\cal L$ with non-vanishing $e_0$-component are identified with oriented circles $|\bx-\bc|^2=r^2$, centred at $\bc\in \R^2$ and of radius $r\in \R$:
\begin{equation}
\label{eq.oriented_circle}
s=\bc+re_5+(|\bc|^2-r^2)e_\infty+e_0.
\end{equation}
Points are circles of radius $r=0$, and oriented lines $(\bv,\bx)_{\R^2}=d$ are elements of $\cal L$ with vanishing $e_0$-component:
\begin{equation}
\label{eq.oriented_line}
p=\bv+e_5+2de_\infty.
\end{equation}
The incidence $<p,s>=0$ is the condition 
\begin{equation}
\label{eq.oriented_contact}
(\bc,\bv)-r=d
\end{equation} 
of oriented contact of a circle and a line.

The Lie sphere transformation group $PO(3,2)$ acting on $P(\R^{3,2})$ preserves the Lie quadric $\cal L$ and maps oriented circles and lines to oriented circles and lines, preserving oriented contact. Its subgroup of Laguerre transformations preserves the set of straight lines or, equivalently, the hyperplane
$$
{\mathsf P}={\rm span}\{e_1,e_2,e_5,e_\infty \}=\{w\in \R^{3,2}|<w,e_\infty>=0\}.
$$
Direct computation shows that the elements of $PO(3,2)$ preserving the hyperplane $\mathsf P$ are of the form 
\bela{e.Laguerre_trafo_5x5}
  \left(\begin{matrix} \lambda B & 0 & \alpha \\[1mm]
                         b^T & 1 & \nu\\[1mm]
                         0 & 0 & \lambda^{-2}
  \end{matrix}\right)
\ela
in the basis $e_1,e_2,e_5,e_\infty,e_0$,
where 
$$
B\in O(2,1),\quad b\in \R^{2,1},\quad \alpha=\frac{\lambda}{2}Bb,\quad \nu=\frac{\lambda}{4}(b,b)_{\R^{2,1}},\quad \lambda\in \R.
$$

In order to pass to the {\em Blaschke cylinder model} of Laguerre geometry, we confine ourselves to the subspace ${\cal P}={\rm span}\{e_1,e_2,e_5,e_\infty \}$. Elements of this space can be identified with straight lines, described (projectively) by 
$$
\tilde{p}=\tilde{\bv}+\tilde{s}e_5+2\tilde{d}e_\infty
$$
as points of the Blaschke cylinder 
\begin{equation}
\label{eq.Blaschke_cylinder_projective}
{\cal Z}=\{ [\tilde{p}]\in P(\R^{2,1,1})| |\tilde{\bv}|^2=\tilde{s}^2\}.
\end{equation}
Identification with \eqref{eq.oriented_line} is made via the normalisation of the $e_0$-component: $\bv=\tilde{\bv}/\tilde{s}$, $d=\tilde{d}/\tilde{s}$. It is noted that the symmetry with the description of circles \eqref{eq.oriented_circle} in Lie sphere geometry is no longer present in the Blaschke cylinder model, and oriented circles are described as the sets of all straight lines in oriented contact, i.e., the sets of lines satisfying the condition \eqref{eq.oriented_contact}, that is
$$
S= \{(\bv,d)\in\R^3| (\bc,\bv)_{\R^2}-r=d \}.
$$
Furthermore, Laguerre transformations restricted to the subspace of lines ${\mathsf P}={\rm span}\{e_1,e_2,e_5,e_\infty \}$ are of the form
\bela{eq.Laguerre_trafo_4x4}
 A= \left(\begin{matrix} \lambda B & 0  \\[1mm]
                         b^T & 1 
  \end{matrix}\right).
\ela
\end{appendix}

\begin{theorem}
The group of Laguerre transformations in the Blaschke (projective) cylinder model (in the basis $e_1,e_2,e_5,e_\infty$) is represented by matrices of the form \eqref{eq.Laguerre_trafo_4x4}, where 
$B\in O(2,1),\, b\in \R^{2,1},\, \lambda\in \R$. These transformations preserve the Blaschke cylinder \eqref{eq.Blaschke_cylinder_projective}.
\end{theorem}

We conclude by observing that Euclidean motions
$$
\bx\to\tilde{\bx}=R\bx+\Delta, \quad R\in O(2),\quad \Delta\in\R^2
$$
are particular Laguerre transformations, the corresponding matrix of which is given by \eqref{eq.Laguerre_trafo_4x4} with 
\begin{equation}
\label{eq.Euclidean_Laguerre}
B=\left(\begin{matrix} R^T & 0  \\[1mm]
                       0 & 1 
  \end{matrix}\right),\quad
b=\left(\begin{matrix} 2R\Delta   \\[1mm]
                       0  
  \end{matrix}\right)
\end{equation}
and $\lambda=1$. 


\end{document}